\documentclass[a4paper,10pt]{amsart}

\usepackage[T1]{fontenc}
\usepackage{ae,aecompl}
\usepackage{amsmath,amsthm,amssymb,amsfonts}
\usepackage[mathscr]{eucal}
\usepackage{graphicx}
\usepackage{ifthen}
\usepackage[perpage,symbol*]{footmisc}
\DefineFNsymbols{cfoot}{\textdagger \textdaggerdbl {\textdagger\textdagger} {\textdaggerdbl\textdaggerdbl}} 
\setfnsymbol{cfoot}

\newcommand{\vphi}{\ensuremath{\varphi}}
\newcommand{\al}{\ensuremath{\alpha}}
\newcommand{\be}{\ensuremath{\beta}}
\newcommand{\ga}{\ensuremath{\gamma}}
\newcommand{\Ga}{\ensuremath{\Gamma}}
\newcommand{\de}{\ensuremath{\delta}}
\newcommand{\De}{\ensuremath{\Delta}}
\newcommand{\eps}{\ensuremath{\varepsilon}}
\newcommand{\ze}{\ensuremath{\zeta}}
\newcommand{\ka}{\ensuremath{\kappa}}
\newcommand{\la}{\ensuremath{\lambda}}
\newcommand{\La}{\ensuremath{\Lambda}}
\newcommand{\sig}{\ensuremath{\sigma}}
\newcommand{\Sig}{\ensuremath{\Sigma}}
\newcommand{\om}{\ensuremath{\omega}}
\newcommand{\Om}{\ensuremath{\Omega}}

\newcommand{\N}{\ensuremath{\mathbf{N}}}
\newcommand{\R}{\ensuremath{\mathbf{R}}}
\newcommand{\Z}{\ensuremath{\mathbf{Z}}}
\newcommand{\Hh}{\ensuremath{\mathbf{H}}}
\newcommand{\Ss}{\ensuremath{\mathbf{S}}}
\newcommand{\D}{\ensuremath{\mathbf{D}}}
\newcommand{\RP}{\ensuremath{\mathbf{RP}}}
\newcommand{\SO}{\ensuremath{\mathbf{SO}}}

\newcommand{\tdef}{\textit}
\newcommand{\tbf}{\textbf}
\newcommand{\tup}{\textup}
\newcommand{\tit}{\textit}
\newcommand{\tsc}{\textsc}
\newcommand{\ttt}{\texttt}
\newcommand*{\mbf}[1]{\boldsymbol{#1}}
\newcommand{\fr}{\mathfrak}

\newcommand*{\sr}[1]{\ensuremath{\mathscr{#1}}}

\newcommand{\?}{\ensuremath{\mspace{2mu}}}
\newcommand{\sz}{\ensuremath{\mspace{2mu}}}

\newcommand{\ol}{\overline}

\newcommand{\te}{\tilde}
\newcommand{\ce}{\check}

\newcommand*{\abs}[1]{\left\vert#1\right\vert}
\newcommand*{\norm}[1]{\left\Vert#1\right\Vert}
\newcommand*{\fl}[1]{\left\lfloor#1\right\rfloor}
\newcommand{\gen}[1]{\left\langle#1\right\rangle}
\newcommand{\ggen}[1]{\pmb\langle#1\pmb\rangle}
\newcommand{\se}[1]{\left\{#1\right\}}

\newcommand{\set}[2]{\big\{#1 \ \text{\large $:$}\ #2 \big\}}

\DeclareMathOperator{\Int}{Int}
\DeclareMathOperator{\arccot}{arccot}
\newcommand{\ring}{\mathring}
\DeclareMathOperator{\imag}{Im}
\renewcommand{\Im}{\imag}

\DeclareMathOperator{\id}{id}
\DeclareMathOperator{\pr}{pr}

\DeclareMathOperator{\Free}{Free}

\newcommand{\bd}{\partial}
\newcommand{\nin}{\notin}
\newcommand{\iso}{\simeq}
\newcommand{\home}{\approx}

\newcommand{\dar}{\ensuremath{\leftrightarrow}}
\newcommand{\incr}{\text{\small$\mspace{1mu}\nearrow\mspace{1mu}$}}
\newcommand{\decr}{\text{\small$\mspace{1mu}\searrow\mspace{2mu}$}}

\newcommand{\Rar}{\ensuremath{\Rightarrow}}

\newcommand{\subs}{\subset}
\newcommand{\nsubs}{\not\subset}
\newcommand{\sups}{\supset}

\newcommand{\ssm}{\smallsetminus}

\newcommand{\bcup}{\bigcup}
\newcommand{\bcap}{\bigcap}
\newcommand{\du}{\sqcup}
\newcommand{\Du}{\bigsqcup}

\newcommand{\inc}{\hookrightarrow}

\newcommand{\tri}{\triangle}

\newcommand{\san}{\sphericalangle}

\newcounter{cthm}
\newtheoremstyle{cmain}
	{}
	{}
	{\itshape}
	{}
	{\bfseries}
	{.}
	{ }
	{(\thmnumber{#2})\:\addtocounter{cthm}{1}\thmname{#1} \Alph{cthm}\thmnote{ (#3)}}
\newtheoremstyle{cplain}
	{}
	{}
	{\itshape}
	{}
	{\bfseries}
	{.}
	{ }
	{(\thmnumber{#2})\:\thmname{#1}\thmnote{ (#3)}}
\newtheoremstyle{cdefinition}
	{}
	{}
	{}
	{}
	{\bfseries}
	{.}
	{ }
	{(\thmnumber{#2})\:\thmname{#1}\thmnote{ (#3)}}
\newtheoremstyle{cremark}
	{}
	{}
	{}
	{}
	{}
	{.}
	{ }
	{{\bfseries(\thmnumber{#2})}\:{\itshape \thmname{#1}}\thmnote{ (#3)}}
\newtheoremstyle{uremark}
	{}
	{}
	{}
	{}
	{}
	{.}
	{ }
	{{\itshape \thmname{#1}}\thmnote{ (#3)}}

\makeatletter
\@addtoreset{equation}{section} 
\makeatother

\theoremstyle{cplain}
\newtheorem{thm}{Theorem}[section]
	
	\newtheorem{prop}[thm]{Proposition}
	\newtheorem{cor}[thm]{Corollary}
	\newtheorem{lemma}[thm]{Lemma}
{
	\theoremstyle{cremark}
	
}
{
\theoremstyle{cremark}
\newtheorem{rem}[thm]{Remark}
\newtheorem{rems}[thm]{Remarks}
}
{
\theoremstyle{uremark}
\newtheorem*{exm}{Example}
\newtheorem*{exms}{Examples}
\newtheorem*{urem}{Remark}

}
{
\theoremstyle{cdefinition}
\newtheorem{defn}[thm]{Definition}
\newtheorem{defns}[thm]{Definitions}

}
{
\theoremstyle{cmain}
\newtheorem{mthm}[thm]{Theorem}
}

\renewcommand{\L}{\sr{L}}
\newcommand{\E}{\mathbf{E}}

\newcommand{\co}[1]{\hat{#1}}
\newcommand{\no}{\mathbf{n}}
\newcommand{\ta}{\mathbf{t}}
\DeclareMathOperator{\tot}{tot}
\newcommand{\gr}{\preccurlyeq}

\begin{document}

\title[On the Components of Spaces of Curves on the 2-Sphere]{On the Components of Spaces of Curves on the 2-Sphere \\ with Geodesic Curvature in a Prescribed Interval}
\author{Nicolau C.~Saldanha}
\author{Pedro Z\"{u}hlke}
\maketitle

\begin{abstract}Let $\sr C_{\ka_1}^{\ka_2}$ denote the set of all closed curves of class $C^r$ on the sphere $\Ss^2$ whose geodesic curvatures are constrained to lie in $(\ka_1,\ka_2)$, furnished with the $C^r$ topology (for some $r\geq 2$ and possibly infinite $\ka_1<\ka_2$). In 1970, J.~Little proved that the space $\sr C_0^{+\infty}$ of closed curves having positive geodesic curvature has three connected components. Let $\rho_i=\arccot \ka_i$ ($i=1,2$). We show that $\sr C_{\ka_1}^{\ka_2}$ has $n$ connected components $\sr C_1,\dots,\sr C_n$, where
\begin{equation*}
	n=\left\lfloor{\frac{\pi}{\rho_1-\rho_2}}\right\rfloor+1
\end{equation*} 
and $\sr C_j$ contains circles traversed $j$ times ($1\leq j\leq n$). The component $\sr C_{n-1}$ also contains circles traversed $(n-1)+2k$ times, and $\sr C_n$ also contains circles traversed $n+2k$ times, for any $k\in \N$. Further, each of $\sr C_1,\dots,\sr C_{n-2}$ ($n\geq 3$) is homeomorphic to $\SO_3\times \E$, where $\E$ is the separable Hilbert space. We also obtain a simple characterization of the components in terms of the properties of a curve and prove that $\sr C_{\ka_1}^{\ka_2}$ is homeomorphic to $\sr C_{\bar\ka_1}^{\bar\ka_2}$ whenever $\rho_1-\rho_2=\bar\rho_1-\bar\rho_2$ ($\bar\rho_i=\arccot \bar\ka_i$). 
\end{abstract}

\setcounter{tocdepth}{1}
\tableofcontents


\setcounter{section}{-1}\section{Introduction}\label{S:introduction}

\subsection*{History of the problem} Consider the set $\sr W$ of all $C^r$ regular closed curves in the plane $\R^2$ (i.e., $C^r$ immersions $\Ss^1\to \R^2$), furnished with the $C^r$ topology ($r\geq 1$). The Whitney-Graustein theorem (\cite{WhiGra}, thm.~1) states that two such curves are homotopic through regular closed curves if and only if they have the same rotation number (where the latter is the number of full turns of the tangent vector to the curve). Thus, the space $\sr W$ has an infinite number of connected components $\sr W_n$, one for each rotation number $n\in \Z$. A typical element of $\sr W_n$ ($n\neq 0$) is a circle traversed $\abs{n}$ times, with the direction depending on the sign of $n$; $\sr W_0$ contains a figure eight curve.

For curves on the unit sphere $\Ss^2\subs \R^3$, there is no natural notion of rotation number. Indeed, the corresponding space $\sr I$ of $C^r$ immersions $\Ss^1\to \Ss^2$ (i.e., regular closed curves on $\Ss^2$) has only two connected components, $\sr I_+$ and $\sr I_-$; this is an immediate consequence of a much more general result of S.~Smale (\cite{Sma56}, thm.~A). The component $\sr I_-$ contains all circles traversed an odd number of times, and the component $\sr I_+$ contains all circles traversed an even number of times. Actually, Smale's theorem implies that $\sr I_{\pm}\iso \SO_3\times \Om\Ss^3$,  where $\Om\Ss^3$ denotes the set of all continuous closed curves on $\Ss^3$, with the compact-open topology; the properties of the latter space are well understood (see \cite{BottTu}, \S 16).\footnote{The notation $X\iso Y$ (resp.~$X\home Y$) means that $X$ is homotopy equivalent (resp.~homeomorphic) to $Y$.}

 In 1970, J.~A.~Little formulated and solved the following problem: Let $\sr C$ denote the set of all $C^r$ closed curves on $\Ss^2$ which have nonvanishing geodesic curvature, with the $C^r$ topology $(r\geq 2)$; what are the connected components of $\sr C$?  Although his motivation to investigate $\sr C$ appears to have been  purely geometric, this space arises naturally in the study of a certain class of linear ordinary differential equations (see \cite{Sal4} for a discussion of this class and further references). In another notation, $\sr C$ is the space $\Free(\Ss^1,\Ss^2)$ of free closed spherical curves. A map $f:M \to N$ is called (second-order) free if the osculating space of order 2 is nondegenerate; for $M = \Ss^1$ and $N = \Ss^2$, this is equivalent to saying that the curve $f$ has nonvanishing geodesic curvature (cf.~\cite{EliashbergMishachev}, \cite{GromovPR}).

Little was able to show (see \cite{Little}, thm.~1) that $\sr C$ has six connected components, $\sr C_{\pm 1}$, $\sr C_{\pm 2}$ and $\sr C_{\pm 3}$, where the sign indicates the sign of the geodesic curvature of a curve in the corresponding component. A homeomorphism between $\sr C_i$ and $\sr C_{-i}$ is obtained by reversing the orientation of the curves in $\sr C_i$.
 \begin{figure}[ht]
	\begin{center}
		\includegraphics[scale=.30]{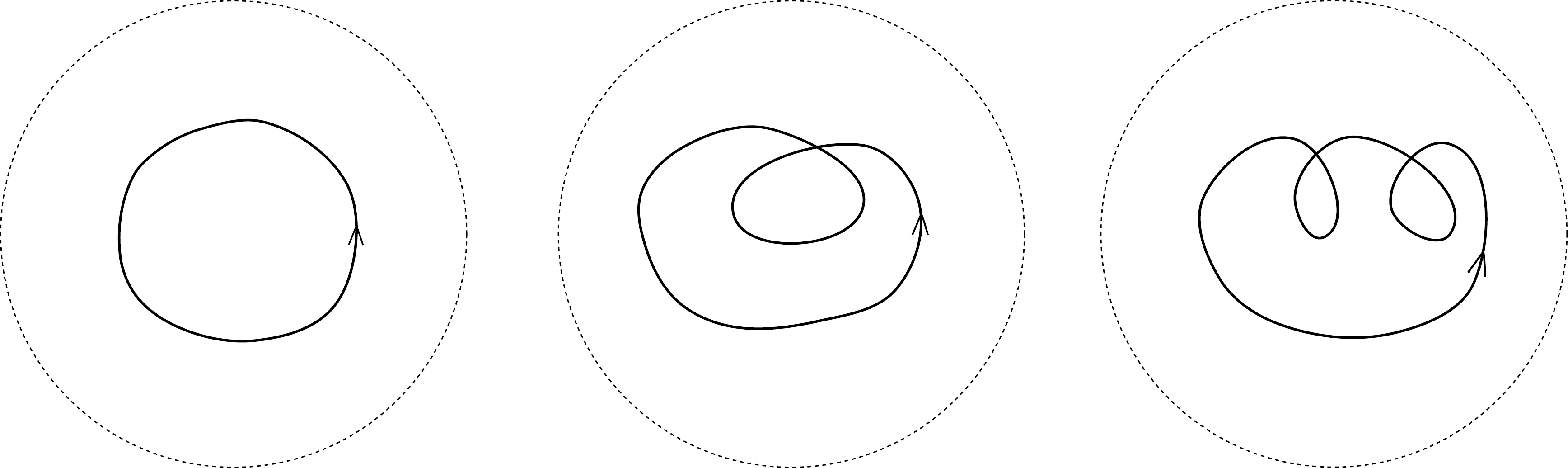}
		\caption{The curves depicted above provide representatives of the components $\sr C_{1}$, $\sr C_{2}$ and $\sr C_{3}$, respectively. All three are contained in the upper hemisphere of $\Ss^2$; the dashed line represents the equator seen from above.}
		\label{F:Little}
	\end{center}
\end{figure}

The topology of the space $\sr C$ has been investigated by quite a few other people since Little. We mention here only B.~Khesin, B.~Shapiro and M.~Shapiro, who studied $\sr C$ and similar spaces in the 1990's (cf.~\cite{KheSha}, \cite{KheSha2}, \cite{ShaSha} and \cite{ShaTop}). They showed that $\sr C_{\pm1}$  are homotopy equivalent to $\SO_3$, and also determined the number of connected components of the spaces analogous to $\sr C$ in $\R^n$, $\Ss^n$ and $\RP^n$, for arbitrary $n$.

The first pieces of information about the homotopy and cohomology groups $\pi_k(\sr C)$ and $H^k(\sr C)$ for $k\geq 1$ were obtained a decade later by the first author in \cite{Sal1} and \cite{Sal2}. Finally, in the recent work \cite{Sal3}, a description of the homotopy type of spaces of locally convex (but not necessarily closed) curves on $\Ss^2$ is presented. It is proved in particular that
\begin{alignat*}{9}
	\sr C_{\pm 2}&\iso \SO_3\times \big(\Om\Ss^3\vee \Ss^2\vee \Ss^6\vee \Ss^{10}\vee\dots\big) \text{\quad and \quad}
	\\\sr C_{\pm 3} &\iso \SO_3\times \big(\Om\Ss^3\vee \Ss^4\vee \Ss^8\vee \Ss^{12}\vee\dots\big).
\end{alignat*}

The reason for the appearance of an $\SO_3$ factor in all of these results is that (unlike in \cite{Sal3}) we have not chosen a basepoint for the unit tangent bundle $UT\Ss^2\home \SO_3$; a careful discussion of this is given in \S 1. 

\subsection*{Brief overview of this work} The main purpose of this work is to generalize Little's theorem to other spaces of closed curves on $\Ss^2$. Let $-\infty\leq\ka_1<\ka_2\leq+\infty$ be given and let $\sr C_{\ka_1}^{\ka_2}$ be the set of all $C^r$ closed curves on $\Ss^2$ whose geodesic curvatures are constrained to lie in the interval $(\ka_1,\ka_2)$, furnished with the $C^r$ topology (for some $r\geq 2$). In this notation, the spaces $\sr C$ and $\sr I$ discussed above become $\sr C_{-\infty}^0 \du \sr C_0^{+\infty}$ and $\sr C_{-\infty}^{+\infty}$, respectively. Naturally, a space of curves on $\Ss^2$ whose curvatures are constrained to lie in an open subset of $\R$ is the disjoint union of spaces of this type.

We present a direct characterization of the connected components of $\sr C_{\ka_1}^{\ka_2}$ in terms of the pair $\ka_1<\ka_2$ and of the properties of curves in $\sr C_{\ka_1}^{\ka_2}$. This yields a simple procedure to decide whether two given curves in $\sr C_{\ka_1}^{\ka_2}$ lie in the same component (that is, whether they are homotopic through closed curves whose geodesic curvatures are constrained to $(\ka_1,\ka_2)$). Another consequence is that the number of components of $\sr C_{\ka_1}^{\ka_2}$ is always finite, and given by a simple formula involving $\ka_1$, $\ka_2$.

More precisely, let $\rho_i=\arccot(\ka_i)$, $i=1,2$, where we adopt the convention that $\arccot$ takes values in $[0,\pi]$, with $\arccot(+\infty)=0$ and $\arccot(-\infty)=\pi$. Also,  let $\lfloor x\rfloor $ denote the greatest integer smaller than or equal to $x$. Then $\sr C_{\ka_1}^{\ka_2}$ has $n$ connected components $\sr C_1,\dots,\sr C_n$, where
\begin{equation}\label{E:number}
	n=\left\lfloor{\frac{\pi}{\rho_1-\rho_2}}\right\rfloor+1
\end{equation} 
and  $\sr C_j$ contains circles traversed $j$ times ($1\leq j\leq n$). The component $\sr C_{n-1}$ also contains circles traversed $(n-1)+2k$ times, and $\sr C_n$ contains circles traversed $n+2k$ times, for $k\in \N$. Moreover, for $n\geq 3$, each of $\sr C_1,\dots,\sr C_{n-2}$ is homeomorphic to $\SO_3\times \E$, where $\E$ is the separable Hilbert space.

This result could be considered a first step towards the determination of the homotopy type of $\sr C_{\ka_1}^{\ka_2}$ in terms of $\ka_1$ and $\ka_2$. In this context, it is natural to ask whether the inclusion $\sr C_{\ka_1}^{\ka_2}\inc \sr C_{-\infty}^{+\infty}=\sr I$ is a homotopy equivalence; as we have already mentioned, the topology of the latter space is well understood. It is shown in \S 10 of \cite{tese} that the answer is negative when $\rho_1-\rho_2\leq \frac{2\pi}{3}$ (note that $\rho_1-\rho_2\in (0,\pi]$), and we expect it to be negative except when $\sr C_{\ka_1}^{\ka_2}=\sr C_{-\infty}^{+\infty}$ (i.e., when $\rho_1-\rho_2=\pi$).  

Actually, we conjecture that $\sr C_{\ka_1}^{\ka_2}$ and $\sr C_{\bar\ka_1}^{\bar\ka_2}$ have different homotopy types if and only if $\rho_1-\rho_2\neq \bar\rho_1-\bar\rho_2$, but here it will only be proved that $\sr C_{\ka_1}^{\ka_2}$ is homeomorphic to $\sr C_{\bar\ka_1}^{\bar\ka_2}$ if $\rho_1-\rho_2=\bar\rho_1-\bar\rho_2$ (where $\rho_i=\arccot \ka_i$ and $\bar\rho_i=\arccot \bar\ka_i$). Our guess is that the homotopy type of the ``large'' components $\sr C_{n-1}$ and $\sr C_{n}$ of $\sr C_{\ka_1}^{\ka_2}$ (with $n$ as in \eqref{E:number}) is that of a space of the form
\[
\SO_3\times \big(\Om\Ss^3\vee \Ss^{2n_1}\vee \Ss^{2n_2}\vee \Ss^{2n_3}\vee\dots\big), 
\]
$n_1 \le n_2 \le n_3 \le \cdots$ being positive integers which can be obtained in terms of $\ka_1$ and $\ka_2$ by formulas similar to \eqref{E:number}.

\subsection*{Outline of the sections} It turns out that it is more convenient, but not essential, to work with curves which need not be $C^2$. The curves that we consider possess continuously varying unit tangent vectors at all points, but their geodesic curvatures are defined only almost everywhere. This class of curves is described in \S 1, where we also relate the resulting spaces $ \sr L_{\ka_1}^{\ka_2}$ to the more familiar spaces $\sr C_{\ka_1}^{\ka_2}$ of $C^r$ curves: The inclusion $\sr C_{\ka_1}^{\ka_2}\inc \sr L_{\ka_1}^{\ka_2}$ is a homotopy equivalence and has dense image. In this section we take the first steps toward the main theorem by proving that the topology of $\sr L_{\ka_1}^{\ka_2}$ depends only on $\rho_1-\rho_2$. A corollary of this result is that any space $\sr L_{\ka_1}^{\ka_2}$ is homeomorphic to a space of type $\sr L_{\ka_0}^{+\infty}$; the latter class is usually more convenient to work with. Some variations of our definition are also investigated. In particular, in this section we consider spaces of non-closed curves.

The main tools in the paper are introduced in \S2. Given a curve $\ga\in \sr L_{\ka_0}^{+\infty}$, we assign to $\ga$ certain maps $B_\ga$ and $C_\ga$, called the regular and caustic bands spanned by $\ga$, respectively. These are ``fat'' versions of the curve, which carry in geometric form important information on the curve. They are obtained by considering translations of $\ga$ in the direction determined by its normal vector. We separate our curves into two main classes: If the image of the caustic band of a curve is contained in a hemisphere, the curve is called condensed; if this image contains two antipodal points, the curve is diffuse. Vaguely speaking, a condensed curve is one which does not wander too much about $\Ss^2$.  

In \S 3, the grafting construction is introduced. Grafting a curve consists in cutting it at well chosen points, moving the resulting arcs and inserting new arcs in the gaps that arise (see fig.~\ref{F:enxerto}). If the curve is diffuse, then we can graft arbitrarily long arcs of circles onto it. These are converted to loops, which are then spread along the curve (see fig.~\ref{F:adding2}), so that it becomes loose in the sense that the constraints on the curvature become irrelevant. As a result, it is possible to deform it into a circle traversed a certain number of times, which is the canonical curve in our spaces. 

We reach the same conclusion for condensed curves in \S 4, where a notion of rotation number for curves of this type is also introduced: If $\ga$ is condensed, then we consider the set of all hemispheres which contain the image of its caustic band. Any such hemisphere is uniquely determined by a unit vector. Taking the average, we obtain a distinguished hemisphere, corresponding to some $h_\ga\in \Ss^2$, containing $\Im(C_\ga)$. This $h_\ga$ depends continuously on $\ga$. The rotation number $\nu(\ga)$ of $\ga$ is defined as the opposite of the usual rotation number of the image of $\ga$ under stereographic projection through $-h_\ga$. For $\ka_0<0$, the proof that $\ga\in \sr L_{\ka_0}^{\infty}$ may be deformed into a circle involves a generalization of the regular band of a curve, and for $\ka_0\geq 0$ the tools are M\"obius transformations and a version of the Whitney-Graustein theorem. Actually, it will be seen that the set of condensed curves in $\sr L_{\ka_0}^{+\infty}$ having a fixed rotation number is weakly homotopy equivalent to $\SO_3$, for any $\ka_0$ (but it may not be a connected component).

Even though there exist curves which are neither condensed nor diffuse, any such curve is homotopic within $\sr L_{\ka_0}^{+\infty}$ to a curve of one of these two types. The results used to establish this are contained in \S 5. There, a more abstract version of rotation number for non-diffuse curves is introduced, and a bound on the total curvature of a non-diffuse curve in $\sr L_{\ka_0}^{+\infty}$ which depends only on $\ka_0$ and its rotation number is obtained. This is used to deduce that, by grafting the curve indefinitely, we must obtain either a condensed or a diffuse curve.

In \S 6 we decide when it is possible to deform a circle traversed $k$ times into a circle traversed $k+2$ times in $\sr L_{\ka_0}^{+\infty}$. It is seen that this is possible if and only if $k\geq n-1=\fl{\frac{\pi}{\rho_0}}$ (where $\rho_0=\arccot \ka_0$), and an explicit homotopy when this is the case is presented. It is also shown that the set of condensed curves in $\sr L_{\ka_0}^{+\infty}$ with fixed rotation number $k< n-1$ is a connected component of this space ($n\geq 3$).

 
The proofs of the main theorems are given in \S 7, after most of the work has been done. We also obtain the following direct characterization of the components. 
Let $\rho_0=\arccot \ka_0$,  $n=\left\lfloor{\frac{\pi}{\rho_0}}\right\rfloor+1$ and $\sig$ be the sign of $(-1)^n$. Then $\ga\in \sr L_{\ka_0}^{+\infty}$ lies in:
	\begin{enumerate}
	\item [(i)] $\sr L_j$ if and only if it is condensed and has rotation number $j$ $(1\leq j\leq n-2)$.
	\item [(ii)] $\sr L_{n-1}$ if and only if $\ga\in \sr I_{-\sig}$ and either it is non-condensed or condensed with rotation number $\nu(\ga)\geq n-1$.
	\item [(iii)] $\sr L_{n}$ if and only if $\ga\in \sr I_{\sig}$ and either it is non-condensed or condensed with rotation number $\nu(\ga)\geq n-1$.
\end{enumerate}
Here $\sr I_{\pm}$ are the two components of $\sr I=\sr C_{-\infty}^{+\infty}$, and $\sr L_j$ is the connected component of $\sr L_{\ka_0}^{+\infty}$ which contains circles traversed $j$ times, $j=1,\dots,n$. In particular, this yields a simple criterion to decide whether two curves $\ga_1,\,\ga_2\in \sr L_{\ka_1}^{\ka_2}$ are homotopic within this space, provided only that we have parametrizations of $\ga_1$ and $\ga_2$. 

Finally, some basic results on convexity in $\Ss^n$ (none of which is new) that are used throughout the work are collected in an appendix.

\subsection*{Acknowledgements}The content of this paper is essentially contained in the Ph.D.~thesis \cite{tese} of the second author, who was advised by the first. Both authors gratefully acknowledge the financial support of \tsc{cnpq}, \tsc{capes} and \tsc{faperj}, particularly during the second author's graduate studies. We thank the members of the Ph.D.~committee for several interesting suggestions. Very special thanks go to Boris Shapiro for helpful conversations with the first author during his visit to Stockholm University, which  inspired the main problem considered in this work. Finally, we thank the anonymous referee and Y.~Eliashberg for their useful suggestions and comments.


\section{Spaces of Curves on the Sphere with Constrained Geodesic Curvature}\label{S:belowandabove}

\subsection*{Basic definitions and notation}Let $M$ denote either the euclidean space $\R^{n+1}$ or the unit sphere $\Ss^n\subs \R^{n+1}$, for some $n\geq 1$. By a \tdef{curve} $\ga$ in $M$ we mean a continuous map $\ga\colon [a,b]\to M$. A curve will be called \tdef{regular} when it has a continuous and nonvanishing derivative; in other words, a regular curve is a $C^1$ immersion of $[a,b]$ into $M$. For simplicity, the interval where $\ga$ is defined will usually be $[0,1]$.

Let $\ga\colon [0,1]\to \Ss^2$ be a regular curve and let $\abs{\ }$ denote the usual Euclidean norm. The \tdef{arc-length parameter} $s$ of $\ga$ is defined by
\begin{equation*}
s(t)=\int_0^t \abs{\dot{\ga}(t)}\,dt,
\end{equation*}
and $L=\int_0^1 \abs{\dot \ga(t)}\,dt$ is called the \tdef{length} of $\ga$. Derivatives with respect to $t$ and $s$ will be systematically denoted by a $\dot{\vphantom{\ga}}$ and a $\vphantom{\ga}'$, respectively; this convention extends to higher-order derivatives as well. 


Up to homotopy, we can always assume that a family of curves is parametrized proportionally to arc-length.
\begin{lemma}\label{L:reparametrize}
	Let $A$ be a topological space and let $a\mapsto \ga_a$ be a continuous map from $A$ to the set of all $C^r$ regular curves $\ga\colon [0,1]\to M$ \tup($r\geq 1$\tup) with the $C^r$ topology. Then there exists a homotopy $\ga_a^u\colon [0,1] \to M$, $u\in [0,1]$, such that for any $a\in A$:
	\begin{enumerate}
		\item [(i)] $\ga_a^0=\ga_a$ and $\ga_a^1$ is parametrized so that $\vert\dot{\ga}_a^1(t)\vert$ is independent of $t$.
		\item [(ii)]  $\ga_a^u$ is an orientation-preserving reparametrization of $\ga_a$, for all $u\in [0,1]$. 
	\end{enumerate} 
\end{lemma}
\begin{proof}
	Let $s_a(t)=\int_0^t\abs{\dot\ga_a(\tau)}\,d\tau$ be the arc-length parameter of $\ga_a$, $L_a$ its length and $\tau_a\colon [0,L_a]\to [0,1]$ the inverse function of $s_a$. Define $\ga_a^u\colon [0,1]\to M$ by:
	\begin{equation*}
\qquad		\ga_a^u(t)=\ga_a\big((1-u)t+u\tau_a(L_at)\big) \qquad (u,t\in [0,1],~a\in A).
	\end{equation*}
Then $\ga_a^u$ is the desired homotopy.
\end{proof}

The unit tangent vector to $\ga$ at $\ga(t)$ will always be denoted by $\ta(t)$. 
Set $M=\Ss^2$ for the rest of this section, and define the \tit{unit normal vector} $\no$ to $\ga$ by 
\begin{equation*}
	\no(t)=\ga(t)\times \ta(t),
\end{equation*} where $\times$ denotes the vector product in $\R^3$. 

Assume now that $\ga$ has a second derivative. By definition, the \tdef{geodesic curvature} $\ka(s)$ of $\ga$ at $\ga(s)$ is given by 
\begin{equation}\label{E:geodesiccurvature}
\ka(s)=\gen{\ta'(s),\no(s)}.
\end{equation}
Note that the geodesic curvature is not altered by an orientation-preserving reparametrization of the curve, but its sign is changed if we use an orientation-reversing reparametrization. 
Since the principal curvatures at any point of the sphere equal 1, the normal curvature of $\ga$ is also identically equal to 1. In particular, its \tdef{Euclidean curvature} $K$,  
\[
K(s)=\sqrt{1+\ka(s)^2},
\] never vanishes. 

Closely related to the geodesic curvature of a curve $\ga\colon [0,1]\to \Ss^2$ is the \tdef{radius of curvature} of $\ga$ at $\ga(t)$, which we define as the unique number $\rho(t)\in (0,\pi)$ satisfying
\begin{equation*}
\cot \rho(t)=\ka(t).
\end{equation*}
Note that the sign of $\ka(t)$ is equal to the sign of  $\frac{\pi}{2}-\rho(t)$. 

\begin{exm}
	A parallel circle of colatitude $\al$, for $0<\al<\pi$, has geodesic curvature $\pm \cot \al$ (the sign depends on the orientation), and radius of curvature $\al$ or $\pi-\al$ at each point. (Recall that the colatitude of a point measures its distance from the north pole along $\Ss^2$.) The radius of curvature $\rho(t)$  of an arbitrary curve $\ga$ gives the size of the radius of the osculating circle to $\ga$ at $\ga(t)$, measured along $\Ss^2$ and taking the orientation of $\ga$ into account.
\end{exm} 
\begin{figure}[ht]
	\begin{center}
		\includegraphics[scale=.30]{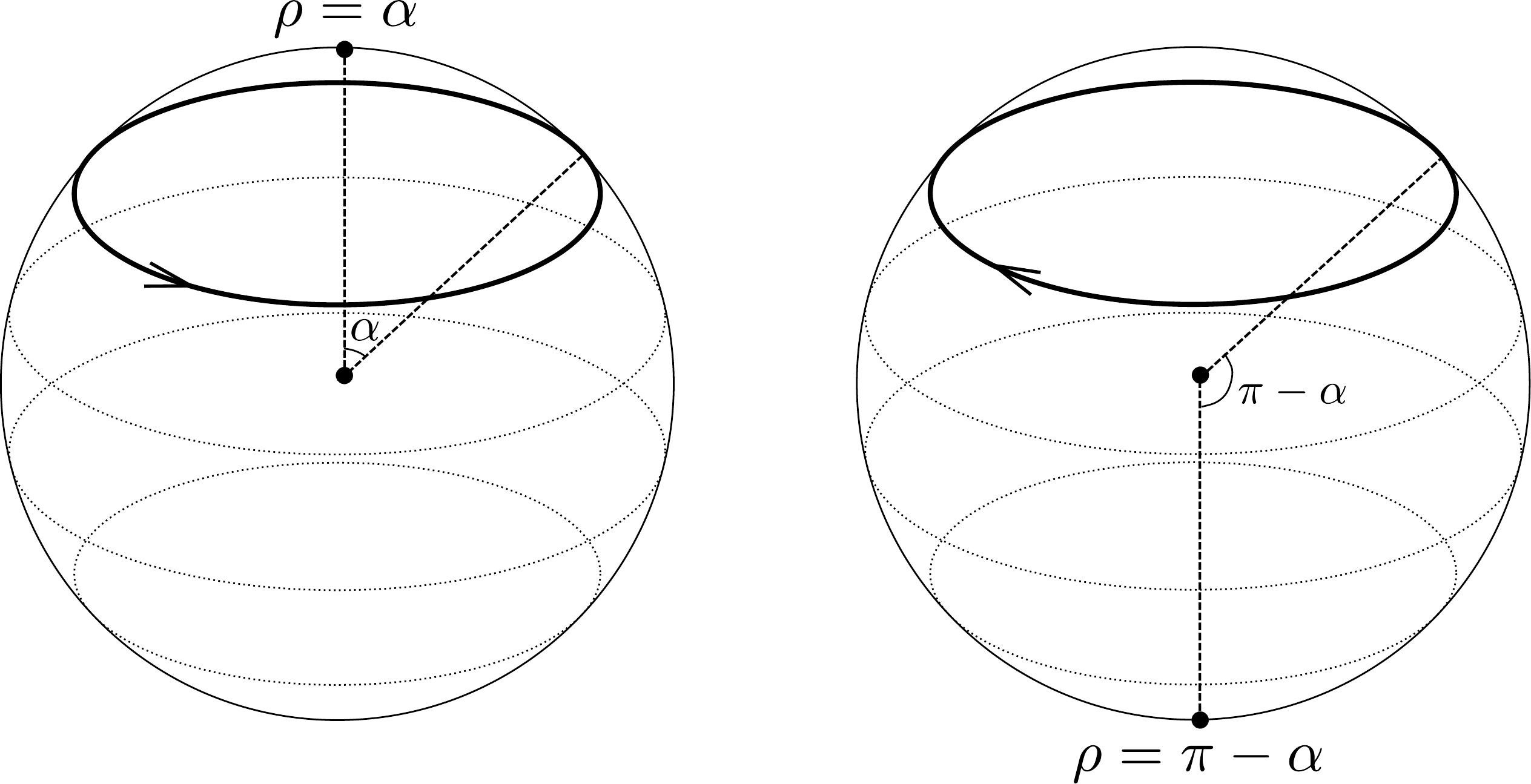}
		\caption{A parallel circle of colatitude $\al$ has radius of curvature $\al$ or $\pi-\al$, depending on its orientation. We adopt the convention that the center (on $\Ss^2$) of a circle lies to its left, hence in the first figure the center is taken to be the north pole, and in the second, the south pole.}
		\label{F:parallel}
	\end{center}
\end{figure}

If we consider $\ga$ as a curve in $\R^3$, then its ``usual'' radius of curvature $R$ is defined by $R(t)=\frac{1}{K(t)}=\sin \rho(t)$. We will rarely mention $R$ or $K$ again, preferring instead to work with $\rho$ and $\ka$, which are their natural intrinsic analogues in the sphere.

\subsection*{Spaces of curves}

Given $p\in \Ss^2$ and $v\in T_p\Ss^2$ of norm 1, there exists a unique $Q\in \SO_3$ having $p\in \R^3$ as first column and $v\in \R^3$ as second column. We obtain thus a diffeomorphism between $\SO_3$ and the unit tangent bundle $UT\Ss^2$ of $\Ss^2$. 

\begin{defn}\label{D:frame}For a regular curve $\ga\colon [0,1]\to \Ss^2$, its \tdef{frame} $\Phi_\ga\colon [0,1]\to \SO_3$ is the map given by
\begin{equation*}\label{E:frame}
\Phi_\ga(t)= \begin{pmatrix}
| & |& | \\
\ga(t) & \ta(t) & \no(t) \\
| & | & |
\end{pmatrix}.\footnote{In previous works of the first author, this is denoted by $\fr F_\ga$ and called the \tdef{Frenet frame} of $\ga$. We will not use this terminology to avoid any confusion with the usual Frenet frame of $\ga$ when it is considered as a curve in $\R^3$.} 
\end{equation*}
In other words, $\Phi_\ga$ is the curve in $UT\Ss^2$ associated with $\ga$, under the identification of $UT\Ss^2$ with $\SO_3$. We emphasize that it is not necessary that $\ga$ have a second derivative for $\Phi_\ga$ to be defined.
\end{defn}

Now let $-\infty\leq \ka_1<\ka_2\leq +\infty$ and $Q\in \SO_3$. We would like to study the space $\sr L_{\ka_1}^{\ka_2}(Q)$ of all regular curves $\ga\colon [0,1]\to \Ss^2$ satisfying:
\begin{enumerate}\label{D:conditions}
\item [(i)] $\Phi_\ga(0)=I$ and $\Phi_\ga(1)=Q$;
\item [(ii)] $\ka_1<\ka(t)<\ka_2$ for each $t\in [0,1]$.
\end{enumerate}
Here $I$ is the $3\!\times\!3$ identity matrix and $\ka$ is the geodesic curvature of $\ga$. Condition (i) says that $\ga$ starts at $e_1$ in the direction $e_2$ and ends at $Qe_1$ in the direction $Qe_2$. 

This definition is incomplete because we have not described the topology of $\sr L_{\ka_1}^{\ka_2}(Q)$. The most natural choice would be to require that the curves in this space be of class $C^2$, and to give it the $C^2$ topology. The foremost reason why we will not follow this course is that we would like to be able to perform some constructions which yield curves that are not $C^2$. 
We shall adopt a more complicated definition in order to avoid using convolutions or other tools all the time to smoothen a curve.


\begin{defn}
A function $f\colon [a,b]\to \R$ is said to be of class $H^1$ if it is an indefinite integral of some $g\in L^2[a,b]$. We extend this definition to maps $F\colon [a,b]\to \R^n$ by saying that $F$ is of class $H^1$ if and only if each of its component functions is of class $H^1$.
\end{defn}
Since $L^2[a,b]\subs L^1[a,b]$, an $H^1$ function is absolutely continuous (and differentiable almost everywhere). 

We shall now present an explicit description of a topology on $\sr L_{\ka_1}^{\ka_2}(Q)$ which turns it into a Hilbert manifold. The definition is unfortunately not very natural. However, we shall prove the following two results relating this space to more familiar concepts: First, for any $r\in \N$, $r\geq 2$, the subset of $\sr L_{\ka_1}^{\ka_2}(Q)$ consisting of $C^r$ curves will be shown to be dense in $\sr L_{\ka_1}^{\ka_2}(Q)$. Second, we will see that the space of $C^r$ regular curves satisfying conditions (i) and (ii) above, with the $C^r$ topology, is  homeomorphic to $\sr L_{\ka_1}^{\ka_2}(Q)$.\footnote{The definitions given here are straightforward adaptations of the ones in \cite{SalSha}, where they are used to study spaces of locally convex curves in $\Ss^n$ (which correspond to the spaces $\sr L_{0}^{+\infty}(Q)$ when $n=2$).}

Consider first a smooth regular curve $\ga\colon [0,1]\to \Ss^2$. From the definition of $\Phi_\ga$ we deduce that
\begin{equation}\label{E:frenet1}
	\dot{\Phi}_\ga(t)=\Phi_\ga(t)\La(t), \text{\ \ where\ \ }\La(t)=\begin{pmatrix}
		 0  & -\abs{\dot{\ga}(t)} & 0  \\
		 \abs{\dot{\ga}(t)}  & 0 & -\abs{\dot{\ga}(t)}\ka(t) \\
		    0 & \abs{\dot{\ga}(t)}\?\ka(t) & 0
	\end{pmatrix}\in \fr{so}_3
\end{equation}
is called the \tdef{logarithmic derivative} of $\Phi_\ga$ and $\ka$ is the geodesic curvature of $\ga$. 

Conversely, given $Q_0\in \SO_3$ and a smooth map $\La\colon [0,1]\to \fr{so}_3$ of the form
\begin{equation}\label{E:frenet2}
	\La(t)=\begin{pmatrix}
		0   & -v(t) & 0  \\
		 v(t)  &   0 & -w(t) \\
		 0  & w(t)  & 0
	\end{pmatrix},
\end{equation} let $\Phi\colon [0,1]\to \SO_3$ be the unique solution to the initial value problem 
\begin{equation}\label{E:ivp}
	\dot{\Phi}(t)=\Phi(t)\La(t),\quad \Phi(0)=Q_0.
\end{equation}
Define $\ga\colon [0,1]\to \Ss^2$ to be the smooth curve given by $\ga(t)=\Phi(t)e_1$. Then $\ga$ is regular if and only if $v(t)\neq 0$ for all $t\in [0,1]$, and it satisfies $\Phi_\ga=\Phi$ if and only if $v(t)>0$ for all $t$. (If $v(t)<0$ for all $t$ then $\ga$ is regular, but $\Phi_\ga$ is obtained from $\Phi$ by changing the sign of the entries in the second and third columns.) 

Equation \eqref{E:ivp} still has a unique solution if we only require that $v,w\in L^2[0,1]$ (cf.~\cite{CodLev}, p.~67). With this in mind, let $\E=L^2[0,1]\times L^2[0,1]$ and let $h\colon (0,+\infty)\to \R$ be the smooth diffeomorphism
\begin{equation}\label{E:thereason}
	h(t)=t-t^{-1}.
\end{equation} 
For each pair $\ka_1<\ka_2\in \R$, let $h_{\ka_1,\,\ka_2} \colon (\ka_1,\ka_2)\to \R$ be the smooth diffeomorphism
\begin{equation*}
	h_{\ka_1,\,\ka_2}(t)=(\ka_1-t)^{-1}+(\ka_2-t)^{-1}
\end{equation*}
and, similarly, set
\begin{alignat*}{10}
&h_{-\infty,+\infty}\colon \R\to \R\qquad  & &h_{-\infty,+\infty}(t)=t \\
 &h_{-\infty,\ka_2}\colon (-\infty,\ka_2)\to \R\qquad &  &h_{-\infty,\ka_2}(t)=t+(\ka_2-t)^{-1} \\
 &h_{\ka_1,+\infty}\colon (\ka_1,+\infty)\to \R \qquad & &h_{\ka_1,+\infty}(t)=t+(\ka_1-t)^{-1}.
\end{alignat*}

\begin{defn}
Let $\ka_1,\ka_2$ satisfy $-\infty\leq \ka_1<\ka_2\leq +\infty$. A curve $\ga\colon [0,1]\to \Ss^2$ will be called $(\ka_1,\ka_2)$-\tdef{admissible} if there exist $Q_0\in \SO_3$ and a pair $(\hat v,\hat w)\in \E$ such that $\ga(t)=\Phi(t)e_1$ for all $t\in [0,1]$, where $\Phi$ is the unique solution to equation \eqref{E:ivp}, with $v,w$ given by 
\begin{equation}\label{E:Sobolev}
	v(t)=h^{-1}(\hat{v}(t)),\text{\quad }w(t)=v(t)\?h^{-1}_{\ka_1,\,\ka_2}(\hat{w}(t)).
\end{equation}
When it is not important to keep track of the bounds $\ka_1,\ka_2$, we shall say more simply that $\ga$ is \tdef{admissible}.
\end{defn}

 In vague but more suggestive language, an admissible curve $\ga$ is essentially an $H^1$ frame $\Phi\colon [0,1]\to \SO_3$ such that $\ga=\Phi e_1\colon [0,1]\to \Ss^2$ has geodesic curvature in the interval $(\ka_1,\ka_2)$. The unit tangent (resp.~normal) vector $\ta(t)=\Phi(t)e_2$ (resp.~$\no(t)=\Phi(t)e_3$) of $\ga$ is thus defined everywhere on $[0,1]$, and it is absolutely continuous as a function of $t$. The curve $\ga$ itself is, like $\Phi$, of class $H^1$. However, the coordinates of its velocity vector $\dot\ga(t)=v(t)\?\Phi(t)e_2$ lie in $L^2[0,1]$, so the latter is only defined almost everywhere. The geodesic curvature of $\ga$, which is also defined a.e., is given by 
\[
\ka(t)=\frac{1}{v(t)}\gen{\dot\ta(t),\no(t)}=\frac{w(t)}{v(t)}=h^{-1}_{\ka_1,\,\ka_2}(\hat{w}(t))\in (\ka_1,\ka_2)
\]
(cf.~\eqref{E:frenet1}, \eqref{E:frenet2} and \eqref{E:Sobolev}). Note also that if we parametrize $\ga$ by (a multiple of) its arc-length parameter instead, then it becomes a $C^1$ curve, for then $\ga'=\ta$ is absolutely continuous.

\begin{urem}
	The reason for the choice of the specific diffeomorphism $h\colon (0,+\infty)\to \R$ in \eqref{E:thereason} (instead of, say, $h(t)=\log t$) is that we need $h^{-1}(t)$ to diverge linearly to $\pm \infty$ as $t\to 0,+\infty$ in order to guarantee that $v=h^{-1}\circ \hat{v}\in L^2[0,1]$ whenever $\hat{v}\in L^2[0,1]$. The reason for the choice of the other diffeomorphisms is analogous.
\end{urem}

\begin{defn}\label{D:Little0}
Let $-\infty\leq \ka_1<\ka_2\leq +\infty$, $Q_0\in \SO_3$. Define $\sr L_{\ka_1}^{\ka_2}(Q_0,\cdot)$ to be the set of all $(\ka_1,\ka_2)$-admissible curves $\ga$ such that
\[
\Phi_\ga(0)=Q_0,
\]
where $\Phi_\ga$ is the frame of $\ga$. This set is identified with $\E$ via the correspondence $\ga\leftrightarrow (\hat v,\hat w)$, and this defines a (trivial) Hilbert manifold structure on $\sr L_{\ka_1}^{\ka_2}(Q_0,\cdot)$.
\end{defn}

In particular, this space is contractible by definition. We are now ready to define the spaces $\sr L_{\ka_1}^{\ka_2}(Q)$, which constitute the main object of study of this work.

\begin{defn}\label{D:Little}
Let $-\infty\leq \ka_1<\ka_2\leq +\infty$, $Q\in \SO_3$. We define $\sr L_{\ka_1}^{\ka_2}(Q)$ to be the subspace of $\sr L_{\ka_1}^{\ka_2}(I,\cdot)$ consisting of all curves $\ga$ in the latter space satisfying 
\begin{equation*}
	\Phi_\ga(0)=I\text{\quad and \quad}\Phi_\ga(1)=Q.\tag{i}
\end{equation*}
Here $\Phi_\ga$ is the frame of $\ga$ and $I$ is the $3\!\times\!3$ identity matrix.\footnote{The letter `L' in $\sr L_{\ka_1}^{\ka_2}(Q)$ is a reference to John A.~Little, who determined the connected components of $\sr L_{0}^{+\infty}(I)$ in \cite{Little}.}
\end{defn}

Because $\SO_3$ has dimension 3, the condition $\Phi_\ga(1)=Q$ implies that $\sr L_{\ka_1}^{\ka_2}(Q)$ is a closed submanifold of codimension 3 in $\E\equiv \sr L_{\ka_1}^{\ka_2}(I,\cdot)$. (Here we are using the fact that the map which sends the pair $(\hat v,\hat w)\in \tbf{E}$ to $\Phi(1)$ is a submersion; a proof of this when $\ka_1=0$ and $\ka_2=+\infty$ can be found in $\S$\?3 of \cite{Sal3}, and the proof of the general case is analogous.) The space $\sr L_{\ka_1}^{\ka_2}(Q)$ consists of closed curves only when $Q=I$. Also, when $\ka_1=-\infty$ and $\ka_2=+\infty$ simultaneously, no restrictions are placed on the geodesic curvature. The resulting space (for arbitrary $Q\in \SO_3$) is known to be homotopy equivalent to $\Om\Ss^3\du \Om\Ss^3$; see the discussion after (\ref{L:bit}). 

Note that we have natural inclusions $\sr L_{\ka_1}^{\ka_2}(Q)\inc \sr L_{\bar\ka_1}^{\bar\ka_2}(Q)$ whenever $\bar\ka_1\leq \ka_1<\ka_2\leq \bar\ka_2$. More explicitly, this map is given by:
\begin{equation*}
	\ga\equiv(\hat v,\hat w)\mapsto \big(\hat v, h_{\bar \ka_1,\bar\ka_2}\circ h_{\ka_1,\ka_2}^{-1}(\hat w)\big);
\end{equation*}
it is easy to check that the actual curve associated with the pair of functions in $\sr L_{\bar\ka_1}^{\bar\ka_2}(Q)$ on the right side (via \eqref{E:frenet2}, \eqref{E:ivp} and \eqref{E:Sobolev}) is the original curve $\ga$, so that the use of the term ``inclusion'' is justified. We remark that although it is a continuous injection, it is not a topological embedding unless $\bar\ka_1=\ka_1$ and $\bar \ka_2=\ka_2$.


\begin{lemma} Let $\sr M, \sr N$ be Banach manifolds. Then:\label{L:Hilbert}
	\begin{enumerate}
		\item [(a)] $\sr M$ is locally path-connected. In particular, its connected components and path components coincide. 
		\item [(b)] If $\sr M,\,\sr N$ are weakly homotopy equivalent, then they are in fact homeomorphic (diffeomorphic if $\sr M,\,\sr N$ are Hilbert manifolds). In particular, if $\sr M$ is weakly contractible, then it is contractible.
		\item [(c)] Let $\E$ and $\mathbf{F}$ be separable Banach spaces. Suppose $i\colon \mathbf{F}\to \E$ is a bounded, injective linear map with dense image and $\sr M\subs \E$ is a smooth closed submanifold of finite codimension.  Then $\sr N=i^{-1}(\sr M)$ is a smooth closed submanifold of\, $\mathbf{F}$ and $i\colon (\mathbf{F},\sr N)\to (\E,\sr M)$ is a homotopy equivalence of pairs.
	\end{enumerate}
\end{lemma}
\begin{proof} Part (a) is obvious. Part (b) follows from thm.~15 in \cite{PalHom} and cor.~3 in \cite{Henderson1}. Part (c) is thm.~2 in \cite{bst}. 
\end{proof}

\begin{lemma}\label{L:net}
	Let $\E$ be a separable Hilbert space, $D\subs \E$ a dense vector subspace,  $L\subs \E$ a submanifold of finite codimension and $U$ an open subset of $L$.  Then the set inclusion $D\cap U\to U$ induces surjections on homotopy groups.
	\end{lemma}
\begin{proof}
	We shall prove the lemma when $L=h^{-1}(0)$ for some submersion $h\colon V\to \R^n$, where $V$ is an open subset of $\E$. This is sufficient for our purposes and the general assertion can be deduced from this by using a partition of unity subordinate to a suitable cover of $L$ by convex open sets. 
	
	Let $T$ be a tubular neighborhood of $U$ in $V$. Let $K$ be a compact simplicial complex and $f\colon K\to U$ a continuous map. We shall obtain a continuous $H\colon [0,2]\times K\to U$ such that $H(0,a)=f(a)$ for every $a\in K$ and $H(\se{2}\times K)\subs D\cap U$. Let $e_j$ denote the $j$-th vector in the canonical basis for $\R^n$, $e_0=-\sum_{j=1}^ne_j$ and let $\De\subs \R^n$ denote the $n$-simplex $[e_0,\dots,e_n]$. Let  $[x_0,x_1,\dots,x_n]\subs T$ be an $n$-simplex and $\vphi\colon \De\to [x_0,x_1,\dots,x_n]$ be given by 
	\begin{equation*}
		\vphi\bigg(\sum_{j=0}^n s_je_j\bigg)=\sum_{j=0}^n s_j x_j, \text{\ \ where\ \ $\sum_{j=0}^n s_j=1$\ \ and\ \ $s_j\geq 0$\ \ for all\ \ $j=0,\dots,n$}.
	\end{equation*}
	We shall say that $[x_0,x_1,\dots,x_n]$ is \tdef{good} if $h\circ \vphi\colon \De\to T$ is a diffeomorphism and $0\in (h\circ \vphi)(\Int \De)$. 
	
	Given $p\in T$, let $dh_p$ denote the derivative of $h$ at $p$ and $N_p=\ker(dh_p)$. Define $w_j\colon T\to \E$ by:
	\begin{equation}\label{E:w_j}
		w_j(p)=\big(dh_p|_{N_p^\perp}\big)^{-1}(e_j)\quad (p\in T,~j=0,\dots,n).
	\end{equation}
	Notice that when $\abs{\la_j}$ is small for each $j$, we have $h(p+\sum_j\la_j w_j(p))\approx \sum_j\la_j e_j$. Using compactness of $K$, we can find $r,\eps>0$ such that:
	\begin{enumerate}
		\item [(i)] For any $p\in f(K)$, $[p+rw_0(p),\dots,p+rw_n(p)]\subs T$ and it is good;
		\item [(ii)] If $p\in f(K)$ and $\abs{q_j-(p+rw_j(p))}<\eps$ for each $j$, then $[q_0,\dots,q_n]\subs T$ and it is good.
	\end{enumerate}
	Let $a_i$ ($i=1,\dots,m$) be the vertices of the triangulation of $K$. Set $v_i=f(a_i)$ and
	\begin{equation*}
		v_{ij}=v_i+rw_j(v_i)\quad (i=1,\dots,m,~j=0,\dots,n).
	\end{equation*}
	For each such $i,j$, choose $\te{v}_{ij}\in D\cap T$ with $\abs{\te{v}_{ij}-v_{ij}}<\frac{\eps}{2}$. Let
	\begin{alignat}{9}
		&v_{ij}(s)=(2-s)v_{ij}+(s-1)\te{v}_{ij},\ \text{ so that } \notag \\ \label{E:simest}
		&\abs{v_{ij}(s)-v_{ij}}<\frac{\eps}{2}\ \ (s\in [1,2],~i=1,\dots,m,~j=0,\dots,n).
	\end{alignat}
	For any $i,i'\in \se{1,\dots,m}$ and $j=0,\dots,n$, we have
	\begin{equation*}
		\abs{v_{ij}-v_{i'j}}\leq \abs{f(a_{i})-f(a_{i'})}+r\abs{w_j\circ f(a_{i})-w_j\circ f(a_{i'})}.
	\end{equation*}
	Since $f$ and the $w_j$ are continuous functions, we can suppose that the triangulation of $K$ is so fine that $\abs{v_{ij}-v_{i'j}}<\frac{\eps}{2}$ for each $j=0,\dots,n$ whenever there exists a simplex having $a_{i}$, $a_{i'}$ as two of its vertices. Let $a\in K$ lie in some $d$-simplex of this triangulation, say,  $a=\sum_{i=1}^{d+1}t_ia_i$ (where each $t_i>0$ and $\sum_i t_i=1$). Set
	\begin{equation*}
		z_{j}(s)=\sum_{i=1}^{d+1}t_iv_{ij}(s)\quad (s\in [1,2],~j=0,\dots,n).
	\end{equation*}
	Then $[z_0(s),\dots,z_n(s)]$ is a good simplex because condition (ii) is satisfied (with $p=v_1$):
	\begin{equation*}
		\bigg\vert \sum_{i=1}^{d+1}t_iv_{ij}(s)-v_{1j}\bigg\vert\leq \sum_{i=1}^{d+1}t_i\big(\abs{v_{ij}(s)-v_{ij}}+\abs{v_{ij}-v_{1j}}\big)<\eps,
	\end{equation*}
	where the strict inequality comes from \eqref{E:simest} and our hypothesis on the triangulation.
	Define $H(s,a)$ as the unique element of $h^{-1}(0)\cap [z_0(s),\dots,z_n(s)]$ ($s\in [1,2]$). Note that $H(2,a)$ is some convex combination of the $\te{v}_{ij}\in D$, hence $H(2,a)\in D\cap U$ for all $a\in K$.
	
	By reducing $r,\eps>0$ (and refining the triangulation of $K$) if necessary, we can ensure that
	\begin{equation*}
		(1-s)f(a)+sH(1,a)\in T\quad \text{for all $s\in [0,1]$ and $a\in K$.}
	\end{equation*}
	Let $\pr\colon T\to U$ be the associated retraction. Complete the definition of $H$ by setting:
	\begin{equation*}
		H(s,a)=\pr \big((1-s)f(a)+sH(1,a)\big)\in T\quad \text{($s\in [0,1]$,~$a\in K$).}
	\end{equation*}
	The existence of $H$ shows that $f$ is homotopic within $U$ to a map whose image is contained in $D\cap U$. Taking $K=\Ss^k$, we conclude that the set inclusion $D\cap U\to U$ induces surjective maps $\pi_k(D\cap U)\to \pi_k(U)$ for all $k\in \N$.
\end{proof}

\begin{lemma}\label{L:dense}
	Let $r\in \se{2,3,\dots,\infty}$. Then the subset of all $\ga\colon [0,1]\to \Ss^2$ of class $C^r$ is dense in $\sr L_{\ka_1}^{\ka_2}(Q)$.
\end{lemma}
\begin{proof}
	This follows from the previous lemma by taking $\E=L^2[0,1]\times L^2[0,1]$, $D=C^\infty[0,1]\times C^\infty[0,1]$ and $U$ any open subset of $L=\sr L_{\kappa_1}^{\kappa_2}(u,v)$.
\end{proof}

\begin{defn}\label{D:Cittle}
Let $-\infty\leq \ka_1<\ka_2\leq +\infty$, $Q\in \SO_3$ and $r\in \N$, $r\geq 2$. Define $\sr C_{\ka_1}^{\ka_2}(Q)$ to be the set, furnished with the $C^r$ topology, of all $C^r$ regular curves $\ga\colon [0,1]\to \Ss^2$ such that:
\begin{enumerate}
	\item [(i)] $\Phi_\ga(0)=I$ and $\Phi_\ga(1)=Q$;
	\item [(ii)]  $\ka_1<\ka(t)<\ka_2$ for each $t\in [0,1]$.
\end{enumerate}
\end{defn}
Notice that these spaces are Banach manifolds. The value of $r$ is not important, in the sense that different values of $r$ yield homeomorphic spaces. Because of this, after the next lemma, when we speak of $\sr C_{\ka_1}^{\ka_2}(Q)$, we will implicitly assume that $r=2$.

\begin{lemma}\label{L:C^2}
	Let $r\in \N$ $(r\geq 2)$, $Q\in \SO_3$ and $-\infty\leq \ka_1<\ka_2\leq +\infty$. Then the set inclusion $i\colon \sr C_{\ka_1}^{\ka_2}(Q)\inc \sr L_{\ka_1}^{\ka_2}(Q)$ is a  homotopy equivalence. Therefore, $\sr C_{\ka_1}^{\ka_2}(Q)$ and $\sr L_{\ka_1}^{\ka_2}(Q)$ are homeomorphic.
\end{lemma}
\begin{proof}In this proof we will highlight the differentiability class by denoting $\sr C_{\ka_1}^{\ka_2}(Q)$ by $\sr C_{\ka_1}^{\ka_2}(Q)^r$. Let $\E=L^2[0,1]\times L^2[0,1]$, let $\tbf{F}=C^{r-1}[0,1]\times C^{r-2}[0,1]$ (where $C^k[0,1]$ denotes the set of all $C^k$ functions $[0,1]\to \R$, with the $C^k$ norm) and let $i\colon \tbf{F}\to \E$ be set inclusion. Setting $\sr M=\sr L_{\ka_1}^{\ka_2}(Q)$, we conclude from (\ref{L:Hilbert}\?(c)) that $i\colon \sr N=i^{-1}(\sr M)\inc \sr M$ is a homotopy equivalence. We claim that $\sr N$ is homeomorphic to $\sr C_{\ka_1}^{\ka_2}(Q)^r$, where the homeomorphism is obtained by associating a pair $(\hat v,\hat w)\in \sr N$ to the curve $\ga$ obtained by solving \eqref{E:ivp} (with $\La$ defined by \eqref{E:frenet2} and \eqref{E:Sobolev} and $Q_0=I$), and vice-versa.
	
	Suppose first that $\ga\in \sr C_{\ka_1}^{\ka_2}(Q)^r$. Then $\abs{\dot\ga}$ (resp.~$\ka$) is a function $[0,1]\to \R$ of class $C^{r-1}$ (resp.~$C^{r-2}$). Hence, so are $\hat{v}=h\circ \abs{\dot\ga}$ and $\hat w=h_{\ka_1}^{\ka_2}\circ \ka$, since $h$ and $h_{\ka_1}^{\ka_2}$ are smooth. Conversely, if $(\hat{v},\hat{w})\in N$, then $v=h^{-1}\circ \hat v$ is of class $C^{r-1}$ and $w=(h_{\ka_1}^{\ka_2})^{-1}\circ \hat w$ of class $C^{r-2}$, and the frame $\Phi$ of the curve $\ga$ corresponding to that pair satisfies
	\begin{equation*}
		\dot \Phi=\Phi\La,\quad \La=\begin{pmatrix}
		0 & -\abs{\dot\ga} & 0 \\
		\abs{\dot\ga} & 0 & -\abs{\dot\ga} \ka \\
		0 & \abs{\dot\ga} \ka & 0
		\end{pmatrix}=\begin{pmatrix}
		0 & -v & 0 \\
		v & 0 & -w \\
		0 & w & 0
		\end{pmatrix}.
	\end{equation*}
	Since the entries of $\La$ are of class (at least) $C^{r-2}$, the entries of $\Phi$ are functions of class $C^{r-1}$. Moreover, $\ga=\Phi e_1$, hence
	\begin{equation*}
		\dot\ga=\dot\Phi e_1=\Phi\La e_1=v\?\Phi e_2,
	\end{equation*}
	and the velocity vector of $\ga$ is seen to be of class $C^{r-1}$. It follows that $\ga$ is a curve of class $C^r$. Finally, it is easy to check that the correspondence $(\hat v,\hat w)\dar \ga$ is continuous in both directions. 

	The last assertion follows from (\ref{L:Hilbert}\?(b)).
\end{proof}

\subsection*{Lifted frames} The (two-sheeted) universal covering space of $\SO_3$ is $\Ss^3$. Let us briefly recall the definition of the covering map $\pi\colon \Ss^3\to \SO_3$.\footnote{See \cite{CoxeterNEG} for more details and further information on quaternions and rotations.} We start by identifying $\R^4$ with the algebra $\Hh$ of quaternions, and $\Ss^3$ with the subgroup of unit quaternions. Given $z\in \Ss^3$, $v\in \R^4$, define a transformation $T_z\colon \R^4\to \R^4$ by $T_z(v)=zvz^{-1}=zv\ol{z}$. One checks easily that $T_z$ preserves the sum, multiplication and conjugation operations. It follows that, for any $v,w\in \R^4$,
\begin{alignat*}{10}
	4\gen{T_z(v),T_z(w)}=&\abs{T_z(v)+T_z(w)}^2-\abs{T_z(v)-T_z(w)}^2 \\
	=&\abs{v+w}^2-\abs{v-w}^2=4\gen{v,w},
\end{alignat*} 
where $\gen{\,,\,}$ denotes the usual inner product in $\R^4$. Thus $T_z$ is an orthogonal linear transformation of $\R^4$. Moreover, $T_z(\mbf{1})=\mbf{1}$ (where $\mbf{1}$ is the unit of $\Hh$), hence the three-dimensional vector subspace $\{0\}\times \R^3\subs \R^4$ consisting of  the purely imaginary quaternions is invariant under $T_z$. The element $\pi(z)\in \SO_3$ is the restriction of $T_z$ to this subspace, where $(a,b,c)\in \R^3$ is identified with the quaternion $a\mbf{i}+b\mbf{j}+c\mbf{k}$.

In what follows we adopt the convention that $\Ss^3$ (resp.~$\SO_3$) is furnished with the Riemannian metric inherited from $\R^4$ (resp.~$\R^9$).

\begin{lemma}\label{L:2sqrt2}
	Let $\gen{\?,\?}$ denote the metric in $\Ss^3$ and $\ggen{\?,\?}$ the metric in $\SO_3$. Then $\pi^*\ggen{\?,\?}=8\gen{\?,\?}$, where $\pi^*\ggen{\?,\?}$ denotes the pull-back of $\ggen{\?,\?}$ by $\pi$.
\end{lemma}
\begin{proof} The proof is a straightforward calculation. The details may be found in \cite{tese}, (2.11).\end{proof}

 \begin{defn}\label{D:liftedframe}
 Let $\Phi_\ga\colon [0,1]\to \SO_3$ be the frame of an admissible curve $\ga$ and let $z\in \Ss^3$ satisfy $\pi(z)=\Phi_\ga(0)$. We define the \tdef{lifted frame} $\te{\Phi}_\ga^z\colon [0,1]\to \Ss^3$ to  be the lift of $\Phi_\ga$ to $\Ss^3$, starting at $z$. When $\Phi_\ga(0)=I$ we adopt the convention that $z=\mbf{1}$, and we denote the lifted frame simply by $\te{\Phi}_\ga$. 
 \end{defn}
 
 Here is a simple but important application of this concept.

\begin{lemma}\label{L:bit}
	Let $\ga_0,\ga_1\in \sr L_{\ka_1}^{\ka_2}(Q)$, for some $Q\in \SO_3$, and suppose that $\ga_0,\ga_1$ lie in the same connected component of this space. Then $\te{\Phi}_{\ga_0}(1)=\te{\Phi}_{\ga_1}(1)$.
\end{lemma}
\begin{proof}
	Since $\sr L_{\ka_1}^{\ka_2}(Q)$ is a Hilbert manifold, its path and connected components coincide. Therefore, to say that $\ga_0,\ga_1$ lie in the same connected component of $\sr L_{\ka_1}^{\ka_2}(Q)$ is the same as to say that there exists a continuous family of curves $\ga_s\in \sr L_{\ka_1}^{\ka_2}(Q)$ joining $\ga_0$ and $\ga_1$, $s\in [0,1]$. The family $\Phi_{\ga_s}$ yields a homotopy between the paths $\Phi_{\ga_0}$ and $\Phi_{\ga_1}$ in $\SO_3$. (Recall that each of the frames $\Phi_{\ga_s}$ is (absolutely) continuous.) By the homotopy lifting property of covering spaces, the paths $\te{\Phi}_{\ga_0}$ and $\te{\Phi}_{\ga_1}$ are also homotopic in $\Ss^3$ (fixing the endpoints).
\end{proof}

\subsection*{The role of the initial and final frames}We will now study how the topology of $\sr L_{\ka_1}^{\ka_2}(Q)$ changes if we consider variations of condition (i) in (\ref{D:Little}); by the end of the section it should be clear that our original definition is sufficiently general. A summary of all the definitions considered here is given in table form on p.~\pageref{Ta:spaces}. 

\label{omegao}For fixed $z\in \Ss^3$, let $\Om_z\Ss^3$ denote the set of all continuous paths $\om\colon [0,1]\to \Ss^3$ such that $\om(0)=\mbf{1}$ and $\om(1)=z$, furnished with the compact-open topology. It can be shown (see \cite{BottTu}, p.~198) that $\Om_z\Ss^3\iso \Om\Ss^3$ for any $z\in \Ss^3$, where $\Om\Ss^3$ is the space of paths in $\Ss^3$ which start and end at $\mbf{1}\in \Ss^3$.\footnote{The notation $X\iso Y$ (resp.~$X\home Y$) means that $X$ is homotopy equivalent (resp.~homeomorphic) to $Y$.} The topology of this space is well understood; we refer the reader to \cite{BottTu}, \S 16, for more information. 

Now let $\ka_1<\ka_ 2$, $z\in \Ss^3$ be arbitrary and $Q=\pi(z)$. Define
\begin{equation}\label{E:Hirsch}
	F\colon \sr L_{\ka_1}^{\ka_2}(Q)\to \Om_{z}\Ss^3\cup \Om_{-z}\Ss^3 \iso \Om\Ss^3\du \Om\Ss^3\text{\ \ by\ \ } F(\ga)=\te{\Phi}_{\ga}.
\end{equation}
In the special case $\ka_1=-\infty$, $\ka_2=+\infty$, it follows from the Hirsch-Smale theorem (see \cite{Sma56}, thm.~C) that this map is a weak homotopy equivalence. In the general case this is false, however. For instance, $\Om \Ss^3\du \Om \Ss^3$ has two connected components, while Little has proved (\cite{Little}, thm.~1) that $\sr L_{0}^{+\infty}(I)$ has three connected components. We take this opportunity to recall the precise statement of Little's theorem and to introduce a new class of spaces.

\begin{defn}\label{D:arbitrary}
Let $-\infty\leq \ka_1<\ka_2\leq +\infty$. Define $\sr L_{\ka_1}^{\ka_2}$ to be the space of all $(\ka_1,\ka_2)$-admissible curves $\ga\colon [0,1]\to \Ss^2$ such that 
\begin{equation*}
	\Phi_\ga(0)=\Phi_\ga(1).
\end{equation*} 
\end{defn}

Note that the only difference between $\sr L_{\ka_1}^{\ka_2}(I)$ and $\sr L_{\ka_1}^{\ka_2}$ is that curves in the latter space may have arbitrary initial and final frames, as long as they coincide. An argument analogous to the one given for the spaces $\sr L_{\ka_1}^{\ka_2}(Q)$ shows that $\sr L_{\ka_1}^{\ka_2}$ is also a Hilbert manifold. In fact, we have the following relationship between the two classes. 

\begin{prop}\label{P:arbitrary}
	The space  $\sr L_{\ka_1}^{\ka_2}$ is homeomorphic to $\SO_3\times \sr L_{\ka_1}^{\ka_2}(I)$.
\end{prop}
\begin{proof}
	For $Q\in \SO_3$ and $\ga\in \sr L_{\ka_1}^{\ka_2}(I)$, let $Q\ga$ be the curve defined by $(Q\ga)(t)=Q(\ga(t))$. Because $Q$ is an isometry, the geodesic curvatures of $Q\ga$ at $(Q\ga)(t)$ and of $\ga$ at $\ga(t)$ coincide. Define $F\colon \SO_3\times \sr L_{\ka_1}^{\ka_2}(I)\to \sr L_{\ka_1}^{\ka_2}$ by $F(Q,\ga)=Q\ga$; clearly, $F$ is continuous. Since it has the continuous inverse $\eta\mapsto (\Phi_\eta(0),\Phi_\eta(0)^{-1}\eta)$, $F$ is a homeomorphism.
\end{proof}

Let us temporarily denote by $\sr L$ the space $\sr L_{-\infty}^0\du \sr L_{0}^{+\infty}$ studied by Little. We have $\sr L_{-\infty}^0 \home \sr L_0^{+\infty}$, since the map which takes a curve in $\sr L$ to the same curve with reversed orientation is a (self-inverse) homeomorphism mapping $\sr L_{-\infty}^0$ onto $\sr L_0^{+\infty}$. What is proved in \cite{Little} is that $\sr L$ has six connected components.\footnote{Little works with $C^2$ curves, but, as we have seen, this is not important.} Using prop.~(\ref{P:arbitrary}) and the fact that $\SO_3$ is connected, we see that Little's theorem is equivalent to the assertion that $\sr L_{0}^{+\infty}(I)$ has three connected components, as was claimed immediately above (\ref{D:arbitrary}).

A natural generalization of the spaces $\sr L_{\ka_1}^{\ka_2}(Q)$ is obtained by modifying condition (i) of (\ref{D:Little}) as follows. 
\begin{defn}\label{D:Little2}
Let $-\infty\leq \ka_1<\ka_2\leq +\infty$ and $Q_0,Q_1\in \SO_3$. Define $\sr L_{\ka_1}^{\ka_2}(Q_0,Q_1)$ to be the space of all $(\ka_1,\ka_2)$-admissible curves $\ga\colon [0,1]\to \Ss^2$  such that 
\begin{equation*}
	\Phi_\ga(0)=Q_0\text{\quad and\quad}\Phi_\ga(1)=Q_1.\tag{i$'$}
\end{equation*}
\end{defn}
Thus, the only difference between condition (i) on p.~\pageref{D:Little} and condition (i$'$) is that the latter allows arbitrary initial frames. 
\begin{prop}\label{T:initialfinal}
	Let $P,\,Q_0,\,Q_1\in \SO_3$. Then $\sr L_{\ka_1}^{\ka_2}(Q_0,Q_1)\home \sr L_{\ka_1}^{\ka_2}(PQ_0,PQ_1)$. In particular, $\sr L_{\ka_1}^{\ka_2}(Q_0,Q_1)\home \sr L_{\ka_1}^{\ka_2}(Q)$, where $Q=Q_0^{-1}Q_1$.
\end{prop}

\begin{proof}
	The proof is similar to that of (\ref{P:arbitrary}). The map $\ga\mapsto P\ga$ takes $\L_{\ka_1}^{\ka_2}(Q_0,Q_1)$ into $\L_{\ka_1}^{\ka_2}(PQ_0,PQ_1)$ and is continuous. The map $\ga\mapsto P^{-1}\ga$, which is likewise continuous, is its inverse.
\end{proof}

Of course, we could also consider the spaces $\sr L_{\ka_1}^{\ka_2}(\cdot,Q)$, consisting of all $(\ka_1,\ka_2)$-admissible curves $\ga$ having final frame $\Phi_{\ga}(1)=Q\in \SO_3$ (but arbitrary initial frame).\label{D:Little00} Like $\sr L_{\ka_1}^{\ka_2}(Q,\cdot)$, this space is contractible. To see this, one can go through the definition to check that it is indeed homeomorphic to $\E$, or, alternatively, one can observe that the map $\ga\mapsto \bar{\ga}$, $\bar{\ga}(t)=\ga(1-t)$, establishes a homeomorphism 
\[
\sr L_{\ka_1}^{\ka_2}(\cdot,Q)\home \sr L_{\ka_1}^{\ka_2}(QR,\cdot),
\]
where 
\begin{equation*}
	R=\begin{pmatrix}
		1   & 0 &0   \\
		  0 & -1 & 0  \\
		  0 & 0 & -1 
	\end{pmatrix}.
\end{equation*}
Finally, we could study the space \label{D:solto}$\sr L_{\ka_1}^{\ka_2}(\cdot,\cdot)$ of all $(\ka_1,\ka_2)$-admissible curves, with no conditions placed on the frames. The argument given in the proof of (\ref{P:arbitrary}) shows that 
\begin{equation*}
	\sr L_{\ka_1}^{\ka_2}(\cdot,\cdot) \home \SO_3\times  \sr L_{\ka_1}^{\ka_2}(I,\cdot).
\end{equation*} 
Hence,  $\sr L_{\ka_1}^{\ka_2}(\cdot,\cdot)$ is homeomorphic to $\SO_3\times \E$, and has the homotopy type of $\SO_3$.

Thus, the topology of the spaces $\sr L_{\ka_1}^{\ka_2}(Q,\cdot)$, $\sr L_{\ka_1}^{\ka_2}(\cdot,Q)$ and $\sr L_{\ka_1}^{\ka_2}(\cdot,\cdot)$ is uninteresting. We will have nothing else to say about these spaces.

\subsection*{The role of the bounds on the curvature}

Having analyzed the significance of condition (i) on p.~\pageref{D:conditions}, let us examine next condition (ii). Notice that we have allowed the bounds $\ka_1$, $\ka_2$ on the curvature to be infinite. The definition of radius of curvature is extended accordingly by setting $\arccot(+\infty)=0$ and $\arccot(-\infty)=\pi$. We can then rephrase (ii) as:
\begin{enumerate}
	\item[(ii)] $\rho(t)\in (\rho_2,\rho_1)$ for each $t\in [0,1]$.
\end{enumerate} 
Here $\rho$ is the radius of curvature of $\ga$ and $\rho_i=\arccot \ka_i\in [0,\pi]$,  $i=1,2$. The main result of this section relates the topology of $\sr L_{\ka_1}^{\ka_2}(Q)$ to the size $\rho_1-\rho_2$ of the interval $(\rho_2,\rho_1)$. Its proof relies on the following construction. 

Given $-\pi<\theta<\pi$ and an admissible curve $\ga\colon [0,1]\to \Ss^2$, define the \tit{translation} $\ga_\theta\colon [0,1]\to \Ss^2$ of $\ga$ by $\theta$ to be the curve given by
\begin{equation}\label{E:translation}
\ga_\theta(t)=\cos \theta\, \ga(t)+\sin \theta\, \no(t)\qquad (t\in [0,1]).
\end{equation}

\begin{exm}
	Let $0<\al<\frac{\pi}{2}$ and let $C$ be the circle of colatitude $\al$. Depending on the orientation, the translation of $C$ by $\theta$, $0\leq \theta\leq \al$, is either the circle of colatitude $\al+\theta$ or the circle of colatitude $\al-\theta$. In particular, taking $\theta=\al$ and a suitable orientation of $C$, the translation degenerates to a single point (the north pole). 
\end{exm}	 
This example shows that some care must be taken in the choice of $\theta$ for the resulting curve to be admissible.

\begin{lemma}\label{L:essentialtranslation}
Let $\ga\colon [0,1]\to \Ss^2$ be an admissible curve and $\rho$ its radius of curvature. Suppose 
\begin{equation}\label{E:regulartranslation}
\rho_2<\rho(t)<\rho_1 \text{\ \,for a.e.~$t\in [0,1]$\ \ and\ \ } \rho_1-\pi\leq \theta\leq\rho_2.
\end{equation}
Then $\ga_\theta$ is an admissible curve and its frame is given by:
\begin{equation}\label{E:rtheta}
\Phi_{\ga_\theta}=\Phi_\ga\?R_\theta\?,\quad\text{where}\quad		R_\theta=\begin{pmatrix}
			\cos \theta   &   0   &   -\sin \theta \\
			 0  &   1 &   0 \\
			 \sin \theta   &   0   &   \cos \theta 
		\end{pmatrix}.
	\end{equation} 
\end{lemma}
\begin{proof}
Let $\Psi=\Phi_\ga\?R_\theta$. Since $\Phi_\ga$ satisfies the differential equation \eqref{E:frenet1}, $\Psi$ satisfies 
\begin{equation*}
	\dot\Psi=\Psi\?(R_\theta^{-1}\La\?R_\theta).
\end{equation*}
A direct calculation shows that 
\begin{equation}\label{E:Latranslation}
	R_\theta^{-1}\La\?R_\theta=\begin{pmatrix}
			0  &   -\big( \cos\theta\? v-\sin\theta\?w \big)   &  0 \\
			\cos\theta\? v-\sin\theta\?w  &   0 &   -\big( \cos\theta\?w+\sin\theta\?v \big) \\
			 0  &  \cos\theta\?w+\sin\theta\?v    &   0
		\end{pmatrix},
\end{equation}
where $v=v(t)=\abs{\dot\ga(t)}$ and $w=w(t)=v(t)\ka(t)$. Also, $\Psi e_1=\ga_\theta$ by construction. To show that $\ga_\theta$ is admissible, it is thus only necessary to show that 
\begin{equation*}
	\cos\theta\? v(t)-\sin\theta\?w(t)=v(t)\big(\cos\theta-\sin\theta\?\cot\rho(t)\big)=\frac{v(t)}{\sin \rho(t)}\sin(\rho(t)-\theta)>0
\end{equation*}
for almost every $t\in [0,1]$, and this is true by our choice of $\theta$ and the fact that $v>0$.
\end{proof}
Thus, for $\theta$ satisfying \eqref{E:regulartranslation}, we obtain from \eqref{E:rtheta} that  the unit tangent vector $\ta_\theta$ and unit normal vector $\no_\theta$ to the translation $\ga_\theta$ of $\ga$ are given by:
\begin{equation}\label{E:normaltranslation}
	\ta_\theta(t)=\ta(t)\text{\quad and \quad}\no_\theta(t)=-\sin \theta\, \ga(t)+\cos \theta\, \no(t)
\end{equation}
for almost every $t\in [0,1]$.
 
 \begin{cor}\label{C:radiusofcurvature}
Let $\ga\colon [0,1]\to \Ss^2$ be an admissible curve and let $\theta$ satisfy \eqref{E:regulartranslation}. Then the radius of curvature $\bar\rho$ of $\ga_\theta$ is given by $\bar\rho=\rho-\theta$.
\end{cor}
\begin{proof}We already calculated the logarithmic derivative $\La_{\ga_\theta}$ of $\ga_\theta$ in \eqref{E:Latranslation}. The geodesic curvature $\bar\ka$ of $\ga_\theta$ is given by the quotient of the (3,2)-entry by the (2,1)-entry of this matrix (cf.~\eqref{E:frenet1}):
\begin{equation*}
	\bar\ka=\frac{\cos\theta w+\sin\theta v}{\cos\theta v-\sin\theta w}=\frac{v\sin\theta }{v\sin\theta}\frac{\cot \theta \?\frac{w}{v}+1}{\cot\theta-\frac{w}{v}}=\frac{\cot\theta\cot\rho+1}{\cot\theta-\cot\rho}=\cot(\rho-\theta),
\end{equation*}
where $v,w$ are the (2,1)- and (3,2)-entries of $\La_\ga$, respectively. Therefore, $\bar\rho=\rho-\theta$.
\end{proof}
\begin{urem}
	A different proof of (\ref{C:radiusofcurvature}) may be found in \cite{tese}. There we verify the formula for a circle, and then use the fact that the osculating circle to the translation $\ga_\theta$ at $\ga_\theta(t)$ is the translation of the osculating circle to $\ga$ at $\ga(t)$.
\end{urem}

\begin{lemma}\label{L:inversetranslation}
	Let $\ga\colon [0,1]\to \Ss^2$ be an admissible curve and suppose that \eqref{E:regulartranslation} holds. Then $(\ga_\theta)_{\vphi}=\ga_{\theta+\vphi}$ for any $\vphi\in (-\pi,\pi)$. In particular, $(\ga_{\theta})_{-\theta}=\ga$.
\end{lemma}
\begin{proof}
Note that $(\ga_{\theta})_\vphi$ is defined because $\ga_\theta$ is admissible, as we have just seen. Using \eqref{E:translation} and \eqref{E:normaltranslation} we obtain that 
\begin{equation*}
(\ga_\theta)_{\vphi}=\cos \vphi\?\big(\cos\theta\, \ga+\sin\theta\, \no\big)+\sin \vphi\?\big(-\sin \theta\, \ga+\cos \theta\, \no\big)=\ga_{\theta+\vphi}.\qedhere
\end{equation*}
\end{proof}

\begin{mthm}\label{T:size}
Let $Q\in \SO_3$, $\ka_1<\ka_2$, $\bar\ka_1<\bar\ka_2$, $\rho_i=\arccot \ka_i$, $\bar\rho_i=\arccot\bar\ka_i$. Suppose that  $\rho_1-\rho_2=\bar\rho_1-\bar\rho_2$. Then $\sr L_{\ka_1}^{\ka_2}(Q)\home \sr L_{\bar\ka_1}^{\bar\ka_2}(R_{-\theta}QR_\theta)$, where $\theta=\rho_2-\bar\rho_2$ and
\begin{equation*}
		R_\theta=\begin{pmatrix}
			\cos \theta   &   0   &   -\sin \theta \\
			 0  &   1 &   0 \\
			 \sin \theta  &   0   &   \cos \theta 
		\end{pmatrix}.
	\end{equation*} 
\end{mthm}
We recall that the bounds $\ka_i$, $\bar\ka_i$ may take on infinite values, and we adopt the conventions that $\arccot(+\infty)=0$ and $\arccot (-\infty)=\pi$.
\begin{proof}
	Let $\ga\in \sr L_{\ka_1}^{\ka_2}(Q)$ and let $\rho$ be its radius of curvature. We have:
\[
\rho_2<\rho(t)<\rho_1\text{ for a.e.~$t\in [0,1]$.}
\]
Set $\theta=\rho_2-\bar\rho_2$. Then \eqref{E:regulartranslation} is satisfied, so $\ga_\theta$ is and admissible curve. By (\ref{C:radiusofcurvature}), the radius of curvature $\bar{\rho}$ of $\ga_{\theta}$ is given by $\bar{\rho}=\rho-\theta$. Thus, 
\[
\bar\rho_2<\bar{\rho}(t)<\bar\rho_1 \text{ for a.e.~$t\in [0,1]$}.
\]
Together with (\ref{L:essentialtranslation}), this says that $F\colon \ga\mapsto \ga_{\theta}$ maps $\sr L_{\ka_1}^{\ka_2}(Q)$ into $\sr L_{\bar\ka_1}^{\bar\ka_2}(R_\theta,QR_\theta)$. Similarly,  translation by $-\theta$ is a map $G\colon \sr L_{\bar\ka_1}^{\bar\ka_2}(R_\theta,QR_\theta)\to \sr L_{\ka_1}^{\ka_2}(Q)$. By (\ref{L:inversetranslation}), the maps $F$ and $G$ are inverse to each other, hence 
\begin{equation*}
	\sr L_{\ka_1}^{\ka_2}(Q) \home \sr L_{\bar\ka_1}^{\bar\ka_2}(R_\theta,QR_\theta).
\end{equation*}
Finally, because $R_\theta^{-1}=R_{-\theta}$, (\ref{T:initialfinal}) guarantees that
\begin{equation*}
	\sr L_{\bar\ka_1}^{\bar\ka_2}(R_\theta,QR_\theta) \home \sr L_{\bar\ka_1}^{\bar\ka_2}(R_{-\theta}QR_\theta).\qedhere
\end{equation*}
\end{proof}

\begin{rem}\label{R:circletocircle}
	Taking $Q=I$ we obtain from (\ref{T:size}) that $\sr L_{\ka_1}^{\ka_2}(I)\home \sr L_{\bar\ka_1}^{\bar\ka_2}(I)$ ($\ka_i$, $\bar\ka_i$ as in the hypothesis of the theorem). It will also be important to us that under the homeomorphisms of (\ref{T:size}) and the following corollaries, the image of any circle traversed $k$ times is another circle traversed $k$ times. Indeed, the homeomorphism is obtained by translating (in the sense of (\ref{E:translation})) all the curves in a space by a fixed distance.
\end{rem}

\begin{cor}\label{C:symmetricinterval}
Let $Q\in \SO_3$ and $\ka_1<\ka_2$. Then $\sr L_{\ka_1}^{\ka_2}(Q)\home \sr L_{-\ka_0}^{+\ka_0}(P)$ for suitable $\ka_0>0$, $P\in \SO_3$. Moreover, if $Q=I$ then $P=I$ also.
\end{cor}
\begin{proof}
	Let $\rho_i=\arccot \ka_i$, $i=1,2$, and set 
	\begin{equation*}
		\bar\rho_1=\frac{\pi}{2}+\frac{\rho_1-\rho_2}{2},\quad \bar\rho_2=\frac{\pi}{2}-\frac{\rho_1-\rho_2}{2}\quad\text{and}\quad \ka_0=\cot(\bar\rho_2).
	\end{equation*}	
	The interval $(\bar\rho_2,\bar\rho_1)$ has the same size as $(\rho_2,\rho_1)$ by construction. Since $\cot (\bar\rho_1)=-\ka_0$, (\ref{T:size}) yields that $\sr L_{\ka_1}^{\ka_2}(Q)\home \sr L_{-\ka_0}^{+\ka_0}(R_{-\theta}QR_\theta)$, where $\theta=\frac{\rho_1+\rho_2-\pi}{2}$.
\end{proof}

\begin{cor}\label{C:belowandabove}
Let $Q\in \SO_3$ and $\ka_1<\ka_2$. Then $\sr L_{\ka_1}^{\ka_2}(Q)\home \sr L_{\ka_0}^{+\infty}(P)$ for suitable $\ka_0\in [-\infty,+\infty)$ and $P\in \SO_3$. Moreover, if $Q=I$ then $P=I$ also.
\end{cor}
\begin{proof}
Let $\rho_i=\arccot \ka_i$, $i=1,2$. Then the interval $(\rho_2,\rho_1)$ has the same size as the interval $(0,\rho_1-\rho_2)$. Hence, by (\ref{T:size}), $\sr L_{\ka_1}^{\ka_2}(Q)\home \sr L^{+\infty}_{\ka_0}(R_{-\theta}QR_\theta)$, where 
\begin{equation*}
	\ka_0=\cot (\rho_1-\rho_2)=\frac{1+\ka_1\ka_2}{\ka_2-\ka_1}\text{\quad and\quad }\theta=\rho_2.\qedhere
\end{equation*}
\end{proof}

Corollaries (\ref{C:symmetricinterval}) and (\ref{C:belowandabove}) both express the fact that, for fixed $Q\in \SO_3$, the topology of the spaces $\sr L_{\ka_1}^{\ka_2}(Q)$ depends essentially on one parameter, not two. The spaces of type $\sr L_{-\ka_0}^{+\ka_0}(Q)$ and $\sr L_{\ka_0}^{+\infty}(Q)$ have been singled out merely because they are more convenient to work with. For spaces of closed curves we have the following result relating the two classes, which is another simple consequence of (\ref{C:symmetricinterval}).

\begin{cor}\label{C:twoviewpoints} Let $\ka_0\in [-\infty,+\infty)$, $\ka_1\in (0,+\infty]$ and $\rho_i=\arccot(\ka_i)$, $i=0,1$. If $\rho_0=\pi-2\rho_1$ then $\sr L_{-\ka_1}^{+\ka_1}(I)\home \sr L_{\ka_0}^{+\infty}(I)$.\qed
\end{cor}

\renewcommand{\thetable}{\Roman{table}}
\begin{table}[ht]
\begin{center}
\begin{tabular}{ c  c  c c}\hline
Space & Definition & Condition on Frames & Topology \rule[-8pt]{0pt}{22pt}  \\ \hline
$\L_{\ka_1}^{\ka_2}(Q)$ & p.\,\pageref{D:Little}, (\ref{D:Little}) & $\Phi(0)=I$, $\Phi(1)=Q$ & depends on $\rho_1-\rho_2$, $Q$ \rule{0pt}{14pt} \\
$\L_{\ka_1}^{\ka_2}$ & p.\,\pageref{D:arbitrary}, (\ref{D:arbitrary}) & $\Phi(0)=\Phi(1)$ arbitrary & $\home \SO_3\times \sr L_{\ka_1}^{\ka_2}(I)$ \rule{0pt}{14pt} \\
$\L_{\ka_1}^{\ka_2}(Q_0,Q_1)$ & p.\,\pageref{D:Little2}, (\ref{D:Little2}) & $\Phi(0)=Q_0$, $\Phi(1)=Q_1$ & $\home \sr L_{\ka_1}^{\ka_2}(Q_0^{-1}Q_1)$ \rule{0pt}{14pt} \\
$\L_{\ka_1}^{\ka_2}(Q,\cdot)$ & p.\,\pageref{D:Little0}, (\ref{D:Little0}) & $\Phi(0)=Q$, $\Phi(1)$ arbitrary & $\home \E$ \rule{0pt}{14pt} \\ 
$\L_{\ka_1}^{\ka_2}(\cdot,Q)$ & p.\,\pageref{D:Little00} & $\Phi(0)$ arbitrary, $\Phi(1)=Q$ & $\home \E$ \rule{0pt}{14pt} \\
$\L_{\ka_1}^{\ka_2}(\cdot,\cdot)$ & p.\,\pageref{D:solto} & none & $\home \SO_3\times \E$ \rule{0pt}{14pt} \\
\end{tabular}\vspace{10pt}
\caption{Spaces of spherical curves with constrained geodesic curvature. Here $Q\in \SO_3$, $-\infty\leq \ka_1<\ka_2\leq +\infty$ and  $\rho_i=\arccot(\ka_i)$. The notation $X\home Y$  means that $X$ is homeomorphic to $Y$, and $\E$ denotes the separable Hilbert space  (to be specific, $\E=L^2[0,1]\times L^2[0,1]$). As we have already remarked, the spaces $\sr L_{\ka_1}^{\ka_2}(\cdot,Q)$, $\sr L_{\ka_1}^{\ka_2}(Q,\cdot)$ and $\sr L_{\ka_1}^{\ka_2}(\cdot,\cdot)$ will not be mentioned again. 
}
\label{Ta:spaces}
\end{center}\vspace{-12pt}
\end{table}



\section{The Connected Components of $\sr L_{\ka_1}^{\ka_2}$}\label{S:connected}

The following theorem is the main result of this work. It presents a description of the components of $\sr L_{\ka_1}^{\ka_2}$ in terms of $\ka_1$ and $\ka_2$. 
\begin{mthm}\label{T:components}
	Let $-\infty\leq \ka_1<\ka_2\leq +\infty$, $\rho_i=\arccot \ka_i$ \tup{(}$i=1,2$\tup{)} and $\fl{x}$ denote the greatest integer smaller than or equal to $x$. Then $\sr L_{\ka_1}^{\ka_2}$ has exactly $n$ connected components $\sr L_1,\dots,\sr L_n$, where
	\begin{equation}\label{E:components}
\qquad		n=\fl{\frac{\pi}{\rho_1-\rho_2}}+1
	\end{equation} 
and  $\sr L_j$ contains circles traversed $j$ times $(1\leq j\leq n)$. The component $\sr L_{n-1}$ also contains circles traversed $(n-1)+2k$ times, and $\sr L_n$ contains circles traversed $n+2k$ times, for $k\in \N$. Moreover, each of $\sr L_1,\dots,\sr L_{n-2}$ $(n\geq 3)$ is homeomorphic to $\SO_3\times \E$, where $\E$ is the separable Hilbert space.
\end{mthm}

\begin{figure}[ht]
	\begin{center}
		\includegraphics[scale=.56]{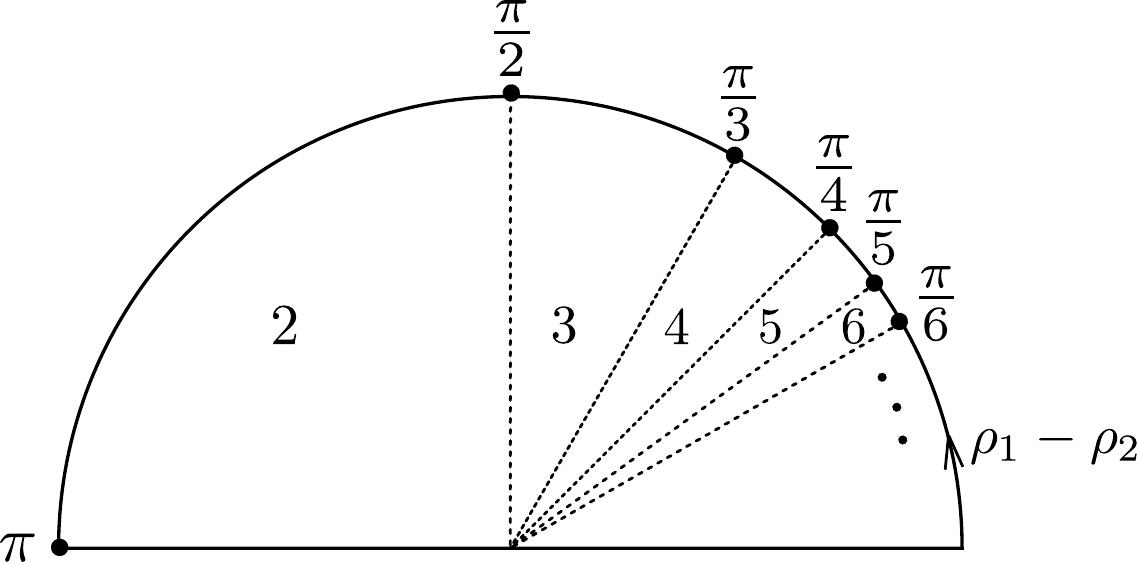}
		\caption{The number of connected components of $\sr L_{\ka_1}^{\ka_2}$, as $\rho_1-\rho_2$ varies in $(0,\pi]$ (where $\rho_i=\arccot \ka_i$). When $\rho_1-\rho_2=\frac{\pi}{n}$, $\sr L_{\ka_1}^{\ka_2}$ has $n+1$ components.}
		\label{F:components}
	\end{center}
\end{figure}


If we replace $\sr L_{\ka_1}^{\ka_2}$ with $\sr L_{\ka_1}^{\ka_2}(I)$ in the statement, then the conclusion is the same, except that the first $n-2$ components, $\sr L_1(I),\dots,\sr L_{n-2}(I)$ (where $\sr L_j(I)$ contains circles traversed $j$ times), are now homeomorphic to $\E$.. This is what will actually be proved; the theorem follows from this and the homeomorphism $\sr L_{\ka_1}^{\ka_2}\home \SO_3\times \sr L_{\ka_1}^{\ka_2}(I)$, which was established in (\ref{P:arbitrary}). We could also have replaced $\sr L_{\ka_1}^{\ka_2}$ by the space $\sr C_{\ka_1}^{\ka_2}$ of all $C^r$ closed curves ($r\geq 2$) whose geodesic curvatures lie in the interval $(\ka_1,\ka_2)$, with the $C^r$ topology, since this space is homeomorphic to the former, by (\ref{L:C^2}). 


We emphasize that the component $\sr L_j$ contains e\sz v\sz e\sz r\sz y parametrized circle in $\sr L_{\ka_1}^{\ka_2}$ traversed $j$ times (notation as in (\ref{T:components})); similarly, $\sr L_{n-1}$ (resp.~$\sr L_n$) contains all circles traversed $(n-1)+2k$ (resp.~$n+2k$) whose geodesic curvatures lie in $(\ka_1,\ka_2)$, for any $k\in \N$. A more direct characterization of the components in terms of the properties of a curve is given in (\ref{T:direct}).

\begin{exms}Let us first discuss some concrete cases of the theorem.

(a) We have already mentioned (on p.~\pageref{E:Hirsch}) that $\sr L_{-\infty}^{+\infty}=\sr I\iso \SO_3\times (\Om\Ss^3\du \Om\Ss^3)$ has two connected components $\sr I_+$ and $\sr I_-$, which are characterized by: $\ga\in \sr I_{+}$ if and only if $\te\Phi_\ga(1)=\te\Phi_\ga(0)$ and $\ga\in \sr I_{-}$ if and only if $\te\Phi_\ga(1)=-\te\Phi_\ga(0)$. This is consistent with (\ref{T:components}).

(b) Suppose $\ka_0<0$. Setting $\rho_2=0$ and $\rho_1=\arccot \ka_0$ in (\ref{T:components}), we find that $\sr L_{\ka_0}^{+\infty}$ also has two connected components. Since we have a continuous injection $\sr L_{\ka_0}^{+\infty}\to \sr L_{-\infty}^{+\infty}$, these components have the same characterization in terms of $\te{\Phi}$: two curves $\ga,\eta\in \sr L_{\ka_0}^{+\infty}$ are homotopic if and only if $\te{\Phi}_{\ga}(1)=\pm \te{\Phi}_\ga(0)$ and $\te{\Phi}_{\eta}(1)=\pm \te{\Phi}_\eta(0)$, with the same choice of sign for both curves.

(c) In contrast, $\sr L_{\ka_0}^{+\infty}$ has at least three connected components when $\ka_0\geq 0$. It has exactly three components in case 
\begin{equation*}
	0\leq \ka_0<\frac{1}{\sqrt{3}}.
\end{equation*}
The case $\ka_0=0$ is Little's theorem (\cite{Little}, thm.~1). If
\begin{equation*}
	\frac{1}{\sqrt{3}}\leq \ka_0< 1
\end{equation*}
it has four connected components and so forth.
\end{exms}

To sum up, as we impose starker restrictions on the geodesic curvatures, a homotopy which existed ``before'' may now be impossible to carry out. For instance, in any space $\sr L_{\ka_0}^{+\infty}$ with $\ka_0<0$, it is possible to deform a circle traversed once into a circle traversed three times. However, in $\sr L_0^{+\infty}$ this is not possible anymore, which gives rise to a new component.

The first part of theorem (\ref{T:components}) is an immediate consequence of the following results.  

\begin{mthm}\label{T:allisround}
		Let $-\infty\leq \ka_1<\ka_2\leq +\infty$. Every curve in $\sr L_{\ka_1}^{\ka_2}(I)$ (resp.~$\sr L_{\ka_1}^{\ka_2}$) lies in the same component as a circle traversed $k$ times, for some $k\in \N$ (depending on the curve).
\end{mthm}

\begin{mthm}\label{T:disconnects}
	Let $-\infty\leq \ka_1<\ka_2\leq +\infty$ and let $\sig_j\in \sr L_{\ka_1}^{\ka_2}(I)$ (resp.~$\sr L_{\ka_1}^{\ka_2}$) denote any circle traversed $j\geq 1$ times. Then $\sig_k$,~$\sig_{k+2}$ lie in the same component of $\sr L_{\ka_1}^{\ka_2}(I)$ (resp.~$\sr L_{\ka_1}^{\ka_2}$) if and only if 
	\begin{equation*}
\qquad		k\geq \left\lfloor{\frac{\pi}{\rho_1-\rho_2}}\right\rfloor\quad (\rho_i=\arccot\ka_i,~i=1,2).
	\end{equation*}
\end{mthm}

The following very simple result will be used implicitly in the sequel; it implies in particular that it does not matter which circle $\sig_k$ we choose in (\ref{T:allisround}) and (\ref{T:disconnects}).

\begin{lemma}\label{L:homocircles}
	Let $\sig,\te\sig\in \sr L_{\ka_1}^{\ka_2}(I)$ (resp.~$\sr L_{\ka_1}^{\ka_2}$) be parametrized circles traversed the same number of times. Then $\sig$ and $\te\sig$ lie in the same connected component of $\sr L_{\ka_1}^{\ka_2}(I)$ (resp.~$\sr L_{\ka_1}^{\ka_2}$).
\end{lemma}
\begin{proof}
	The proof is easy, and will be omitted. See \cite{tese}, lemma (4.4) for the details.
\end{proof}

Next we introduce the main concepts and tools used in the proofs of the theorems listed above. From now on we shall work almost exclusively with spaces of type $\sr L_{\ka_0}^{+\infty}$ and $\sr L_{\ka_0}^{+\infty}(I)$; we are allowed to do so by (\ref{C:belowandabove}).

\subsection*{The bands spanned by a curve}
Let $\ga\colon [0,1]\to \Ss^2$ be a $C^2$ regular curve. For $t\in [0,1]$, let $\chi(t)$ (or $\chi_\ga(t)$) be the center, on $\Ss^2$, of the osculating circle to $\ga$ at $\ga(t)$. (There are two possibilities for the center on $\Ss^2$ of a circle. To distinguish them we use the orientation of the circle, as in fig.~\ref{F:parallel}. The radius of curvature $\rho(t)$ is the distance from $\ga(t)$ to the center $\chi(t)$, measured along $\Ss^2$.)  The point $\chi(t)$ will be called the \tdef{center of curvature} of $\ga$ at $\ga(t)$, and the correspondence $t\mapsto \chi(t)$ defines a new curve $\chi\colon [0,1]\to \Ss^2$, the \tdef{caustic} of $\ga$. In symbols,
\begin{equation}\label{E:caustic}
	\chi(t)=\cos \rho(t)\?\ga(t)+\sin\rho (t)\?\no(t).
\end{equation}
Here, as always, $\rho=\arccot \ka$ is the radius of curvature and $\no$ the unit normal to $\ga$. Note that the caustic of a circle degenerates to a single point, its center. This is explained by the following result.

\begin{lemma}\label{L:caustic}
	Let $r\geq 2$, $\ga\colon [0,1]\to \Ss^2$ be a $C^r$ regular curve and $\chi$ its caustic. Then $\chi$ is a curve of class $C^{r-2}$. When $\chi$ is differentiable, $\dot\chi(t)=0$ if and only if $\dot\ka(t)=0$, where $\ka$ is the geodesic curvature of $\ga$.
\end{lemma}
\begin{proof}
	Again, the proof will be left to the reader. See \cite{tese}, (4.5) for the details.
\end{proof}

\begin{defns}\label{D:bands}
Let $\ka_0\in \R$, $\rho_0=\arccot \ka_0$ and $\ga\in \sr L_{\ka_0}^{+\infty}$. Define the \tdef{regular band} $B_\ga$ and the \tdef{caustic band} $C_\ga$ to be the maps
\begin{equation*}
B_\ga \colon [0,1]\times [\rho_0-\pi,0]\to \Ss^2 \quad\text{and\quad}C_\ga\colon [0,1]\times [0,\rho_0]\to \Ss^2 
\end{equation*} 
given by the same formula:
\begin{equation}\label{E:bands}
	(t,\theta)\mapsto \cos \theta\, \ga(t)+\sin \theta\, \no(t).
\end{equation}
The image of $C_\ga$ will be denoted by $C$, and the geodesic circle orthogonal to $\ga$ at $\ga(t)$ will be denoted by $\Ga_t$. As a set, 
\[
\Ga_t=\set{\cos \theta\, \ga(t)+\sin \theta\, \no(t)}{\theta\in [-\pi,\pi)}.
\]
\end{defns}
\begin{figure}[ht]
	\begin{center}
		\includegraphics[scale=.44]{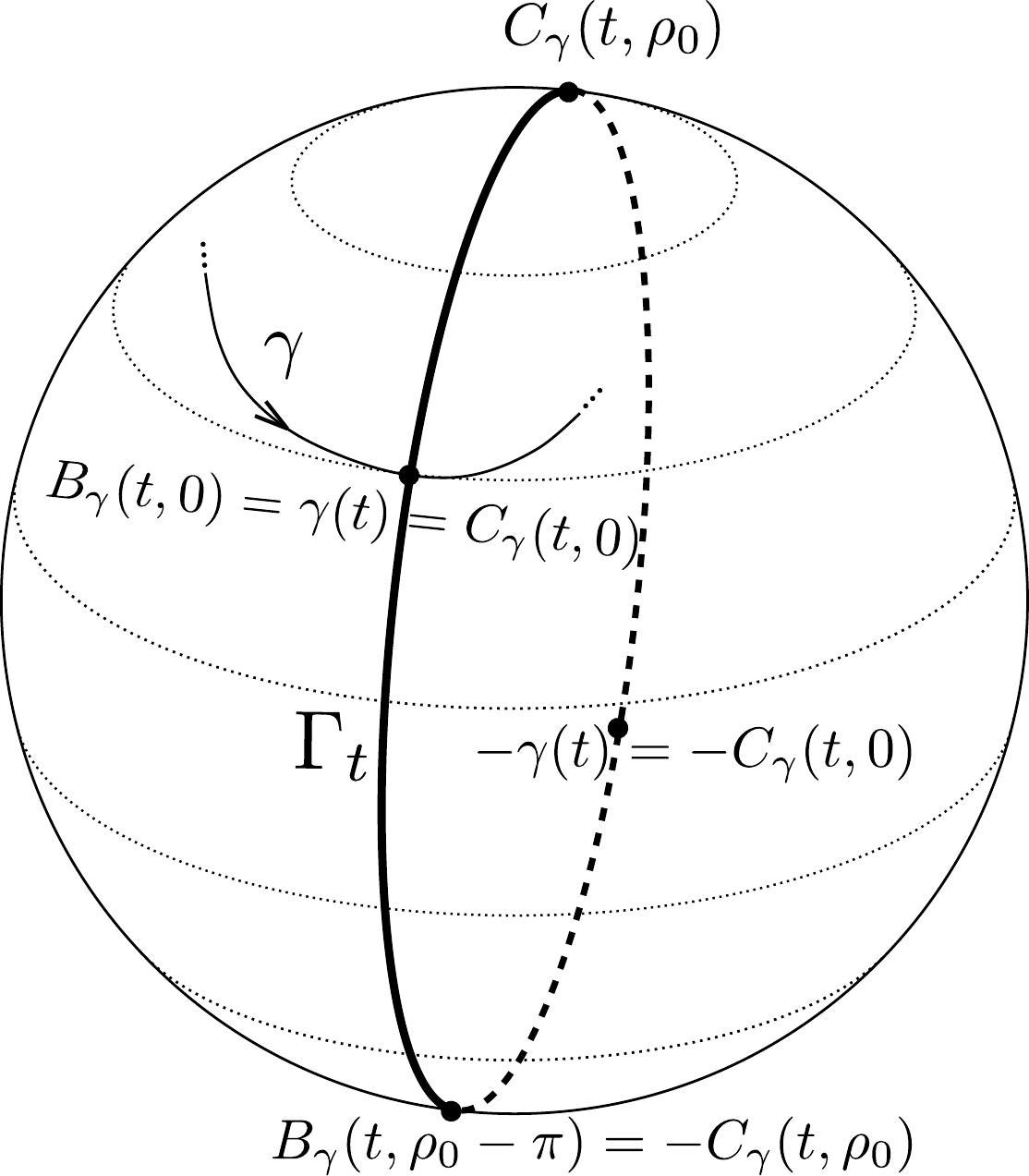}
		\caption{}
		\label{F:Gammat}
	\end{center}
\end{figure}

For fixed $t$, the images of $\pm B_\ga(t,\cdot)$ and $\pm C_\ga(t,\cdot)$ divide the circle $\Ga_t$ in four parts. Note also that $\chi_\ga(t)=C_\ga(t,\rho(t))$. 

\begin{lemma}\label{L:band}
	Let $\ga\in \sr L_{\ka_0}^{+\infty}$ and let $B_\ga\colon [0,1]\times [\rho_0-\pi,0]\to \Ss^2$ be the regular band spanned by $\ga$. Then:
	\begin{enumerate}
		\item The derivative of $B_\ga$ is an isomorphism at every point.
		\item $\frac{\bd B_{\ga}}{\bd \theta}(t,\theta)$ has norm 1 and is orthogonal to $\frac{\bd B_{\ga}}{\bd t}(t,\theta)$. Moreover, 
		\begin{equation*}
			\det \Big(B_\ga\,,\,\frac{\bd B_{\ga}}{\bd t}\, ,\, \frac{\bd B_{\ga}}{\bd \theta}  \Big)>0.
		\end{equation*}
		\item $C_\ga$ fails to be an immersion precisely at the points $(t,\rho(t))$ whose images form the caustic $\chi$.
	\end{enumerate} 
\end{lemma}
\begin{proof} We have:
	\begin{equation}\label{E:partialtheta}
		\frac{\bd B_{\ga}}{\bd \theta}(t,\theta)=-\sin\theta\, \ga(t)+\cos\theta\, \no(t).
	\end{equation}
and
	\begin{alignat}{10}
		\frac{\bd B_{\ga}}{\bd t}(t,\theta)&=\abs{\dot{\ga}(t)}\big(\cos\theta-\ka(t)\?\sin\theta\big)\?\ta(t) \label{E:partialt1}\\
		&=\frac{\abs{\dot{\ga}(t)}}{\sin\rho(t)}\sin(\rho(t)-\theta)\?\ta(t)\label{E:partialt2},
	\end{alignat}
where $\rho(t)=\arccot \ka(t)$ is the radius of curvature of $\ga$ at $\ga(t)$. The inequality $\ka_0<\ka<+\infty$ translates into $0<\rho<\rho_0$, hence the factor multiplying $\ta(t)$ in \eqref{E:partialt2} is positive for $\theta$ satisfying $\rho_0-\pi\leq \theta\leq 0$, and this implies (a) and (b). Part (c) also follows directly from \eqref{E:partialt2}, because $C_\ga$ and $B_\ga$ are defined by the same formula.
\end{proof}
%

Thus, $B_\ga$ is an immersion (and a submersion) at every point of its domain. It is merely a way of collecting the regular translations of $\ga$ (as defined on p.~\pageref{E:translation}) in a single map. 

If we fix $t$ and let $\theta$ vary in $(0,\rho_0)$, the section $C_\ga(t,\theta)$ of $\Ga_t$ describes the set of ``valid'' centers of curvature for $\ga$ at $\ga(t)$, in the sense that the circle centered at $C_\ga(t,\theta)$ passing through $\ga(t)$, with the same orientation as $\ga$, has geodesic curvature greater than $\ka_0$. This interpretation is important because it motivates some of the constructions that we consider ahead.

\subsection*{Condensed and diffuse curves}

\begin{defn}
Let $\ka_0\in \R$ and $\ga\in \sr L_{\ka_0}^{+\infty}$. We shall say that $\ga$ is \tdef{condensed} if the image $C$ of $C_\ga$ is contained in a closed hemisphere, and \tdef{diffuse} if $C$ contains antipodal points (i.e., if $C\cap {-}C\neq \emptyset$). 
\end{defn}
\begin{exms}
	 A circle in $\sr L_{\ka_0}^{+\infty}$ is always condensed for $\ka_0\geq 0$, but when $\ka_0<0$ it may or may not be condensed, depending on its radius. If a curve contains antipodal points then it must be diffuse, since $C_\ga(t,0)=\ga(t)$. By the same reason, a condensed curve is itself contained in a closed hemisphere. 
	 
	 There exist curves which are condensed and diffuse at the same time; an example is a geodesic circle in $\sr L_{\ka_0}^{+\infty}$, with $\ka_0<0$. There also exist curves which are neither condensed nor diffuse. To see this, let $\Ss^1$ be identified with the equator of $\Ss^2$ and let $\ze\in \Ss^1$ be a primitive third root of unity. Choose small neighborhoods $U_i$ of $\ze^i$ ($i=0,1,2$) and $V$ of the north pole in $\Ss^2$. Then the set $G$ consisting of all geodesic segments joining points of $U_1\cup U_2\cup U_3$ to points of $V$ does not contain antipodal points, nor is it contained in a closed hemisphere, by (\ref{L:closedhemisphere}). By taking $\rho_0=\arccot \ka_0$ to be very small, we can construct a curve $\ga\in \sr L_{\ka_0}^{+\infty}$ for which $C=\Im(C_\ga)\subs G$, but $\ze^i\in C$ for each $i$, so that $\ga$ is neither condensed nor diffuse. 
	
		 To sum up, a curve may be condensed, diffuse, neither of the two, or both simultaneously, but this ambiguity is not as important as it seems.
\end{exms} 

There exists a two-way correspondence between the unit sphere $\Ss^2$ and the set consisting of its closed (or open) hemispheres; namely, with $h\in \Ss^2$ we can associate
\begin{equation*}
H=\set{p\in \Ss^2}{\gen{h,p}\geq 0}.
\end{equation*}

Let $\ga\colon [0,1]\to \Ss^n$ be a (continuous) curve contained in the interior of $H$. As a consequence of the compactness of $[0,1]$, if $\te h\in \Ss^n$ is sufficiently close to $h$, then $\ga$ is also contained in the hemisphere $\te H$ corresponding to $\te h$. It is desirable to be able to select, in a natural way, a distinguished hemisphere among those which contain $\ga$. 

\begin{lemma}\label{L:closedbarycenter}
	Let $\ka_0\in \R$ and $\sr H\subs \sr L_{\ka_0}^{+\infty}$ be the subspace consisting of all $\ga$ whose image is contained in some closed hemisphere \tup(depending on $\ga$\tup). Then the map $h\colon \sr H\to \Ss^2$, which associates to $\ga$ the barycenter $h_\ga$ on $\Ss^2$ of the set of closed hemispheres that contain $\ga$, is continuous.\qed
\end{lemma}

More explicitly, the barycenter $h_\ga$ is obtained as follows: For fixed $\ga$, consider the set 
\begin{equation*}
	\sr S_\ga=\set{h\in \Ss^2}{\gen{\ga(t),h}\geq 0\ \text{for each $t\in [0,1]$}}.
\end{equation*}
It can be proved that the centroid of $\sr S_\ga$ in $\R^3$ is not the origin. The barycenter $h_\ga$ is taken to be the image on $\Ss^2$ of this centroid under gnomic (central) projection. We refer the reader to \S 3 of \cite{tese} for the proof of this lemma and also of the following one.

\begin{lemma}\label{L:causticbarycenter}
	Let $\ka_0\in \R$ and let $\sr O\subs \sr L_{\ka_0}^{+\infty}$ denote the subspace consisting of all condensed curves. Define a map $h\colon \sr O\to \Ss^2$ by $\ga\mapsto h_\ga$, where $h_\ga$ is the image under gnomic (central) projection of the centroid, in $\R^3$, of the set of closed hemispheres which contain $\Im(C_\ga)$. Then $h\colon \sr O \to \Ss^2$, $\ga\mapsto h_\ga$, is continuous.\qed
\end{lemma}


\section{Grafting}

 \begin{defn}\label{D:bycurvature}
 Let $\ga\colon [a,b]\to \Ss^2$ be an admissible curve. The \tdef{total curvature} $\tot(\ga)$ of $\ga$ is given by
 \begin{equation*}
 	\tot(\ga)=\int_a^bK(t)\abs{\dot\ga(t)}\,dt,
 \end{equation*}
 where 
 \begin{equation}\label{E:totalen}
 	K=\sqrt{1+\ka^2}=\csc\rho
 \end{equation}
 is the Euclidean curvature of $\ga$. We say that $\ga\colon [0,T]\to \Ss^2$, $u\mapsto \ga(u)$, is a \tdef{parametrization of $\ga$ by curvature} if
 \begin{equation*}
 	\big\vert\Phi_\ga'(u)\big\vert=\sqrt{2}\text{\ or, equivalently, \ }\big\vert\tilde{\Phi}_\ga'(u)\big\vert=\frac{1}{2}\text{\  for a.e.~$u\in [0,T]$}.
 \end{equation*}
 \end{defn}
 
 The equivalence of the two equalities comes from (\ref{L:2sqrt2}). The next result justifies our terminology.
 \begin{lemma}\label{L:parametrizedbycurvature}
 	Let $\ga\colon [0,T]\to \Ss^2$ be an admissible curve. Then:
 	\begin{enumerate}
 		\item [(a)] $\ga$ is parametrized by curvature if and only if 
 \begin{equation*}
\qquad 	\tot \big( \ga|_{[0,u]} \big)=u\text{ for every $u\in [0,T]$}.
 \end{equation*}
 		\item [(b)] If $\ga$ is parametrized by curvature then its logarithmic derivatives $\La=\Phi_\ga^{-1} \Phi_\ga'$ and $\te\La=\te\Phi_\ga^{-1}\te\Phi'$ are given by:
 		\begin{equation*}
 			\La(u)=\begin{pmatrix}
 				0   & -\sin \rho(u) & 0 \\
 				 \sin \rho(u)  &  0 & -\cos \rho(u) \\
 				 0 & \cos\rho(u) & 0   
 			\end{pmatrix},\quad \te{\La}(u)=\frac{1}{2}\big(\cos\rho(u)\mbf{i}+\sin\rho(u)\mbf{k}\big).
 		\end{equation*}
 	\end{enumerate}	
 \end{lemma}
 Here, as always, $\rho$ is the radius of curvature of $\ga$. In the expression for $\te{\La}$ above and in the sequel we are identifying $\Ss^3$ with the unit quaternions and the Lie algebra $\widetilde{\mathfrak{so}}_3=T_{\mbf{1}}\Ss^3$ (the tangent space to $\Ss^3$ at $\mbf{1}$) with the vector space of all imaginary quaternions. Also, it follows from (a) that if $\ga\colon [0,T]\to \Ss^2$ is parametrized by curvature then $T=\tot(\ga)$.
 \begin{proof}Let us denote differentiation with respect to $u$ by $'$. Using \eqref{E:totalen}, we deduce that 
 	\begin{equation}\label{E:tot0}
 		\La(u)=\abs{\ga'(u)}\begin{pmatrix}
 		 0 & -1 & 0  \\
 		1 & 0 & -\ka(u)\\
 		0 & \ka(u) & 0
 		\end{pmatrix}=K(u)\abs{\ga'(u)}\begin{pmatrix}
 				0   & -\sin \rho(u) & 0 \\
 				 \sin \rho(u)  &  0 & -\cos \rho(u) \\
 				 0 & \cos\rho(u) & 0   
 			\end{pmatrix},
 	\end{equation}
 	hence $|\Phi'(u)|=\abs{\La(u)}=\sqrt{2}\?K(u)\abs{\ga'(u)}$. Therefore, $\ga$ is parametrized by curvature if and only if
\begin{equation*}
\quad	K(u)\abs{\ga'(u)}=1\text{\, for a.e.~$u\in [0,T]$.}
\end{equation*}
Integrating we deduce that this is equivalent to 
\begin{equation*}
\quad	\tot(\ga|_{[0,u]})=u\text{ for every $u\in [0,T]$,}
\end{equation*} 
which proves (a). The expression for $\te{\La}$ is obtained from \eqref{E:tot0}, using that under the isomorphism $\widetilde{\mathfrak{so}}_3\to \mathfrak{so}_3$ induced by the projection $\Ss^3\to \SO_3$, $\tfrac{\mbf{i}}{2}$, $\tfrac{\mbf{j}}{2}$ and $\tfrac{\mbf{k}}{2}$ correspond respectively to
\begin{equation*}
	\begin{pmatrix}
	 0 & 0 & 0 \\
	 0 & 0 & -1 \\
	 0 & 1 & 0 
	\end{pmatrix},\quad 	\begin{pmatrix}
	 0 & 0 & 1 \\
	 0 & 0 & 0 \\
	 -1 & 0 & 0 
	\end{pmatrix},\text{\quad and \quad}	\begin{pmatrix}
	 0 & -1 & 0 \\
	 1 & 0 & 0 \\
	 0 & 0 & 0 
	\end{pmatrix}.\qedhere
\end{equation*} 
 \end{proof}
 
We now introduce the essential notion of grafting.
 \begin{defn}\label{D:graft} Let $\ga_i\colon [0,T_i]\to\Ss^2$ ($i=0,1$) be admissible curves parametrized by curvature.
 \begin{enumerate}
 	\item [(a)] A \tdef{grafting function} is a function $\phi\colon [0,s_0]\to [0,s_1]$ of the form
 	\begin{equation}\label{E:countable}
 		\phi(t)=t+\sum_{x<t,\,x\in X^+}\de^+(x)+\sum_{x\leq t,\,x\in X^-}\de^-(x),
 	\end{equation}
 	where $X^+\subs [0,s_0)$ and $X^-\subs [0,s_0]$ are countable sets and $\de^{\pm}\colon X^{\pm}\to (0,+\infty)$ are arbitrary  functions. 
 	\item [(b)] We say that $\ga_1$ is \tdef{obtained from $\ga_0$ by grafting}, denoted $\ga_0\gr \ga_1$, if there exists a grafting function $\phi\colon [0,T_0]\to [0,T_1]$ such that $\La_{\ga_0}=\La_{\ga_1}\circ \phi$. 
 	\item [(c)] Let $J$ be an interval (not necessarily closed). A \tdef{chain of grafts} consists of a homotopy $s\mapsto \ga_s$, $s\in J$, and a family of grafting functions $\phi_{s_0,s_1}\colon [0,s_0]\to [0,s_1]$, $s_0<s_1\in J$, such that:
 	\begin{enumerate}
 		\item [(i)] $\La_{\ga_{s_0}}=\La_{\ga_{s_1}}\circ \phi_{s_0,s_1}$ whenever $s_0<s_1$;
 		\item [(ii)] $\phi_{s_0,s_2}=\phi_{s_1,s_2}\circ \phi_{s_0,s_1}$ whenever $s_0<s_1<s_2$.
 	\end{enumerate}
 	Here every curve is admissible and parametrized by curvature. 
 \end{enumerate}
 \end{defn}
 
 \begin{rems}\label{R:grafting}\ 
 
 (a) A function $\phi\colon [0,s_0]\to [0,s_1]$, $s_0\leq s_1$, is a grafting function if and only if it is increasing and there exists a countable set $X\subs [0,s_0]$ such that $\phi(t)=t+c$ whenever $t$ belongs to one of the intervals which form $(0,s_0)\ssm X$, where $c\geq 0$ is a constant depending on the interval. 
 
 (b) Observe that in eq.~\eqref{E:countable}, $x<t$ in the first sum, while $x\leq t$ in the second sum. We do not require $X^+$ and $X^-$ to be disjoint, and they may be finite (or even empty). 
 		
 		(c) If $\phi\colon [0,s_0]\to [0,s_1]$ is a grafting function then it is monotone increasing and has derivative equal to 1 a.e.. Moreover, $\phi(t+h)-\phi(t)\geq h$ for any $t$ and $h\geq 0$; in particular, $s_0\leq s_1$.
 		
 		(d)  As the name suggests, $\ga_0\gr \ga_1$ if $\ga_1$ is obtained by inserting a countable number of pieces of curves (e.g., arcs of circles) at chosen points of $\ga_0$ (see fig.~\ref{F:enxerto}). This can be used, for instance, to increase the total curvature of a curve. The difficulty is that it is usually not clear how we can graft pieces of curves onto a closed curve so that the resulting curve is still closed and the restrictions on the geodesic curvature are not violated.
 
 (e) Two curves $\ga_0,\ga_1\in \sr L_{\ka_1}^{\ka_2}(Q)$ agree if and only if $\La_{\ga_0}=\La_{\ga_1}$ a.e.~on $[0,1]$. Indeed, $\ga_i=\Phi_{\ga_i}e_1$, where $\Phi_{\ga_i}$ is the unique solution to an initial value problem as in eq.~\eqref{E:ivp} of \S 1. Of course, if the curves are parametrized by curvature instead, then the latter condition should be replaced by $T_0=T_1$ and $\La_{\ga_0}=\La_{\ga_1}$ a.e.~on $[0,T_0]=[0,T_1]$. 
 
 \end{rems}

 For a grafting function $\phi\colon [0,s_0]\to [0,s_1]$ and $t\in [0,s_0]$, define:
 \begin{equation*}
 	\om^+(t)=\lim_{h\to 0^+}\phi(t+h)-\phi(t),\qquad \om^-(t)=\lim_{h\to 0^+}\phi(t)-\phi(t-h).
 \end{equation*}	
We also adopt the convention that $\om^+(s_0)=0$, while $\om^-(0)=\phi(0)$. Note that the limits above exist because $\phi$ is increasing. 
 \begin{lemma}\label{L:grafting}Let $\phi\colon [0,s_0]\to [0,s_1]$ be a grafting function, and let $X^{\pm}$ and $\de^{\pm}$ be as in definition (\ref{D:graft}\?(a)).
 	\begin{enumerate}
 		\item [(a)] $t\in X^{\pm}$ if and only if $\om^{\pm}(t)>0$. In this case, $\de^{\pm}(t)=\om^{\pm}(t)$.
 		\item [(b)]  $X^{\pm}$ and $\de^{\pm}$ are uniquely determined by $\phi$.
 		\item [(c)] If $\phi_0\colon [0,s_0]\to [0,s_1]$ and $\phi_1\colon [0,s_1]\to [0,s_2]$ are grafting functions then so is $\phi=\phi_1\circ \phi_0$. Moreover,
 		\begin{equation*}
 			X_0^{\pm}\subs X^{\pm}\text{\quad and \quad }\de_0^{\pm}\leq \de^\pm.
 		\end{equation*}
 		(Here $\de_0^\pm$ correspond to $\phi_0$, $\de^{\pm}$ correspond to $\phi$, and so forth.)
 	\end{enumerate}
 \end{lemma}
 
 \begin{proof}The proof will be split into parts.
 	\begin{enumerate}
 		\item [(a)]  Firstly, $\om^+(s_0)=0$ by convention and $s_0\nin X^+$ because $X^+\subs [0,s_0)$. Secondly, $\om^-(0)=\phi(0)$ by convention, and \eqref{E:countable} tells us that $0\in X^-$ if and only if $\phi(0)\neq 0$, in which case $\de^-(0)=\phi(0)$. This proves the assertion for $t=0$ (resp.~$t=s_0$) and $X^-$ (resp.~$X^+$).
 		
 		Since 
 		\begin{equation*}
 			\sum_{x\in X^+}\de^+(x)+\sum_{x\in X^-}\de^-(x)\leq s_1-s_0,
 		\end{equation*}
 		given $\eps>0$ there exist finite subsets $F^{\pm}\subs X^{\pm}$ such that
 		\begin{equation*}
 			\sum_{x\in X^+\ssm F^+}\de^+(x)+\sum_{x\in X^-\ssm F^-}\de^-(x)<\eps.
 		\end{equation*}
 		Suppose $t\nin X^+$, $t<s_0$. Then there exists $\eta$, $0<\eta<\eps$, such that $[t,t+\eta]\cap F^+=\emptyset$ and $[t,t+\eta]\cap F^-$ is either empty or $\se{x}$. In any case,
 		\begin{equation*}
 			\om^+(t)\leq \phi(t+\eta)-\phi(t)<\eta+\eps<2\eps,
 		\end{equation*}
		which proves that $\om^+(t)=0$. 
		
		Conversely, suppose that $t\in X^+$. Then clearly $\om^+(t)\geq \de^+(t)$. Moreover, an argument entirely similar to the one above shows that $\om^+(t)\leq \de^+(t)+2\eps$ for any $\eps>0$, hence $\om^+(t)=\de^+(t)>0$. The results for $X^-$ (and $t>0$) follow by symmetry.

 		\item [(b)] Since $\om^{\pm}$ are determined by $\phi$, the same must be true of $X^{\pm}$ and $\de^{\pm}$, by part (a). The converse is an obvious consequence of the definition of grafting function in \eqref{E:countable}.
	\item [(c)] 
	
	Let $\phi_1$,\,$\phi_0$ be as in the statement and set $X_i=X_i^-\cup X_i^+$, $i=0,1$, and $X=X_0\cup \phi_0^{-1}(X_1)$. Then $X$ is countable since both $X_0$ and $X_1$ are countable and $\phi_0$ is injective. Moreover, if $(a,b)\subs (0,s_0)\ssm X$ then 
	\begin{equation*}
\quad		\phi_1(\phi_0(t))=\phi_1(t+c_0)=t+c_0+c_1 \quad (t\in (a,b))
	\end{equation*}
	for some constants $c_0,c_1\geq 0$. In addition, $\phi_1\circ \phi_0$ is increasing, as $\phi_1$ and $\phi_0$ are both increasing. Thus, $\phi_1\circ \phi_0$ is a grafting function by (\ref{R:grafting}\?(e)).
	
		For the second assertion, let $x\in X_0^+$ and $h>0$ be arbitrary. Then 
 		\begin{equation*}
 			\quad\phi_1(\phi_0(x+h))-\phi_1(\phi_0(x))\geq \phi_0(x+h)-\phi_0(x)\geq \om_0^+(x),
 		\end{equation*}
 		hence $\om^+(x)\geq \om^+_0(x)>0$. Similarly, if $x\in X_0^-$ then $\om^-(x)\geq \om^-_0(x)>0$.  Therefore, it follows from part (a) that $X_0^{\pm}\subs X^{\pm}$ and $\de^{\pm}_0\leq \de^{\pm}$. \qedhere
 	\end{enumerate}
 \end{proof}
 
  \begin{lemma}\label{L:eqrel}
 	The grafting relation $\gr$ is a partial order over $\sr L_{\ka_1}^{\ka_2}(Q)$.
 \end{lemma}
 \begin{proof}
 Suppose $\ga_0,\ga_1$ are as in (\ref{D:graft}), with $\ga_0\gr \ga_1$ and $\ga_1\gr\ga_0$. Let $\phi_0\colon [0,T_0]\to [0,T_1]$ and $\phi_1\colon [0,T_1]\to [0,T_0]$ be the corresponding grafting functions. By (\ref{R:grafting}\?(d)), the existence of such functions implies that $T_0=T_1$, which, in turn, implies that $\phi_0(t)=t=\phi_1(t)$ for all $t$. Hence $\La_{\ga_0}=\La_{\ga_1}\circ \phi_0=\La_{\ga_1}$, and it follows that $\ga_0=\ga_1$. This proves that $\gr$ is antisymmetric.
 
 Now suppose $\ga_0\gr \ga_1$, $\ga_1\gr \ga_2$ and let $\phi_i\colon [0,T_i]\to [0,T_{i+1}]$ be the corresponding grafting functions, $i=0,1$. By (\ref{L:grafting}\?(c)), $\phi=\phi_1\circ \phi_0$ is also a grafting function. Furthermore,
 \begin{equation*}
\qquad 	\La_{\ga_0}=\La_{\ga_1}\circ \phi_0=(\La_{\ga_2}\circ \phi_1)\circ \phi_0=\La_{\ga_2}\circ \phi
 \end{equation*} by hypothesis, so $\ga_0\gr \ga_2$, proving that $\gr$ is transitive.
 
 Finally, it is clear that $\gr$ is reflexive.
 \end{proof}
 
 \begin{lemma}\label{L:chainofgrafts}
 	Let $\Ga=(\ga_s)_{s\in [a,b)}$, $\ga_s\in \sr L_{\ka_1}^{\ka_2}(Q)$, be a chain of grafts.  Then there exists a unique extension of $\Ga$ to a chain of grafts on $[a,b]$. 
 \end{lemma}
 \begin{proof}
 	For $s_0<s_1\in [a,b]$, let $\phi_{s_0,s_1}\colon [0,s_0]\to [0,s_1]$ be the grafting function corresponding to $\ga_{s_0}\gr \ga_{s_1}$ and similarly for $X^\pm_{s_0,s_1}$, $\de^{\pm}_{s_0,s_1}$, $\om^{\pm}_{s_0,s_1}$.
 	
	Suppose $s_0<s_1<s_2$. By hypothesis, $\phi_{s_0,s_2}=\phi_{s_1,s_2}\circ \phi_{s_0,s_1}$. Therefore, by (\ref{L:grafting}\?(c)),  
 		\begin{equation}\label{E:Xs0}
\qquad 	X_{s_0,s_1}^{\pm}\subs X^{\pm}_{s_0,s_2}\text{\quad and\quad }\de^{\pm}_{s_0,s_1}\leq \de^{\pm}_{s_0,s_2} \text{\quad ($s_0<s_1<s_2$)}.
 	\end{equation}	
 	Fix $s_0\in [a,b)$ and set
 	\begin{equation*}
 		X^{\pm}_{s_0,b}=\bcup_{s_0<s<b}X^{\pm}_{s_0,s}\text{\quad and\quad}\de^{\pm}_{s_0,b}=\sup_{s_0<s<b}\se{\de^{\pm}_{s_0,s}}.
 	\end{equation*}
 	Since $\big(X^{\pm}_{s_0,s}\big)$ is an increasing family of countable sets, $X^{\pm}_{s_0,b}$ must also be countable. Define $\phi_{s_0,b}\colon [0,s_0]\to [0,b]$ by 
 	\begin{equation*}
 		\phi_{s_0,b}(t)=t+\sum_{x<t,\,x\in X_{s_0,b}^+}\de_{s_0,b}^+(x)+\sum_{x\leq t,\,x\in X_{s_0,b}^-}\de_{s_0,b}^-(x).
 	\end{equation*} 
  Then $\phi_{s_0,b}$ is a grafting function for any $s_0$ by construction, and for $s_0<s_1$ we have 
 \begin{equation*}
 	\phi_{s_0,b}=\lim_{s\to b-}\phi_{s_0,s}=\lim_{s\to b-}\phi_{s_1,s}\circ \phi_{s_0,s_1}=\phi_{s_1,b}\circ \phi_{s_0,s_1}.
 \end{equation*}

Before defining the curve $\ga_b$, we construct its logarithmic derivative $\La$. For each $s<b$, let 
 	\begin{equation*}
 		E_s=\phi_{s,b}\big([0,s]\big),\quad E=\bcup_{s<b}E_s.\quad
 	\end{equation*}
 	Then $\mu(E_s)=s$ for all $s$, hence $[0,b]\ssm E$ has measure zero, which implies that $E$ is measurable and $\mu(E)=b$. (Here $\mu$ denotes Lebesgue measure.) For $u\in E$, $u=\phi_{s,b}(t)$ for some $t\in [0,s]$ and $s\in [a,b)$, set
 	\begin{equation}\label{E:uniquegraft}
\qquad	\La(u)=\La(\phi_{s,b}(t))=\La_{s}(t)\qquad  (u\in E),
 	\end{equation}
	where $\La_s$ denotes the logarithmic derivative of $\ga_s$. Observe that $\La$ is well-defined, for if $\phi_{s_0,b}(t_0)=u=\phi_{s_1,b}(t_1)$, with $s_0<s_1$, then
 	\begin{equation*}
		\phi_{s_1,b}(t_1)=\phi_{s_0,b}(t_0)=\phi_{s_1,b}\circ \phi_{s_0,s_1}(t_0),
 	\end{equation*}
 	hence $t_1=\phi_{s_0,s_1}(t_0)$ (because $\phi_{s_0,s_1}$ is increasing) and thus
 	\begin{equation*}
 		\La_{{s_1}}(t_1)=\La_{{s_1}}(\phi_{s_0,s_1}(t_0))=\La_{{s_0}}(t_0).
 	\end{equation*}
 	Moreover, by (\ref{L:parametrizedbycurvature}),
 	\begin{equation*}
 		\La(u)=\begin{pmatrix}
 				0   & -\sin \rho(u) & 0 \\
 				 \sin \rho(u)  &  0 & -\cos \rho(u) \\
 				 0 & \cos\rho(u) & 0   
 			\end{pmatrix}
 	\end{equation*}
 	where $\rho(u)=\rho_{{s_0}}(t)$ if $u=\phi_{s_0,b}(t)$. The measurability of $\rho$ follows from that of each $\rho_{s}$. Thus, the entries of $\La$ belong to $L^2[0,b]$ and the initial value problem $\dot\Phi=\Phi\?\La$, $\Phi(0)=I$, has a unique solution $\Phi\colon [0,b]\to \SO_3$. Naturally, we define $\ga_b(t)=\Phi (t)e_1$. 
 	
 	Let $X_{s,b}=X^+_{s,b}\cup X^-_{s,b}$ and suppose that $(\al,\be)$ is one of the intervals which form $(0,s)\ssm X_{s,b}$. Then $\phi_{s,b}(\al,\be)\subs E_s\subs [0,b]$ is an interval of measure $\be-\al$; we have $\La(t)=\La_s(t-c)$ for $t\in \phi_{s,b}(\al,\be)$ and a constant $c\geq 0$, so that the restriction of $\ga_b$ to this interval is just $\ga_s|[\al,\be]$ composed with a rotation of $\Ss^2$. In particular, we deduce that the geodesic curvature $\ka$ of $\ga_b$ satisfies $\ka_1<\ka<\ka_2$ a.e.~on $\phi_s(\al,\be)$. Since $\lim_{s\to b}\mu(E_s)=b$, this argument shows that $\ka_1<\ka<\ka_2$ a.e.~on $[0,b]$. We claim also that $\Phi(b)=Q$. To see this, let $\bar{\La}_s\colon [0,b]\to \mathfrak{so}_3$ be the extension of $\La_{s}$ by zero to all of $[0,b]$. If $\bar\Phi_s$ is the solution to the initial value problem $\dot{\bar\Phi}_s=\bar\Phi_s\bar\La_s$, $\bar\Phi_s(0)=I$, we have $\Phi_s(b)=\Phi_{s}(s)=Q$. Since $\bar\La_s$ converges to $\La$ in the $L^2$-norm, it follows from continuous dependence on the parameters of a differential equation that 
 	\begin{equation*}
 		\abs{\Phi(b)-Q}=\lim_{s\to b}\abs{\Phi(b)-\Phi_s(b)}=0.
 	\end{equation*}
 	
 	The curve $\ga_b$ satisfies $\ga_s\gr \ga_b$ for any $s\leq b$ by construction. Conversely, if this condition is satisfied then \eqref{E:uniquegraft} must hold, showing that $\ga_b$ is the unique curve with this property. This completes the proof.
 \end{proof}
 
 \subsection*{Adding loops}This subsection presents adaptations of a few concepts and results contained in \S5 of \cite{Sal3}. Let $\ka_0\in \R$, $\rho_0=\arccot \ka_0$ and $Q\in \SO_3$ be fixed throughout the discussion.
 
 For arbitrary $\rho_1\in (0,\rho_0)$, define $\sig^{\rho_1}$ to be the unique circle in $\sr L_{\ka_0}^{+\infty}(I)$ of radius of curvature $\rho_1$:
 \begin{equation*}
 	\sig^{\rho_1}(t)=\cos\rho_1(\cos\rho_1,0,\sin\rho_1)+\sin\rho_1\big(\sin\rho_1\cos(2\pi t),\sin(2\pi t),-\cos\rho_1\cos(2\pi t)\big),
 \end{equation*} 
 and let $\sig^{\rho_1}_n\in \sr L_{\ka_0}^{+\infty}(I)$ be $\sig^{\rho_1}$ traversed $n$ times; in symbols, $\sig^{\rho_1}_n(t)=\sig^{\rho_1}(nt)$, $t\in [0,1]$. As we have seen in (\ref{L:homocircles}), if $\rho_1,\rho_2<\rho_0$ then $\sig^{\rho_1}$ and $\sig^{\rho_2}$ are homotopic within $\sr L_{\ka_0}^{+\infty}(I)$. 
 
 Now let $\ga\in \sr L_{\ka_0}^{+\infty}(Q)$, $n\in \N$, $\eps>0$ be small and $t_0\in (0,1)$. Let $\ga^{[t_0\#n]}$ be the curve obtained by inserting (a suitable rotation of) $\sig^{\rho_1}_{n}$ at $\ga(t_0)$, as depicted in fig.~\ref{F:adding}. More explicitly,
 \begin{equation*}
 	\ga^{[t_0\#n]}(t)=\begin{cases}
 		\ga(t)   &   \text{ if }0\leq t\leq t_0-2\eps \\
 		 \ga(2t-t_0+2\eps)  &   \text{ if }t_0-2\eps\leq t\leq t_0-\eps \\
 		 \Phi_\ga(t_0)\sig_n^{\rho_1}\Big(\frac{t-t_0+\eps}{2\eps}\Big) &  \text{ if }t_0-\eps\leq t\leq t_0+\eps \\
 		 \ga(2t-t_0-2\eps)  &   \text{ if }t_0+\eps\leq t\leq t_0+2\eps \\
  		\ga(t)   &   \text{ if }t_0+2\eps\leq t\leq 1
  		 \end{cases}
 \end{equation*}
 
 \begin{figure}[ht]
 	\begin{center}
 		\includegraphics[scale=.45]{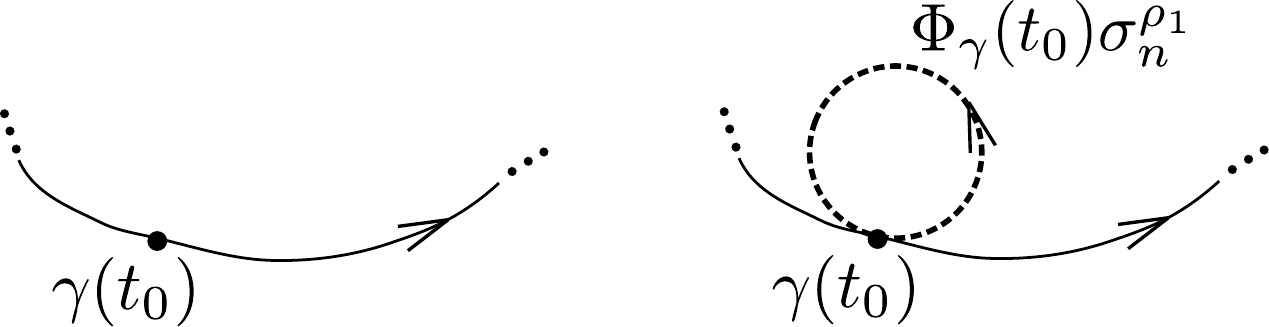}
 		\caption{A curve $\ga\in \sr L_{\ka_0}^{+\infty}(Q)$ and the curve $\ga^{[t_0\#n]}$ obtained from $\ga$ by adding loops at $\ga(t_0)$.}
 		\label{F:adding}
 	\end{center}
 \end{figure}
 
 The precise values of $\eps$ and $\rho_1$ are not important, in the sense that different values of both parameters yield curves that are homotopic. For $t_0\neq t_1\in (0,1)$ and $n_0,n_1\in \N$, the curve $\big(\ga^{[t_0\# n_0]}\big)^{[t_1\#n_1]}$ will be denoted by $\ga^{[t_0\#n_0;t_1\#n_1]}$. 
 
We shall now explain how to spread loops along a curve, as in fig.~\ref{F:adding2}; to do this, a special parametrization is necessary. Given $\ga\in \sr L_{-\infty}^{+\infty}(Q)$, let $\La_\ga=(\Phi_\ga)^{-1}\dot \Phi_\ga\colon [0,1]\to \mathfrak{so}_3$ denote its logarithmic derivative. Since the entries of $\La_\ga$ are $L^2$ functions and $[0,1]$ is bounded, 
 	\begin{equation}\label{E:bycurvature2}
 		M=\int_0^1\abs{\La_\ga(t)}\?dt<+\infty.
 	\end{equation}
 Define a function $\tau\colon [0,1]\to [0,1]$ by 
 	\begin{equation*}
 		\tau(t)=\frac{1}{M}\int_0^t \abs{\La_\ga(u)}\?du.
 	\end{equation*}
 	Then $\tau$ is a monotone increasing function, hence it admits an inverse. If we reparametrize $\ga$ by $\tau\mapsto \ga(t(\tau))$, $\tau\in [0,1]$, then its logarithmic derivative with respect to $\tau$ satisfies
 	\begin{equation*}
 		\abs{\La_\ga(\tau)}=\vert\dot\Phi_\ga(t(\tau))\vert\?\dot t(\tau)=\abs{\La_\ga(t(\tau))}\frac{M}{\abs{\La_\ga(t(\tau))}}=M.\footnote{The parameter $\tau$ is a multiple of the curvature parameter considered in (\ref{D:bycurvature}).}
 	\end{equation*}
 	Therefore, using (\ref{L:reparametrize}), we may assume at the outset that all curves $\ga\in \sr L_{-\infty}^{+\infty}(Q)$ are parametrized so that $\vert\dot \Phi_\ga\vert=\abs{\La_\ga}$ is constant (and finite). With this assumption in force, let $n\in \N$, $\rho_1\in (0,\pi)$ and define a map $F_n\colon \sr L_{-\infty}^{+\infty}(Q)\to \sr L_{-\infty}^{+\infty}(Q)$ by:
 \begin{equation}\label{E:F_n}
 \qquad	F_n(\ga)(t)=\Phi_\ga(t)\sig^{\rho_1}_n(t)\quad (\ga\in \sr L_{-\infty}^{+\infty}(Q),~t\in [0,1]).
  \end{equation} 
  
  \begin{figure}[ht]
 	\begin{center}
 		\includegraphics[scale=.48]{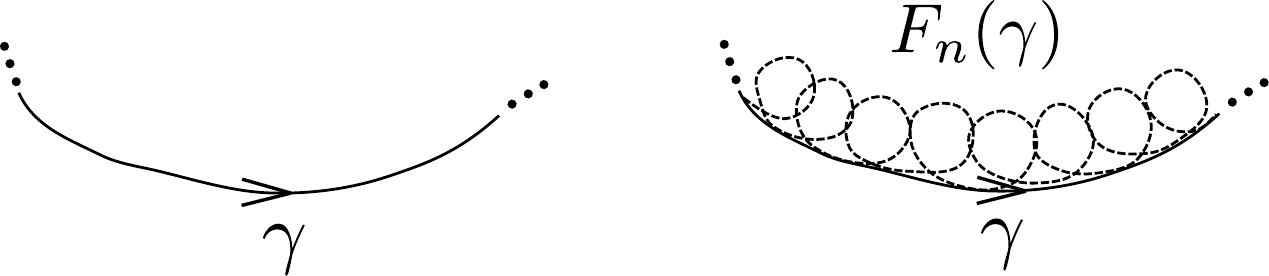}
 		\caption{A curve $\ga\in \sr L_{\ka_0}^{+\infty}(Q)$ and a ``phone wire'' approximation $F_n(\ga)$.}
 		\label{F:adding2}
 	\end{center}
 \end{figure}
 
Using that $\dot\Phi_\ga=\Phi_\ga \La_\ga$ (where $\dot{\phantom{a}}$ denotes differentiation with respect to $t$), we find that
 \begin{equation}\label{E:velo}
 	\dot F_n(\ga)=\Phi_\ga\big(\La_\ga\sig_n^{\rho_1}+\dot\sig_n^{\rho_1}\big),
 \end{equation}
and this allows us to conclude that  $\Phi_{F_n(\ga)}(0)=\Phi_\ga(0)$ and $\Phi_{F_n(\ga)}(1)=\Phi_\ga(1)$ for any admissible curve $\ga$, so that $F_n$ does indeed map $\sr L_{-\infty}^{+\infty}(Q)$ to itself.  Moreover, $F_n(\ga)$ is never homotopic to $F_{m}(\ga)$ when $m\not\equiv n \pmod{2}$. This is because the two curves have different final lifted frames: $\te{\Phi}_{F_n(\ga)}(1)=(-1)^{n-m}\te{\Phi}_{F_m(\ga)}(1)$ in $\Ss^3$. 

 \begin{lemma}\label{L:Sal1}
 	Let $\ka_0=\cot \rho_0\in \R$, $Q\in \SO_3$, $\rho_1\in (0,\rho_0)$, $K$ be compact and $f\colon K\to \sr L_{-\infty}^{+\infty}(Q)$ be continuous. Then the image of  $F_n\circ f$ is contained in $\sr L_{\ka_0}^{+\infty}(Q)$ for all sufficiently large $n$.
 \end{lemma}
 \begin{proof}
 	In order to simplify the notation, we will prove the lemma when $K$ consists of a single point. The proof still works in the more general case because all that we need is a uniform bound on $\vert\La_{f(a)}\vert$ for $a\in K$. Denoting $\sig_1^{\rho_1}$ simply by $\sig$, we may rewrite \eqref{E:velo} as:
 	\begin{equation}\label{E:velo1}
 	\quad	\dot F_n(\ga)(t)=n\?\Phi_\ga(t)\big(\dot\sig(nt)+O(\tfrac{1}{n})\big)\quad (t\in [0,1]),
 	\end{equation}
 	where $O(\frac{1}{n})$ denotes a term such that $n\abs{O(\tfrac{1}{n})}$ is uniformly bounded over $[0,1]$ as $n$ ranges over all of $\N$. (In this case, $n\abs{O(\tfrac{1}{n})}=\abs{\La_\ga(t)}=M$ for all $t\in [0,1]$, with $M$ as in \eqref{E:bycurvature2}.) Therefore,
 	\begin{equation}\label{E:nor}
 		F_n(\ga)(t)\times \frac{\dot F_n(\ga)(t)}{\vert\dot F_n(\ga)(t)\vert}=\Phi_\ga(t)\left(\sig(nt)\times \frac{\dot\sig(nt)}{\abs{\dot\sig(nt)}}\right)+O(\tfrac{1}{n}).
 	\end{equation}
 	Let $\Phi_{F_n(\ga)}$ (resp.~$\Phi_\sig$) denote the frame of $F_n(\ga)$ (resp.~$\sig$) and $\La_{F_n(\ga)}$ (resp.~$\La_\sig$) its logarithmic derivative. It follows from \eqref{E:F_n}, \eqref{E:velo1} and \eqref{E:nor} that
 	\begin{equation*}
 		\Phi_{F_n(\ga)}(t)=\Phi_\ga(t)\Phi_\sig(nt)+O(\tfrac{1}{n}).
 	\end{equation*}
 	Differentiating both sides of this equality, we obtain that
 	\begin{equation*}
 		\dot\Phi_{F_n(\ga)}(t)=\dot\Phi_\ga(t)\Phi_\sig(nt)+n\?\Phi_\ga(t)\dot\Phi_\sig(nt)+O(1)=n\big(\Phi_\ga(t)\dot\Phi_\sig(nt)+O(\tfrac{1}{n})\big).
 	\end{equation*}
 	Multiplying on the left by the inverse of $\Phi_{F_n(\ga)}$, we finally conclude that
 	\begin{equation}\label{E:loga}
 		\La_{F_n(\ga)}(t)=n\big(\La_\sig(nt)+O(\tfrac{1}{n})\big).
 	\end{equation}
 	Recall that, by the definition of logarithmic derivative (eq.~\eqref{E:frenet1}, \S 1), 
 	\begin{equation}\label{E:kappan1}
 		\La_{F_n(\ga)}=\begin{pmatrix}
 		0 & -\vert\dot F_n(\ga)\vert & 0 \\
 		\vert\dot F_n(\ga)\vert & 0 & -\vert\dot F_n(\ga)\vert\ka_{F_n(\ga)} \\
 		0 & \vert\dot F_n(\ga)\vert\ka_{F_n(\ga)} & 0
 		\end{pmatrix}\text{\  and \ }\La_\sig=\begin{pmatrix}
 		0 & -\abs{\dot\sig} & 0 \\
 		\abs{\dot\sig} & 0 & -\abs{\dot\sig}\ka_1 \\
 		0 & \abs{\dot\sig}\ka_1 & 0
 		\end{pmatrix},
 	\end{equation}
 	where $\ka_{F_n(\ga)}$ (resp.~$\ka_1=\cot\rho_1$) denotes the geodesic curvature of $F_n(\ga)$ (resp.~$\sig$). Comparing the (3,2)-entries of \eqref{E:loga} and \eqref{E:kappan1}, and using \eqref{E:velo1}, we deduce that 
 	\begin{equation*}
 		n\big(\abs{\dot\sig(nt)}+O(\tfrac{1}{n})\big)\ka_{F_n(\ga)}(t)=n\big(\abs{\dot\sig(nt)}\ka_1+O(\tfrac{1}{n})\big).
 	\end{equation*}
 	Therefore $\lim_{n\to +\infty} \ka_{F_n(\ga)}=\ka_1>\ka_0$ uniformly over $[0,1]$, as required.
 \end{proof}

 \begin{lemma}\label{L:Sal2}
 	Let $\ga\in \sr L_{\ka_0}^{+\infty}(Q)$, $t_0\in (0,1)$. Then $ \ga^{[t_0\#n]}\iso F_n(\ga)$ within $\sr L_{\ka_0}^{+\infty}(Q)$ for all sufficiently large $n\in \N$.
 \end{lemma}
 \begin{proof}
 	Informally, the homotopy is obtained by pushing the loops in $F_n(\ga)$ towards $\ga(t_0)$. If $n$ is large enough, then we can guarantee that the curvature remains greater than $\ka_0$ throughout the deformation; the proof is similar to that of (\ref{L:Sal1}), so we will omit it. See lemma 5.4 in \cite{Sal3} for the details when $\ka_0=0$. 
 \end{proof}
 The next result states that after we add enough loops to a curve, it becomes so flexible that any condition on the curvature may be safely forgotten.
This is yet another instance of the ``phone wire'' construction
already present in \cite{GromovPR}, \cite{EliashbergMishachev} and \cite{Levy};
we refer the reader to  \cite{MostovoySadykov} for a thorough discussion of this kind
of construction in terms of the h-principle.
 \begin{lemma}\label{L:Sal3}
 	Let $\ga_0,\ga_1\in \sr L_{\ka_0}^{+\infty}(Q)$ be two curves in the same component of $\sr I(Q)=\sr L_{-\infty}^{+\infty}(Q)$. Then $F_{n}(\ga_0)$ and $F_{n}(\ga_1)$ lie in the same component of $\sr L_{\ka_0}^{+\infty}(Q)$ for all sufficiently large $n\in \N$.
 \end{lemma}
 \begin{proof}
	Let $\ga_0,\,\ga_1$ be two curves in the same component of $\sr L_{-\infty}^{+\infty}(Q)$. Taking $K=[0,1]$ and $h\colon K\to \sr L_{-\infty}^{+\infty}(Q)$ to be a path joining $\ga_0$ and $\ga_1$, we conclude from (\ref{L:Sal1}) that $g=F_{n}\circ h$ is a path in $\sr L_{\ka_0}^{+\infty}(Q)$ joining both curves if $n$ is sufficiently large. 
 \end{proof}
 
Thus, if we can find a way to deform $\ga_i$ into $F_{2n}(\ga_i)$ for large $n$, $i=0,1$, then the question of deciding whether $\ga_0$ and $\ga_1$ are homotopic reduces to the easy verification of whether their final lifted frames $\te{\Phi}_{\ga_0}(1)$ and $\te{\Phi}_{\ga_1}(1)$ agree. One way to deform $\ga$ into $F_{2n}(\ga)$ is to graft arbitrarily long arcs of circles onto it; this is possible if $\ga$ diffuse (see fig.~\ref{F:enxerto} below).

 \subsection*{Grafting non-condensed curves}

 \begin{prop}\label{P:diffuseloose}
 	Let $\ka_0\in \R$ and suppose that $\ga\in \sr L_{\ka_0}^{+\infty}$ is diffuse. Then $\ga$ is homotopic to a circle traversed a number of times.
 \end{prop}
 \begin{proof}
 	Let $\ga\colon [0,T]\to \Ss^2$ be parametrized by curvature and let $\te{\La}\colon [0,T]\to \te{\mathfrak{so}}_3$ be its (lifted) logarithmic derivative. Since $\ga$ is diffuse, we can find $0<t_1<t_2<T$ and $\rho_1,\rho_2\in [0,\rho_0]$ such that $C_\ga(t_1,\rho_1)=-C_\ga(t_2,\rho_2)$. By deforming $\ga$ in a neighborhood of $\ga(t_2)$ if necessary, we can actually assume that $\rho_1,\rho_2\in (0,\rho_0)$. Set $z_i=\te{\Phi}(t_i)$,
 	\begin{equation*}
	\chi_i=C_\ga(t_i,\rho_i)=\cos \rho_i\?\ga(t_i)+\sin \rho_i\?\no(t_i)\quad\text{and}\quad	\la_i=\cos \rho_i\?\mbf{i}+\sin\rho_i\?\mbf{k}\quad (i=1,2).
 	\end{equation*}
 	Identifying $\Ss^2$ with the unit imaginary quaternions, we have
 	\begin{equation}\label{E:chila}
 		z_i\la_iz_i^{-1}=\chi_i\quad (i=1,2).
 	\end{equation}
 	We will define a family of curves $s\mapsto \ga_s$, $s\geq 0$, as follows: First, let $\te{\La}_s\colon [0,T+2s]\to \te{\frak{so}}_3$ be given by:
 	\begin{equation*}
 		\te{\La}_s(t)=\begin{cases}
 			\te{\La}(t)   &   \text{ if }\quad 0\leq t\leq t_1 \\
 			\frac{1}{2}\la_1   &   \text{ if }\quad t_1\leq t\leq t_1+s \\
 			\te{\La}(t-s)   &   \text{ if }\quad t_1+s\leq t\leq t_2+s \\
			\frac{1}{2}\la_2   &   \text{ if }\quad t_2+s\leq t\leq t_2+2s \\
 			\te{\La}(t-2s)   &   \text{ if }\quad t_2+2s\leq t\leq T+2s \\
 		\end{cases}
 	\end{equation*}
 	Next, let $\La_s\in \mathfrak{so}_3 $ correspond to $\te{\La}_s\in \widetilde{\mathfrak{so}}_3$  and define $\Phi_s$ to be the unique solution to the initial value problem $\Phi_s(0)=I$, $\dot{\Phi}_s=\Phi_s\La_s$. Finally, set $\ga_s=\Phi_se_1$.  Geometrically, when $s=2\pi k$, $\ga_s$ is obtained from $\ga$ by grafting a circle of radius $\rho_1$ traversed $k$ times at $\ga(t_1)$ and another circle of radius $\rho_2$ traversed $k$ times at $\ga(t_2)$ (see fig.~\ref{F:enxerto}).  We claim that $\ga_s\in \sr L_{\ka_0}^{+\infty}(I)$ for all $s\geq 0$.
 	
 	Indeed, we have 
 	\begin{equation*}
 		\te{\Phi}_s(t)=\begin{cases}
 			 \te{\Phi}(t)  &   \text{ if }\quad 0\leq t\leq t_0 \\
 			 z_1\exp\big(\frac{\la_1}{2}(t-t_1)\big) &   \text{ if }\quad t_1\leq t\leq t_1+s \\ 
 			 \exp\big(\tfrac{\chi_1}{2}s\big)\te{\Phi}(t-s) &   \text{ if }\quad t_1+s\leq t\leq t_2+s \\ 
 			 \exp\big(\tfrac{\chi_1}{2}s\big)z_2\exp\big(\tfrac{\la_1}{2}(t-t_2-s)\big)  &   \text{ if }\quad t_2+s\leq t\leq t_2+2s \\ 
 			 \exp\big(\tfrac{\chi_1}{2}s\big)\exp\big(\tfrac{\chi_2}{2}s\big)\te{\Phi}(t-2s)  &   \text{ if }\quad t_2+2s\leq t\leq T+2s 
 		\end{cases}
 	\end{equation*}
 	where we have used \eqref{E:chila} to write
 	\begin{equation*}
 		\big(z_1\exp\big(\tfrac{s\la_1}{2}\big)\big)\big(z_1^{-1}\te{\Phi}(t-s)\big)=\exp\big(\tfrac{s\chi_1}{2}\big)\te{\Phi}(t-s),
 	\end{equation*}
 	which yields the expression for $\te{\Phi}(t)$ when $t\in [t_1, t_1+s]$, and similarly for the interval $[t_2+2s,T+2s]$. In particular, we deduce that the final lifted frame is:
 	\begin{equation*}
 		\te{\Phi}_s(T+2s)= \exp\big(\tfrac{s\chi_1}{2}\big)\exp\big(\tfrac{s\chi_2}{2}\big)\te{\Phi}(T)=\te{\Phi}(T),
 	\end{equation*}
 	as $\chi_2=-\chi_1$ by hypothesis.
 	This proves that each $\ga_s$ has the correct final frame. The curvature $\ka^s$ of $\ga_s$ clearly satisfies $\ka^s>\ka_0$ almost everywhere in $[0,t_1]\cup [t_1+s,t_2+s]\cup [t_2+2s,T+2s]$, because, by construction, the restriction of  $\ga_s$ to each of these intervals is the composition of a rotation of $\Ss^2$ with an arc of $\ga$. Moreover, the restriction of $\ga_s$ to the interval $[t_1,t_1+s]$ is an arc of circle of radius of curvature $\rho_1<\rho_0$; similarly, the restriction of $\ga_s$ to $[t_2+s,t_2+2s]$ is an arc of circle of radius of curvature $\rho_2<\rho_0$. Therefore $\ka^s>\ka_0$ almost everywhere on $[0,T+2s]$, and we conclude that $\ga_s\in \sr L_{\ka_0}^{+\infty}(I)$.
 	
 	We have thus proved that $\ga$ is homotopic to $\ga^{[t_0\#n;t_1\#n]}$ for all $n\in \N$ when $\ga$ is diffuse. The proposition now follows from (\ref{L:Sal2}) and (\ref{L:Sal3}) combined.  
 \end{proof}
 
 \begin{figure}[ht]
	\begin{center}
		\includegraphics[scale=.40]{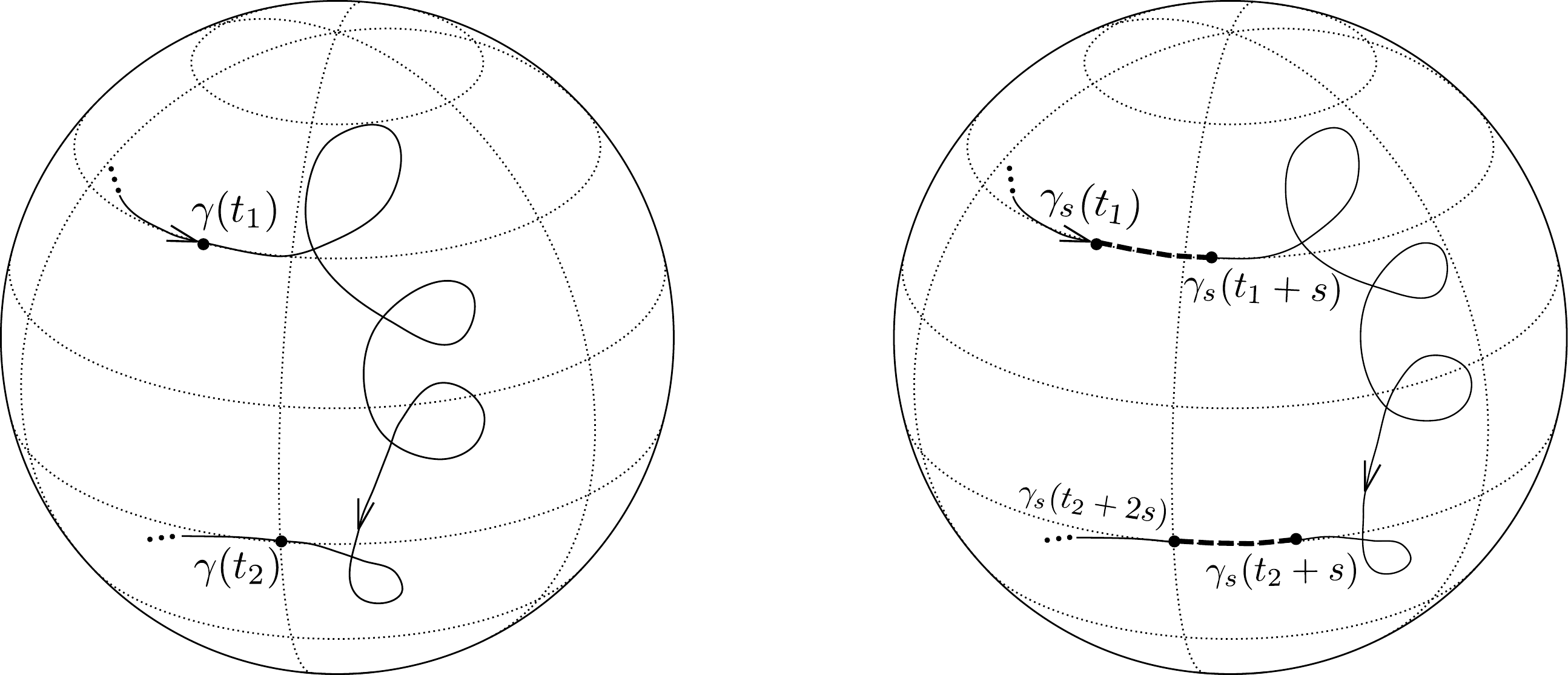}
		\caption{Grafting arcs of circles onto a diffuse curve, as described in (\ref{P:diffuseloose}).}
		\label{F:enxerto}
	\end{center}
\end{figure}

The next result says that we can still graft small arcs of circle onto $\ga$ even when it is not diffuse, as long as it is also not condensed.
 \begin{prop}\label{P:totting}
 	Suppose that $\ga\in \sr L_{\ka_0}^{+\infty}(I)$ is non-condensed. Then there exist $\eps>0$ and a chain of grafts $(\ga_s)$ such that $\ga_0=\ga$, $\ga_s\in \sr L_{\ka_0}^{+\infty}(I)$ and $\tot(\ga_s)=\tot(\ga)+s$ for all $s\in [0,\eps)$.
 \end{prop}
 \begin{proof}(In this proof the identification of $\Ss^2$ with the set of unit imaginary quaternions used in (\ref{P:diffuseloose}) is still in force.)
 	Let $\ga\colon [0,T]\to \Ss^2$ be parametrized by curvature and let $\te{\La}\colon [0,T]\to \te{\mathfrak{so}}_3$ be its (lifted) logarithmic derivative. Since $\ga$ is not condensed, $0$ lies in the interior of the convex closure of the image $C$ of $C_\ga$ by (\ref{L:closedhemisphere}). Hence, by (\ref{L:Steinitz}), we can find a 3-dimensional simplex with vertices in $C$ containing $0$ in its interior. In symbols, we can find $0< t_1<t_2<t_3<t_4<T$  and $s_1,s_2,s_3,s_4>0$, $s_1+s_2+s_3+s_4=1$, such that
 	\begin{equation}\label{E:concomb}
 		0=s_1\chi_1+s_2\chi_2+s_3\chi_3+s_4\chi_4,
 	\end{equation} 
 	where $\chi_i=C_\ga(t_i,\rho_i)$, for some $\rho_i\in (0,\rho_0)$, and the $\chi_i$ are in general position. Furthermore, these numbers $s_i$ are the only ones which have these properties (for this choice of the $\chi_i$). Define a function $G\colon \R^4\to \Ss^3$ by 
 	\begin{equation*}
 		G(\sig_1,\sig_2,\sig_3,\sig_4)=\exp\Big(\frac{\sig_1\chi_1}{2}\Big)\exp\Big(\frac{\sig_2\chi_2}{2}\Big)\exp\Big(\frac{\sig_3\chi_3}{2}\Big)\exp\Big(\frac{\sig_4\chi_4}{2}\Big).
 	\end{equation*}
 	Then $G(0,0,0,0)=\tbf{1}$ and 
 	\begin{equation*}
 		DG_{(0,0,0,0)}(a,b,c,d)=\frac{1}{2}\big(a\chi_1+b\chi_2+c\chi_3+d\chi_4\big).
 	\end{equation*}
 	Since the $\chi_i$ are in general position by hypothesis, we can invoke the implicit function theorem to find some $\de>0$ and, without loss of generality, functions $\bar{\sig}_2$,\?$\bar\sig_3$,\?$\bar\sig_4\colon (-\de,\de)\to \R$ of $\sig_1$ such that 
 	\begin{equation*}
\qquad 		G\big(\sig_1,\bar\sig_2(\sig_1),\bar\sig_3(\sig_1),\bar\sig_4(\sig_1)\big)=\tbf{1} \quad (\sig_1\in (-\de,\de)).
 	\end{equation*}

	Differentiating the previous equality with respect to $\sig_1$ at $0$ and comparing \eqref{E:concomb} we deduce that
 	\begin{equation*}
\qquad 		\bar\sig_i'(0)=\frac{s_i}{2s_1}>0\quad (i=2,3,4).
 	\end{equation*}
 	Let $s(\sig_1)=\sig_1+\bar\sig_2(\sig_1)+\bar\sig_3(\sig_1)+\bar\sig_4(\sig_1)$. Then $s'(\sig_1)>0$, hence we can write $\sig_1$, $\sig_2$, $\sig_3$ and $\sig_4$ as a function of $s$ in a neighborhood of $0$. The conclusion is thus that there exist $\eps>0$ and non-negative functions $\sig_1,\sig_2,\sig_3,\sig_4$ of $s$ such that $\sig_1(s)+\sig_2(s)+\sig_3(s)+\sig_4(s)=s$ and
 	\begin{equation*}
 		\exp\Big(\frac{\sig_1\chi_1}{2}\Big)\exp\Big(\frac{\sig_2\chi_2}{2}\Big)\exp\Big(\frac{\sig_3\chi_3}{2}\Big)\exp\Big(\frac{\sig_4\chi_4}{2}\Big)=\tbf{1}\text{\quad for all $s\in [0,+\eps)$}.
 	\end{equation*}
 	We will now use these functions to obtain $\ga_s$, $s\in [0,+\eps)$. 
 	
 	Define $\te{\La}_s\colon [0,T+s]\to \te{\mathfrak{so}_3}$ by:
 	\begin{equation*}
 		\te{\La}_s(t)=\begin{cases}
 			\te{\La}(t)   &   \text{ if }\quad 0\leq t\leq t_1 \\
 			\frac{1}{2}\la_1   &   \text{ if }\quad t_1\leq t\leq t_1+\sig_1 \\
 			\te{\La}(t-\sig_1)   &   \text{ if }\quad t_1+\sig_1\leq t\leq t_2+\sig_1 \\
			\frac{1}{2}\la_2   &   \text{ if }\quad t_2+\sig_1\leq t\leq t_2+\sig_1+\sig_2 \\
 			\te{\La}(t-\sig_1-\sig_2)   &   \text{ if }\quad t_2+\sig_1+\sig_2\leq t\leq t_3+\sig_1+\sig_2 \\
 			\frac{1}{2}\la_3   &   \text{ if }\quad t_3+\sig_1+\sig_2\leq t\leq t_3+\sig_1+\sig_2+\sig_3 \\
 			\te{\La}(t-\sig_1-\sig_2-\sig_3)   &   \text{ if }\quad t_3+\sig_1+\sig_2+\sig_3\leq t\leq t_4+\sig_1+\sig_2+\sig_3 \\
 			\frac{1}{2}{\la_4}   &   \text{ if }\quad t_4+\sig_1+\sig_2+\sig_3\leq t\leq t_4+s \\
 			\te{\La}(t-s)   &   \text{ if }\quad t_4+s\leq t\leq T+s \\
 		\end{cases}
 	\end{equation*}
 	where $\sig_i=\sig_i(s)$ ($i=1,2,3,4$) are the functions obtained above. Let $\te{\Phi}_s\colon [0,T+s]\to \Ss^3$ be the solution to the initial value problem $\te{\Phi}'=\te{\Phi}\te{\La}$, $\te{\Phi}(0)=\tbf{1}$ and let $\Phi\colon [0,T+s]\to \SO_3$ be its projection. Then using the relation $\chi_i=z_i\la_iz_i^{-1}$ one finds by a verification entirely similar to the one in the proof of (\ref{P:diffuseloose}) that
 	\begin{equation*}
 		\te{\Phi}_s(T+s)=\exp\Big(\frac{\sig_1\chi_1}{2}\Big)\exp\Big(\frac{\sig_2\chi_2}{2}\Big)\exp\Big(\frac{\sig_3\chi_3}{2}\Big)\exp\Big(\frac{\sig_4\chi_4}{2}\Big)\te{\Phi(T)}=\te{\Phi}(T).
 	\end{equation*} 
 	Hence, each $\ga_s=\Phi_se_1$ has the correct final frame. In addition, over each of the subintervals of $[0,T+s]$ listed above, $\ga_s$ is either the composition of a rotation of $\Ss^2$ with an arc of $\ga$, or an arc of circle of radius $\rho_i\in (0,\rho_0)$ $(i=1,2,3,4)$. We conclude from this that the geodesic curvature $\ka^s$ of $\ga_s$  satisfies $\ka^s>\ka_0$ almost everywhere on $[0,T+s]$, that is, $\ga_s\in \sr L_{\ka_0}^{+\infty}(I)$ as we wished. Finally, 
 	\begin{equation*}
 		\tot(\ga_s)=T+s=\tot(\ga)+s
 	\end{equation*}
 	because $\ga_s$ is parametrized by curvature (see (\ref{L:parametrizedbycurvature})), and $(\ga_s)$ is a chain of grafts by construction.
 \end{proof}

\section{Condensed Curves}\label{S:condensed}
\subsection*{Rotation number of a condensed curve}
\label{rotation}The \tdef{rotation number} $N(\eta)$ of a regular closed plane curve $\eta\colon [0,1]\to \R^2$ is simply the degree of its unit tangent vector $\ta\colon \Ss^1\to \Ss^1$ (we may consider $\ga$ and $\ta$ to be defined on $\Ss^1$ since $\ga$ is closed). Suppose now that  $\eta\colon [0,L]\to \R^2$ is parametrized by arc-length, and write 
\begin{equation}\label{E:asin}
	\ta(s)=\exp(i\theta(s)),
\end{equation}
for some angle-function $\theta\colon [0,L]\to \R$. Then the curvature $\ka$ of $\eta$ is given by 
\begin{equation}\label{E:textbook}
\ka(s)=\theta'(s);	
\end{equation}
furthermore, the rotation number $N(\eta)$ of $\eta$ is given by $2\pi N(\eta)=\theta(L)-\theta(0)$. These facts are explained in any textbook on differential geometry. The Whitney-Graustein theorem (\cite{WhiGra}, thm.~1) states that two regular closed plane curves are homotopic through regular closed curves if and only if they have the same rotation number. 

Now suppose $\ga\in \sr L_{\ka_0}^{+\infty}$ has image contained in some closed hemisphere. Let $h_\ga$ be the barycenter, on $\Ss^2$, of the set of closed hemispheres which contain $\Im(\ga)$ (cf.~(\ref{L:closedbarycenter})), and let $\pr\colon \Ss^2\to \R^2$ denote stereographic projection from $-h_\ga$. Define the \tdef{rotation number} $\nu(\ga)$ of $\ga$ by $\nu(\ga)=-N(\eta)$, where $\eta=\pr\circ \ga$. Recall that a curve $\ga \in \sr L_{\ka_0}^{+\infty}$ is called \tdef{condensed} if the image $C$ of its caustic band $C_\ga\colon [0,1]\times [0,\rho_0]\to \Ss^2$ is contained in some closed hemisphere. Because $C_\ga(t,0)=\ga(t)$, any condensed curve is contained in a closed hemisphere, hence we may speak of its rotation number. 

\begin{urem}
We orient the plane on which the sphere is projected by $\pr$ as follows: A basis $\se{v_1,v_2}$ of this plane is positively oriented if and only if $(v_1,v_2,-h_\ga)$ is a positively oriented basis of $\R^3$. This corresponds to looking at the plane from $-h_\ga$ (as is usual with the stereographic projection). The sign in $\nu(\ga)=-N(\pr \circ \ga)$ is introduced to guarantee that a condensed circle traversed $\nu$ times $(\nu\geq 1)$ has rotation number $\nu$. In fact, the rotation number is always positive.
\end{urem}

\begin{lemma}\label{L:positive}
	Let $\ka_0\in \R$ and let $\ga\in \sr L_{\ka_0}^{+\infty}$ be a condensed curve. Then $\nu(\ga)\geq 1$.
\end{lemma}
\begin{proof}
	Let $\ga\in \sr L_{\ka_0}^{+\infty}$ be condensed, with 
	\[
	\Im(C_\ga)\subs H=\set{p\in \Ss^2}{\gen{p,h}\geq 0}
	\]
	and let $\hat{\ga}(t)=B_\ga(t,\rho_0-\pi)$ be the other boundary curve of $B_\ga$. If $\hat\ga(t_0)\in \Int H$ for some $t_0\in [0,1]$ then, by convexity, $H$ must also contain the geodesic segment 
	\begin{equation*}
		B_\ga(\se{t_0}\times [\pi-\rho_0,0])
	\end{equation*}
	joining $\ga(t_0)$ to $\hat\ga(t_0)$. Further, $C_\ga(\se{t_0}\times [0,\rho_0])\subs H$ by hypothesis, hence $H$ contains a geodesic of length $\pi$, and at least one of its enpoints (viz., $\hat\ga(t_0)$) lies in $\Int H$. This contradicts the fact that $H$ has diameter $\pi$. We conclude that $\Im(\hat\ga)\subs -H$. (See fig.~\ref{F:rotation}.)
	
	We lose no generality in assuming that $h=e_3$ and that $\ga$ is $C^1$ (the latter may be achieved by reparametrizing $\ga$ by arc-length). Let $\sig\colon \R\to \Ss^2$ be the standard parametrization of $\bd H$, $\sig(\tau)=(\cos(2\pi \tau),\sin (2\pi \tau),0)$. Because $\ga(t)\in H$, while $\hat\ga(t)\in {-}H$ for each $t\in [0,1]$, there exists $\theta(t)\in [0,\rho_0]$ such that $B_\ga(t,\theta(t))\in \bd H$, that is, $B_\ga(t,\theta(t))=\sig(\tau(t))$. This $\theta$ is unique because $B_\ga(t,\cdot)$ intersects $\bd H$ transversally for all $t$ (otherwise $\dot\ga(t)$ would be orthogonal to $\bd H$, and $\Im(\ga)\nsubs H$).  In addition, both $\theta$ and $\tau$ are $C^1$ functions by the implicit function theorem. We claim that $\tau'>0$ over $[0,1]$.
	
	For $\vphi\in [0,\rho_0]$, let $\ta_{\vphi}$ (resp.~$\no_{\vphi}$) denote the unit tangent (resp.~normal) vector to $\ga_{\vphi}$. Then $\no_{\theta(t)}(t)$ is the unit tangent vector to the curve $u\mapsto B_{\ga}(t,u)$ at $u=\theta(t)$. Since $B_{\ga}(t,0)=\ga(t)\in H$ and $B_{\ga}(t,\rho_0-\pi)=\hat\ga(t)\in {-}H$, with at least one of them lying off $\bd H$,  we have $\gen{\no_{\theta(t)}(t),e_3}>0$ for any $t\in [0,1]$. From
	\begin{alignat*}{9}
		\ta_{\theta(t)}(t)&=\no_{\theta(t)}(t)\times \ga_{\theta(t)}(t)=\no_{\theta(t)}(t)\times \sig(\tau(t)) \quad \text{and}\quad \frac{\dot\sig(\tau(t))}{\abs{\dot\sig(\tau(t))}}=e_3\times \sig(\tau(t)),
	\end{alignat*}
we deduce that
\[
\gen{\ta_{\theta(t)}(t),\dot \sig(\tau(t))}=\abs{\dot\sig(\tau(t))}\gen{\no_{\theta(t)}(t),e_3}>0
\]
(where $\dot\sig(\tau(t))$ denotes the derivative of $\sig$ with respect to $\tau$, at $\tau(t)$). Moreover, $\ta_{\vphi}=\ta_0$ for any $\vphi\in [0,\rho_0-\pi]$ (as seen in eq.~\eqref{E:normaltranslation} in \S 1). Hence,
\begin{equation*}
	\gen{\ta_{\vphi}(t),\dot \sig(\tau(t))}>0\text{\quad for each $\vphi\in [0,\rho_0-\pi],~t\in [0,1]$}.
\end{equation*} 
This implies that $\tau'(t)>0$, as claimed. 
	
	Let $\pr$ denote stereographic projection from ${-}h={-}e_3$  and let $F\colon [0,1]\times [0,1]\to \Ss^2$ be given by 	$F(s,t)=B_{\ga}\big(t,s\theta(t)\big)$. Then $F$ is a regular homotopy between $\ga$ and the geodesic circle $\sig$ traversed a certain number $\nu\geq 1$ of times, in the direction indicated in fig.~\ref{F:rotation} and determined by the parametrization we have chosen. Therefore $\pr \circ \ga$ and $\pr \circ \sig$ are regularly homotopic as well, and $\nu(\ga)=-N(\pr\circ \ga)=-N(\pr \circ \sig)=\nu$ as we wished to show.
\end{proof}

	\begin{figure}[ht]
		\begin{center}
			\includegraphics[scale=.50]{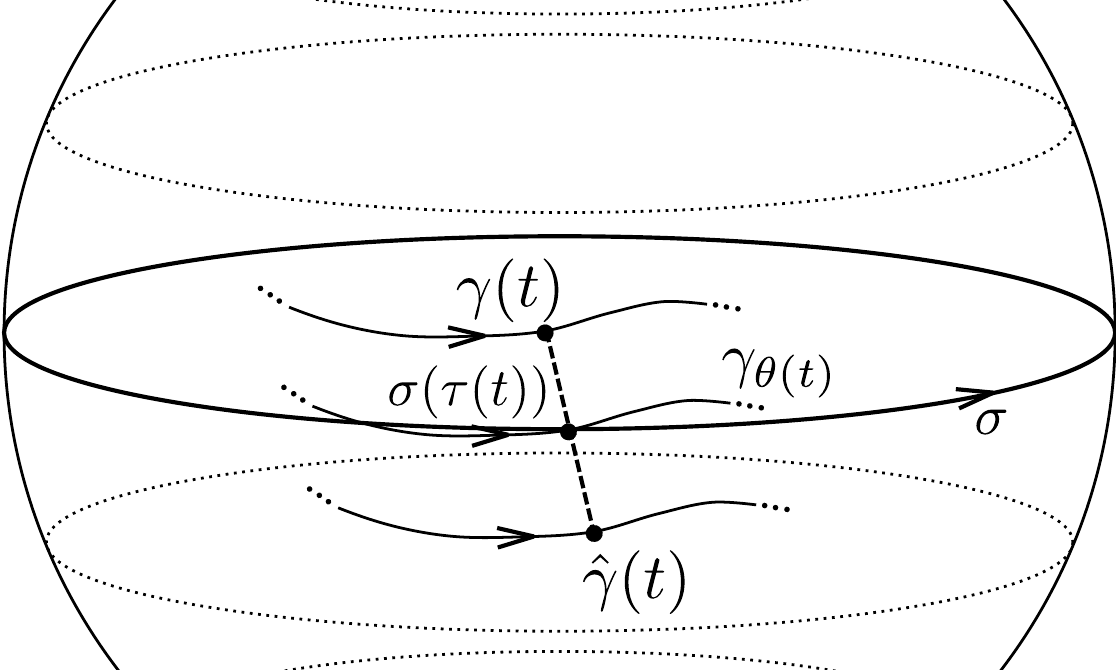}
			\caption{}
			\label{F:rotation}
		\end{center}
	\end{figure}

\begin{urem} It is natural to ask why this notion of rotation number is not extended to a larger class of curves. If $\ga$ is any admissible curve then by parametrizing it by arc-length (so that it becomes $C^1$) and applying Sard's theorem, we deduce that there exists some point $p\in \Ss^2$ not in the image of $\ga$. We could use stereographic projection from $p$ to define the rotation number of $\ga$. The trouble is that it is not clear how $p$ can be chosen so that the resulting number is continuous (i.e., locally constant) on $\sr L_{\ka_0}^{+\infty}$: A different choice of $p$ yields a different rotation number (although its parity remains the same).  In fact, the class of spherical curves for which a meaningful notion of rotation number exists must be restricted, since it is always possible to deform a circle traversed $\nu$ times into a circle traversed $\nu+2$ times in $\sr L_{\ka_0}^{+\infty}$ if $\nu$ is sufficiently large.
\end{urem}

\subsection*{Condensed curves in $\sr L_{\ka_0}^{+\infty}$ for $\ka_0\geq 0$}

\begin{prop}\label{P:positivecondensed}
	Let $A$ be a connected compact space, $\ka_0\geq 0$ and $f\colon A\to \sr L_{\ka_0}^{+\infty}(I)$ be such that $f(a)$ is condensed for all $a\in A$. Then there exists $\nu\geq 1$ such that $f$ is homotopic in $\sr L_{\ka_0}^{+\infty}(I)$ to the constant map $a\mapsto \sig_\nu$, $\sig_\nu$ a circle traversed $\nu$ times. 
\end{prop}
The idea of the proof is to use M\"obius transformations to make the curves $\eta_a=f(a)$ so small that they become approximately plane curves. The hypothesis that the curves are condensed guarantees that the geodesic curvature does not decrease during the deformation.  A slight variation of the Whitney-Graustein theorem is then used to deform the curves to a circle traversed $\nu$ times, where $\nu$ is the common rotation number of the curves.  

We will also need the following technical result, which is a corollary of the proof of (\ref{P:positivecondensed}).
\begin{cor}\label{C:positivecondensed}
Let $\ka_0\geq 0$ and $\ga\in \sr L_{\ka_0}^{+\infty}$ be a condensed curve. Then there exists a homotopy $s\mapsto \ga_s\in \sr L_{\ka_0}^{+\infty}$ $(s\in [0,1])$ such that $\ga_1=\ga$, $\ga_0$ is a parametrized circle and $\Im(C_{\ga_s})$ is contained in an open hemisphere for each $s\in [0,1)$.
\end{cor}

We start by defining spaces of closed curves in $\R^2$ which are analogous to the spaces $\sr L_{\ka_1}^{\ka_2}$ of curves on $\Ss^2$.\footnote{These spaces of plane curves will only be considered in this section.} Let $-\infty\leq \ka_1<\ka_2\leq +\infty$. A \tdef{$(\ka_1,\ka_2)$-admissible plane curve} is an element $(c,z,\hat v,\hat w)$ of $\R^2\times \Ss^1\times L^2[0,1]\times L^2[0,1]$. With such a 4-tuple we associate the unique curve $\ga\colon [0,1]\to \R^2$ satisfying
\begin{equation*}
	\ga(t)=c+\int_0^t v(\tau)\ta(\tau)\?d\tau,\quad \ta(0)=z,\quad \ta'(t)=w(t)i\ta(t)\quad (t\in [0,1]),
\end{equation*}
where $v$ and $w$ are given by eq.~\eqref{E:Sobolev} on p.~\pageref{E:Sobolev} and $i=(0,1)$ is the imaginary unit. The space of all $(\ka_1,\ka_2)$-admissible plane curves is thus given the structure of a Hilbert manifold, and we define $\sr W_{\ka_1}^{\ka_2}$ to be its subspace consisting of all closed curves.

Although $\dot\ga$ is defined only almost everywhere for a curve $\ga\in \sr W_{\ka_1}^{\ka_2}$, its unit tangent vector $\ta$ is defined over all of $[0,1]$, and if we parametrize $\ga$ by a multiple of arc-length instead, then $\dot\ga$ is defined and nonzero everywhere. More importantly, since $\ta$ is (absolutely) continuous, we may speak of the rotation number of $\ga$ and \eqref{E:textbook} still holds a.e..


\begin{lemma}\label{L:Whitney}
Let $A$ be compact and connected, $\ka_0\geq 0$ and $A\to \sr W_{\ka_0}^{+\infty}$, $a\mapsto \eta_a$, be a continuous map. Then there exists a homotopy $[0,1]\times A\to \sr W_{\ka_0}^{+\infty}$, $(s,a)\mapsto \eta_a^s$, such that $\eta_a^0=\eta_a$ and 
\begin{equation*}
\qquad	\eta_a^1(t)=\sig_N(t+t_a)\text{\quad for all $a\in A$, $t\in [0,1]$},
\end{equation*}
where $\sig_N(t)=R_0\exp(2\pi iN t)$ is a circle traversed $N>0$ times. In addition, if the image of $\eta_a$ is contained in some ball $B(0;R)$ for all $a\in A$, then we can arrange that $\eta_a^s$ have the same property for all $s\in [0,1]$ and $a\in A$.
\end{lemma}
Thus, given a family of curves in $\sr W_{\ka_0}^{+\infty}$ indexed by a compact connected set, we may deform all of them to the same parametrized circle $\sig_N$, except for the starting point of the parametrization. 

\begin{proof} Since $A$ is connected, all the curves $\eta_a$ have the same rotation number $N$. Moreover, $N>0$ because of \eqref{E:textbook} and the fact that $\ka_0\geq 0$. 



For $\eta\in \sr W_{\ka_0}^{+\infty}$, let $z_\eta=\ta_\eta(0)$, where $\ta_\eta$ is the unit tangent vector to $\eta$. The homotopy $g\colon \sr [0,1]\times A\to \sr W_{\ka_0}^{+\infty}$ by translations,
\begin{equation*}
	g(s,a)(t)=\eta_a(t)-s\big(iz_{\eta_a}+\eta_a(0)\big)\qquad (s,\,t\in [0,1],~a\in A), 
\end{equation*} 	
preserves the curvature and, for any $a\in A$, $g(1,a)$ has the property that it starts at some $z\in \Ss^1$ in the direction $i z$. Thus, we may assume without loss of generality that the original curves $\eta_a$ have this property.

Let $\rho_0=\frac{1}{\ka_0}$, $L(\eta_a)$ denote the length of $\eta_a$, $L_0=\min_{a\in A}\se{L(\eta_a)}$ and let $R_1>0$ satisfy
\begin{equation}\label{E:easy}
	R_1<\min\Big\{\frac{L_0}{2\pi N},\rho_0\Big\}.\footnote{If $\ka_0=0$ then we adopt the convention that $\rho_0=+\infty$.}
\end{equation}
Define $f\colon [0,1]\times A\to \sr W_{\ka_0}^{+\infty}$ to be the homotopy given by
\begin{equation*}
\qquad	f(s,a)(t)=\eta_a(0)+\Big((1-s)+s\frac{2\pi N R_1}{L(\eta_a)}\Big)\big(\eta_a(t)-\eta_a(0)\big)\qquad (s,\,t\in [0,1],~a\in A).
\end{equation*}
Then $f(1,a)$ has length $L=2\pi N R_1$ for all $a\in A$. In addition, the curvature of $f(s,a)$ is bounded from below by $\ka_0$ for all $s\in [0,1]$, $a\in A$ and almost every $t\in [0,1]$, as an easy calculation using \eqref{E:easy} shows. 

The conclusion is that we lose no generality in assuming that the curves $\eta_a$ all have the same length $L=2\pi NR_1$. Further, by (\ref{L:reparametrize}), we can assume that they are all parametrized by a multiple of arc-length. This implies that $\dot\eta_a$ takes values on the circle $L \?\Ss^1$ of radius $L$. Using angle-functions $\theta_{a}$ with $\theta_{a}(0)=0$ and $\theta_{a}(1)=2\pi N$, we can write:
	\begin{equation*}
\qquad		\dot\eta_a(t)=Lz_{a}\exp\big(i\theta_{a}(t)\big)\qquad (t\in [0,1]),
	\end{equation*}
where $z_a=\ta_{\eta_a}(0)$. Let $\theta(t)=2\pi N t$, $t\in [0,1]$, and define 
\begin{equation*}
\quad	\theta_a^s(t)=(1-s)\theta_a(t)+s\theta(t),\qquad \bar\tau_a^s(t)= Lz_a\exp(i\theta_a^s(t)) \quad (s,\,t\in [0,1],~a\in A).
\end{equation*}
Then $\theta_a^s(0)=0$ and $\theta_a^s(1)=2\pi N$ for all $s\in [0,1]$, $a\in A$. The idea is that $\bar\tau_a^s$ should be the tangent vector to a curve; the problem is that this curve need not be closed. We can fix this by defining instead
\begin{equation*}
	\tau_a^s(t)=\bar\tau_a^s(t)-\int_0^1\bar\tau_a^s(v)\,dv,\qquad \eta_a^s(t)=-iz_a+\int_0^t\tau_a^s(v)\,dv.
\end{equation*}

The conditions $\int_0^1\tau_a^s(t)\,dt=0$ and $\tau_a^s(0)=\tau_a^s(1)$ then guarantee that $\eta_a^s$ is a closed curve. Because $\theta_a^s(1)=2\pi N$ and $N>0$, $\bar\tau_a^s$ must traverse all of $L\?\Ss^1$, so that $\int_0^1\bar\tau_a^s(v)\,dv$ lies in the interior of the disk bounded by this circle for any $s\in [0,1]$, $a\in A$. Consequently, $\tau_a^s(t)$ never vanishes. Moreover,
\begin{equation*}
\qquad	\eta_a^0=\eta_a\quad\text{ and }\quad\eta_a^1(t)=-iz_{\eta_a}\exp(2\pi N i t)\quad\text{for all $a\in A$}.
\end{equation*}

Finally,  $\eta_a^s$ has positive curvature for all $s\in [0,1]$ and $a\in A$. Although it is easier to see this using a geometrical argument, the following computation suffices: The curvature $\ka_a^s$ of $\eta_a^s$ is given by
\begin{alignat*}{10}
	\ka_a^s(t)&=\frac{\det\big(\tau_a^s(t),\dot{\tau}_a^s(t)\big)}{\abs{\tau_a^s(t)}^3}=\frac{L^2\dot\theta_a^s(t)}{\abs{\tau_a^s(t)}^3}\Big[1-\det\Big(\int_0^1\exp(i\?\theta_a^s(v))\,dv,\,i\exp(i\?\theta_a^s(t))\Big)\Big].
\end{alignat*}
Because $\theta_a^s=(1-s)\theta_a+s\theta$ is monotone increasing (recall that $\theta_a'=\ka_a>\ka_0\geq 0$ a.e.~by hypothesis), the map $t\mapsto \exp(i\theta_a^s(t))$ runs over all of $\Ss^1$ for any $s$ and $a$. As a consequence, the integral above has norm strictly less than 1, hence so does the determinant. In fact, since $A$ is compact, we can find a constant $C>0$, independent of $a$ and $s$, such that 
\begin{equation*}
	\ka_a^s>C\ka_0.
\end{equation*}
For $\la>0$ and an admissible plane curve $\ga$,  the curve $\la\ga$ has curvature given by $\tfrac{\ka}{\la}$, where $\ka$ is the curvature of  $\ga$. Again using compactness of $A$,  we may find a smooth function $\la\colon [0,1]\to (0,1]$ such that $\la(0)=1$ and $\la(s)$ is as small as necessary for $s\in (0,1]$ to guarantee that $\ka_a^s>\ka_0$ for all $s\in [0,1]$ and $a\in A$ if we replace $\eta_a^s$ with $\la(s)\eta_a^s$. In addition, we can choose $\la$ so that the image of $\la(s)\eta_a^s$ is contained in the ball $B_R(0)$ if this is the case for each $\eta_a$. This establishes the lemma with $R_0=\la(1)$.
\end{proof}

The next result states that the geodesic curvature of a curve $\ga\colon [0,1]\to \Ss^2$ and the curvature of the plane curve obtained by projecting $\ga$ orthogonally on $T_p
\Ss^2$ are roughly the same, as long as the curve is contained in a small neighborhood of $p$.

\begin{lemma}\label{L:projection}
	Let $\ka_0<\ka_1<\ka_2$ and $p\in \Ss^2$ be given. Identifying $T_p\Ss^2$ with $\R^2$, with $p$ corresponding to the origin, let $P\colon \Ss^2\to \R^2$ be the orthogonal projection. Then there exists $\eps>0$ such that:
	\begin{enumerate}
		\item [(a)] If $\ga\in \sr L_{\ka_2}^{+\infty}$ satisfies $d(\ga(t),p)<\eps$ for all $t\in [0,1]$, then $\eta=P \circ\ga\in \sr W_{\ka_1}^{+\infty}$.
		\item [(b)]  If $\eta\in \sr W_{\ka_1}^{+\infty}$ satisfies $\abs{\eta(t)}<\eps$ for all $t\in [0,1]$, then $\ga=P^{-1} \circ\eta\in \sr L_{\ka_0}^{+\infty}$.
	\end{enumerate}
\end{lemma} 
In part (a), $d$ denotes the distance function on $\Ss^2$ and the transformation $P^{-1}$ in part (b) is to be understood as the inverse of $P$ when restricted to the hemisphere $\set{q\in \Ss^2}{\gen{q,p}>0}$. 

\begin{proof}
	The proof is straightforward and will be omitted. See \cite{tese}, (6.4).
\end{proof}

\begin{lemma}\label{L:Moebius}
	Let $h\in \Ss^2$, $H=\set{q\in \Ss^2}{\gen{q,h}\geq 0}$, let $\pr\colon \Ss^2\to \R^2$ denote stereographic projection from ${-}h$. Let $\ka_0>0$ and $\ga\in \sr L_{\ka_0}^{+\infty}$ be such that $\Im(C_\ga)\subs H$. Define $T_r\colon \Ss^2\to \Ss^2$ to be the M\"obius transformation (dilatation) given by
	\begin{equation*}
	\quad	T_r(p)=\pr^{-1}\big(r\pr (p)\big) \quad (r\in (0,1],~p\in \Ss^2).
	\end{equation*}
	Then, given $\ka_1>\ka_0$, there exists $r_0>0$, depending only on $\ka_0$ and $\ka_1$, such that the geodesic curvature $\ka^r$ of $T_r(\ga)$ satisfies $\ka^r>\ka_1$ a.e.~for any $r\in (0,r_0)$.
\end{lemma}
\begin{proof}
	Suppose that $\ga\in \sr L_{\ka_0}^{+\infty}$ is parametrized by its arc-length and let $\sig$ be a parametrization, also by arc-length, of an arc of the osculating circle to $\ga$ at $\ga(s_0)$, i.e., let $\sig$ satisfy:
	\begin{equation*}
		   \sig(s_0)=\ga(s_0),\qquad \sig'(s_0)=\ga'(s_0),\qquad \sig''(s_0)=\ga''(s_0).
	\end{equation*}
(It makes sense to speak of $\ga''$ (as an $L^2$ map) because $\ga'=\ta$ is $H^1$ by hypothesis.) Then $T_r\circ \sig$ has contact of order 3 with $T_r\circ \ga$ at $s_0$, hence their geodesic curvatures at the corresponding point agree. Therefore, it suffices to prove the result for a circle $\Sig$ whose center $\chi$ lies in $\Int H$. Let $\rho_i=\arccot \ka_i$, $i=0,1$, and $\rho$ be the radius of curvature of $\Sig$,  $\rho<\rho_0< \frac{\pi}{2}$. If $d$ denotes the distance function on $\Ss^2$, then $\Sig\subs B_d\big(h;\frac{\pi}{2}+\rho_0\big)$ (where the latter denotes the set of $q\in \Ss^2$ such that $d(h,q)<\frac{\pi}{2}+\rho_0$). Choose $r_0$ such that 
	\begin{equation*}
		T_r\big(B_d\big(h;\tfrac{\pi}{2}+\rho_0\big)\big)\subs B_d\big(h;\rho_1\big) \text{\ for all $r\in (0,r_0)$;}
	\end{equation*}
such an $r_0$ exists because $B_d\big(h;\tfrac{\pi}{2}+\rho_0\big)$ is a distance $\frac{\pi}{2}-\rho_0>0$ away from $-h$. Then $T_r(\Sig)$ is a circle, for a M\"obius transformation such as $T_r$ maps circles to circles, and its diameter is at most $2\rho_1$. Thus, its geodesic curvature must be greater than $\ka_1$. Moreover, it is clear that the choice of $r_0$ does not depend on $h$ or on $\Sig$, only on $\rho_0$ and $\rho_1$.
\end{proof}


\begin{proof}[Proof of (\ref{P:positivecondensed})]
	Let $\ga_a$ denote $f(a)$ and let $h_a$ be the barycenter of the set of closed hemispheres which contain $\Im(C_{\ga_a})$; by (\ref{L:causticbarycenter}), the map $h\colon A\to \Ss^2$ so defined is continuous. 
	
	Let $\pr_a$ denote stereographic projection $\Ss^2\to \R^2$ from $-h_a$, so that $h_a$ is projected to the origin, and define a family $T^s_a\colon \Ss^2\to \Ss^2$ of M\"obius transformations by:
	\[
\qquad \quad	T^s_a(q)=\pr_a^{-1}(s\pr_a(q))\qquad (q\in \Ss^2,~s\in (0,1],~a\in A).
	\]
Set $\ga_a^s=T^s_a\ga_a$. 

Assume first that $\ka_0>0$. From (\ref{L:Moebius}) it follows that we can choose $\de>0$ so small that the geodesic curvature of $\ga_a^\de$ is greater than $\ka_0+2$ a.e.~for any $a\in A$. Now choose $\eps>0$ as in (\ref{L:projection}), with $\ka_1=\ka_0+1$, $\ka_2=\ka_0+2$. By reducing $\de$ if necessary, we can guarantee that the curves $\ga_a^\de$ have image contained in $B_d(h_a;\eps)$, for each $a$. Let $\eta_a$ be the orthogonal projection of $\ga_a^\de$ onto $T_{h_a}\Ss^2$.  We are then in the setting of (\ref{L:Whitney}). The conclusion is that we can deform all $\eta_a$ to a single circle $\sig_\nu$, modulo the starting point of the parametrization, in such a way that the curves have image contained in $B(0;\eps)$ and curvature greater than $\ka_0+1$ throughout the deformation. By (\ref{L:projection}) again, when we project this homotopy back to $\Ss^2$, the geodesic curvature of the curves is always greater than $\ka_0$.

To sum up, we have described a homotopy $H\colon [0,1]\times A\to \Ss^2$ such that $H(0,a)=\ga_a$ and $H(1,a)$ is a circle traversed $\nu$ times for all $a\in A$; further, the geodesic curvature $\ka_a^s$ of $H(s,a)$ satisfies $\ka_a^s(t)>\ka_0$ for each $s\in [0,1]$ and almost every $t \in [0,1]$. These curves $H(a,s)$ do not satisfy $\Phi(0)=I=\Phi(1)$, but we can correct this by setting 
\begin{equation*}
	\bar H(s,a)=\Phi_{H(a,s)}(0)^{-1}H(a,s)\?
\end{equation*}
and using $\bar H$ instead; this has no effect on the geodesic curvature and finishes the proof that $f$ is null-homotopic, since $\bar H(1,a)$ is the same parametrized circle for all $a$.

We shall now indicate how to modify the proof when $\ka_0=0$. With the notation as above, let 
\[
d_0=\sup\set{d\big(C_{\ga_a^{0.5}}(t,\theta),h_a\big)}{t\in [0,1],~\theta\in [0,\rho_0],~a\in A}.
\]
Then $d_0<\frac{\pi}{2}$ for all $a\in A$ because $A$ is compact and $\Im(C_{\ga_a^s})$ is contained in the o\sz p\sz e\sz n hemisphere $\set{q\in \Ss^2}{\gen{q,h_a}>0}$ for all $a\in A$,~$s\in (0,1)$ (see the first part of the proof of (\ref{C:positivecondensed}) below). Choose $d_1$ with $d_0<d_1<\frac{\pi}{2}$. If $\Sig$ is an osculating circle to some $\ga_a^{0.5}$, we can assert that its center lies in $B_d\big(h_a;d_1\big)$, hence $\Sig\subs B_d\big(h_a;d_1+\frac{\pi}{2}\big)$. This uniform estimate allows us to repeat the reasoning in the proof of (\ref{L:Moebius}) to find $\de>0$ such that the geodesic curvature of $\ga_a^\de$ is greater than $\ka_0+2$ a.e.~for any $a\in A$. The rest of the proof is the same as when $\ka_0>0$.
\end{proof}

We now provide a proof of (\ref{C:positivecondensed}). This result will be used to show that a notion of rotation number for non-diffuse curves, which will be introduced in the next section, coincides with the one presented at the beginning of this section. 

\begin{proof}[Proof of (\ref{C:positivecondensed})]
	Let $h_\ga$ be the barycenter on $\Ss^2$ of the set of closed hemispheres which contain $\Im(C_\ga)$ and, as in the proof of (\ref{P:positivecondensed}), define $\ga_s=T^s\circ \ga$, where
	\begin{equation}\label{E:Moeb}
		T^s(q)=\pr^{-1}(s\pr(q))\quad (q\in \Ss^2,~s\in (0,1])
	\end{equation}
	and $\pr$ denotes stereographic projection from $-h_\ga$. Let $H=\set{p\in \Ss^2}{\gen{p,h_\ga}>0}$. We claim that $\Im(C_{\ga_s})\subs H$ for all $s\in (0,1)$. This follows from the following two assertions:
	\begin{enumerate}
		\item [(i)] If $\Im(C_{\ga_{s}})\subs \bar H$, then there exists $\eps>0$ such that $\Im(C_{\ga_\sig})\subs H$ for all $\sig\in (s-\eps,s)$;
		\item [(ii)] If $\Im(C_{\ga_{s}})\nsubs H$, then there exists $\eps>0$ such that $\Im(C_{\ga_\sig})\nsubs \bar H$ for all $\sig\in (s,s+\eps)$. 
	\end{enumerate}
	For any $s$, the boundary of $\Im(C_{\ga_{s}})$ is contained in the union of the images of $\ga_{s}=C_{\ga_{s}}(\cdot,0)$ and $\ce\ga_{s}=C_{\ga_{s}}(\cdot,\rho_0)$. Moreover, $\ga$  has positive geodesic curvature by hypothesis, and a straightforward calculation shows that $\ce\ga$ also does (the details may be found in (\ref{L:cega})). 
	
	If $\Im(C_{\ga_{s}})\subs H$ then (i) is obviously true, since $H$ is an open hemisphere; similarly, (ii) clearly holds if $\Im(C_{\ga_s})\nsubs \bar H$. Suppose then that $\Im(C_{\ga_{s}})\subs \bar H$, but $\Im(C_{\ga_{s}})\nsubs H$ for some $s>0$. This means that there exists $t_0\in [0,1]$ such that either $\ga_s$ or $\ce{\ga}_s$ is tangent to $\bd H$ at $\ga_{s}(t_0)$ or $\ce\ga_s(t_0)$, respectively. In the first case, $\no_{\ga_s}(t_0)=h_\ga$, and in the second $\no_{\ga_s}(t_0)=-h_\ga$. In either case, $C_{\ga_{s}}\big(\{t_0\}\times[0,\rho_0]\big)$ is an arc of the geodesic through $\ga_{s}(t_0)$ and $h_\ga$. Such geodesics through $h_\ga$ are mapped to lines through the origin by $\pr$, hence \eqref{E:Moeb} implies that there exists $\eps>0$ such that $C_\ga(t,\sig)\subs H$ for any $t\in (t_0-\eps,t_0+\eps)$ and $\sig\in (s-\eps,s)$ and $C_\ga(t_0,\sig)\nsubs \bar H$ for any $\sig\in (s,s+\eps)$. Furthermore, since the geodesic curvatures of $\ga$, $\ce\ga$ are positive and $\bd H$ is a geodesic, the set of $t_0\in [0,1]$ where $\ga,\ce\ga$ are tangent to $\bd H$ must be finite. This implies (i) and (ii).
	
	Now let $S=\set{s\in (0,1)}{\Im(C_{\ga_s})\nsubs H}$. Assume that $S\neq \emptyset$ and let $s_0=\sup S$. Applying (i) to $\ga_{1}=\ga$ we conclude that there exists $\eps>0$ with $S\cap (1-\eps,1)=\emptyset$. Hence, $s_0<1$ and $\Im(C_{\ga_{s_0}})\nsubs H$ by construction. An application of (ii) yields a contradiction. Thus, $S=\emptyset$.
	
	Let $\rho_0=\arccot\ka_0$  and $r=\frac{\pi}{2}-\rho_0$. Choosing $\de>0$ so that $\Im(\ga_\de)\subs B_d(h_\ga;r)$, and proceeding as in the proof of (\ref{P:positivecondensed}), we can extend $s\mapsto \ga_s$ $(s\in [\de,1])$ to all of $[0,1]$ so that $\ga_0$ is a parametrized circle and $\Im(\ga_s)\subs B_d(h_\ga;r)$ for all $s\in [0,\de]$ (where $d$ denotes the distance function on $\Ss^2$). The inequality $d(\eta(t),C_\eta(t,\theta))=\theta<\rho_0$, which holds for any $\eta\in \sr L_{\ka_0}^{+\infty}$, implies that 
	\begin{equation*}
		d(h_\ga,C_{\ga_s}(t,\theta))<\frac{\pi}{2}\quad\text{for any $t\in [0,1],~\theta\in [0,\rho_0]$ and $s\in [0,\de]$}.
	\end{equation*}
	Hence $\Im(C_{\ga_s})\subs H$ for all $s\in [0,\de]$. The same inclusion for $s\in [\de,1)$ was established above, so the proof is complete.
\end{proof}

\begin{cor}\label{L:super}
	Let $\ka_0\geq 0$ and $1\leq \nu\in \N$.
	\begin{enumerate}
		\item [(a)] The subset $\sr O$ (resp.~$\sr O_\nu$) of $\sr L_{\ka_0}^{+\infty}(I)$ consisting of all condensed curves (resp.~all condensed curves having rotation number $\nu$) is the closure of an open set.
		\item [(b)] If $\ga\in \sr O_\nu$ and $\sr U\subs \sr L_{\ka_0}^{+\infty}(I)$ is any open set containing $\ga$, then $\ga$ is homotopic to a smooth curve within $\sr O_\nu\cap \sr U$.
	\end{enumerate} 
\end{cor}
\begin{proof}
	Let $\sr S\subs \sr O$ be the subset consisting of all curves $\ga\in \sr L_{\ka_0}^{+\infty}(I)$ such that $\Im(C_\ga)$ is contained in an o\sz p\sz e\sz n hemisphere. Then $\sr S$ is open, because if the compact set $C=\Im(C_\ga)$ is such that $\gen{c,h}>0$ for some $h\in \Ss^2$ and all $c\in C$, then the same inequality holds for all $c\in \Im(C_\eta)$ whenever $\eta\in \sr L_{\ka_0}^{+\infty}(I)$ is sufficiently close to $\ga$. Similarly, $\sr O$ is closed. For if $\ga\nin \sr O$, then, by (\ref{L:closedhemisphere}) and (\ref{L:Steinitz}), we can find a 3-dimensional simplex with vertices in $\Im(C_\ga)$ containing $0\in \R^3$ in its interior. If $\eta\in \sr L_{\ka_0}^{+\infty}(I)$ is sufficiently close to $\ga$ then we can also find a simplex $\De_\eta$ with vertices in $\Im(C_\eta)$ such that $0\in \Int \De_\eta$.  It follows that $\bar{\sr S}\subs \sr O$.
	
	Let $\ga\in \sr O$. Define a family $T^s\colon \Ss^2\to \Ss^2$ of M\"obius transformations by \eqref{E:Moeb}, where $\pr\colon \Ss^2\to \R^2$ denotes stereographic projection from $-h_\ga$, and $h_\ga$ is the barycenter of the set of closed hemispheres which contain $C=\Im(C_\ga)$ (cf.~(\ref{L:causticbarycenter})). Then $\ga_s=T^s\circ \ga\in \sr S$ for all $s\in (0,1)$ by (\ref{C:positivecondensed}), establishing the reverse inequality $\bar {\sr S}\sups \sr O$. The proof of the assertion about $\sr O_\nu$ is analogous and will be omitted.


To prove (b), let $\eps>0$ be such that $\ga_s=T^s\circ \ga \in \sr U$ for all $s\in [1-\eps,1]$. Choose a path-connected neighborhood $\sr V\subs \sr S\cap \sr U$ of $\ga_{1-\eps}$, and, for $s\in [0,1-\eps]$, let $\ga_s$ be a path in $\sr V$ joining a smooth curve $\ga_0$ to $\ga_{1-\eps}$. As each $\ga_s$ is condensed ($s\in [0,1]$), $\nu(\ga_s)$ is defined for all $s$; since it can only take on integral values, it must be independent of $s$. Thus, $s\mapsto \ga_s$ $(s\in [0,1]$) is the desired path.
\end{proof}


\subsection*{Condensed curves in $\sr L_{\ka_0}^{+\infty}$ for $\ka_0<0$}

The purpose of this subsection is to prove the following analogue to (\ref{P:positivecondensed}).
\begin{prop}\label{P:negativecondensed}
	Let $K$ be a connected compact space, $\ka_0< 0$ and $f\colon K\to \sr L_{\ka_0}^{+\infty}(I)$ be such that $f(p)$ is condensed for all $p\in K$. Then there exists $\nu\geq 1$ such that $f$ is homotopic in $\sr L_{\ka_0}^{+\infty}(I)$ to the constant map $p\mapsto \sig_\nu$, $\sig_\nu$ a circle traversed $\nu$ times. 
\end{prop}

Let $1\leq \nu\in \N$ and let $\Ss^2_\nu$ denote the $\nu$-sheeted connected covering of $\Ss^2\ssm \se{\pm \text{point}}$, where we may assume that the point is the north pole $N$. We will identify $\Ss^1\times (-\frac{\pi}{2},\frac{\pi}{2})$ with $\Ss^2\ssm \se{\pm N}$ through the homeomorphism $h$ given by $h(z,\phi)=(\cos\phi\? z,\sin\phi)$. This, in turn, yields an identification of $\Ss^2_\nu$ with  $\Ss^1_\nu\times (-\frac{\pi}{2},\frac{\pi}{2})$, where $\Ss^1_\nu$ is the $\nu$-sheeted connected covering space of $\Ss^1$. 
We will prefer to work with the space $\Ss^1_\nu\times (-\frac{\pi}{2},\frac{\pi}{2})$ instead of $\Ss^2_\nu$, but its Riemannian metric is the one induced on the latter space by $\Ss^2$ through the covering map.

\begin{defn}\label{D:acceptable}
\hspace{-3pt}\footnote{These notions will only be used in this subsection.}
Let $0<R<\frac{\pi}{2}$. An \tdef{acceptable band} $A\colon [0,1]\times [0,1]\to \Ss^1_\nu \times (-\frac{\pi}{2},\frac{\pi}{2}) \equiv \Ss^2_\nu$ is a map given by
\begin{equation}\label{E:acceptable}
	A(t,u)=\big(\exp(2\pi \nu it)\?,\?(1-u)\theta_-(t)+u\theta_+(t)\big)\quad (t,\,u\in [0,1])
\end{equation}
and satisfying the following conditions:
\begin{enumerate}
	\item [(i)] $\theta_{\pm}\colon [0,1]\to (-\frac{\pi}{2},\frac{\pi}{2})$ are continuous, $0\leq \theta_+\leq R$ and  $-R\leq \theta_-\leq 0$.
	\item [(ii)] Let $\bd A_{+}$ (resp.~$\bd A_-$) denote the image of $[0,1]\times \se{1}$ (resp.~$[0,1]\times \se{0}$) under $A$. Then $d(p,\bd A_-)\geq R$ and $d(q,\bd A_+)\geq R$ for every $p\in \bd A_+$ and every $q\in \bd A_-$.\footnote{Here and in what follows, $d$ denotes the distance function on $\Ss^2_\nu$ (or on $\Ss^2$).}
\end{enumerate}
	The \tdef{interior} $\ring A$ of $A$ is simply the interior of the image of $A$. The set of all acceptable bands (for fixed $R$) will be denoted by $\sr A$ and furnished with the $C^0$ (uniform) topology.  Finally, we denote by $\sr G$ the subspace of $\sr A$ consisting of all acceptable bands $A$ such that $d(p,\bd A_-)=R=d(q,\bd A_+)$ for any $p\in \bd A_+$ and $q\in \bd A_-$. Such a band will be called \tdef{good} and $R$ its \tdef{width}.
\end{defn}
The motivation for this definition comes from the following lemma.

\begin{lemma}\label{L:regularisgood}
	Let $\ka_0=\cot \rho_0<0$ and $\ga\in \sr L_{\ka_0}^{+\infty}$ be a condensed curve having rotation number $\nu$. Then the image of the lift of the regular band $B_\ga\colon [0,1]\times [\rho_0-\pi,0]\to \Ss^2$ of $\ga$ to $\Ss^2_\nu$ is the image of a good band of width $\pi-\rho_0$.
\end{lemma}
Recall that the rotation number $\nu$ of $\ga$ must be positive by (\ref{L:positive}).
\begin{proof}
	By hypothesis, the image of the caustic band $C_\ga$ is contained in a hemisphere, say, 
	\begin{equation*}
		H=\set{p\in \Ss^2}{\gen{p,N}\geq 0}.
	\end{equation*}
Let $\hat\ga$ be the other boundary curve of $B_\ga$, $\hat\ga(t)=B_\ga(t,\rho_0-\pi)$. Then $\hat\ga(t)=-C_\ga(t,\rho_0)\in -H$ for all $t\in [0,1]$. Since $d(\ga(t),\hat\ga(t))=\pi-\rho_0<\frac{\pi}{2}$, $\Im(\ga)\subs H$ and $\Im(\hat\ga)\subs {-}H$, the image of the regular band is actually contained in $\Ss^1\times [\rho_0-\pi,\pi-\rho_0]$ (where we are identifying $\Ss^2\ssm \se{\pm N}$ with $\Ss^1\times (-\frac{\pi}{2},\frac{\pi}{2})$). 
	
	Let $\te{B}_\ga\colon [0,1]\times [\rho_0-\pi,0]\to \Ss^2_\nu$ be the lift of $B_\ga$ to $\Ss^2_\nu \equiv \Ss^1_\nu\times (-\frac{\pi}{2},\frac{\pi}{2})$. For each $z\in \Ss^1_\nu$, let the \tdef{meridian} $\mu_z$ be the geodesic parametrized by $\mu_z(t)=(z,t)$, $t\in (-\frac{\pi}{2},\frac{\pi}{2})$. By what we have just proved and the fact that $\ga$ has rotation number $\nu$, we may define continuous functions $\theta_{\pm}\colon \Ss^1_\nu\to (-\frac{\pi}{2},\frac{\pi}{2})$ by the relations
	\[
	\mu_z(\theta_+(z))\in \te{B}_\ga([0,1]\times \se{0}) \text{\quad and \quad}\mu_z(\theta_-(z))\in \te{B}_\ga([0,1]\times \se{\rho_0-\pi}). 
	\]
	Then the map $A\colon [0,1]\times [0,1] \to \Ss^1_\nu \times (-\frac{\pi}{2},\frac{\pi}{2}) \equiv \Ss^2_\nu$ given by
\begin{equation*}
	A(t,u)=\big(\exp(2\pi \nu it)\?,\?(1-u)\theta_-(t)+u\theta_+(t)\big)\quad (t,\,u\in [0,1])
\end{equation*}
defines an acceptable band whose image coincides with that of $\te{B}_\ga$. Furthermore, the equality $d(\ga(t),\hat{\ga}(t))=\pi-\rho_0$ implies that $d(p,\bd A_{\pm})\leq \pi-\rho_0$ for any $p\in \bd A_{\mp}$. We claim that $A$ is a good band of width $\pi-\rho_0$. To see this, suppose $\eta\colon [0,1]\to \Ss^2_\nu$ is a piecewise $C^1$ curve joining $\bd A_-$ to $\bd A_+$ and write $\eta(u)=\te{B}_\ga(t(u),\theta(u))$. Then the length is minimized when $\theta$ is monotone and $\dot t(u)=0$ for all $u\in [0,1]$, hence the minimal length is $\pi-\rho_0$; for the proof see the similar argument in \cite{tese}, (10.5).
\end{proof}

\begin{lemma}\label{L:contratil}
	The space $\sr A$ is contractible.
\end{lemma}
\begin{proof}
 	Let $A\in \sr A$ be given by \eqref{E:acceptable} and let $s\in [0,1]$. Define a family of acceptable bands $A_s$ by
	\begin{equation*}
		A_s(t,u)=\big(\exp(2\pi \nu it)\?,\?(1-u)\theta^s_-(t)+u\theta^s_+(t)\big),
	\end{equation*}
	where
	\begin{equation*}
		\theta^s_+(t)=(1-s)\theta_+(t)+sR\text{\quad and \quad}\theta^s_-(t)=(1-s)\theta_-(t)-sR
	\end{equation*}
	Then the map $\sr A\times [0,1]\to \sr A$ given by $(A,s)\mapsto A_s$ is a contraction of $\sr A$.
\end{proof}

\begin{lemma}\label{L:retrato}
	The subspace $\sr G$ is a retract of $\sr A$.
\end{lemma}
\begin{proof}
Let $A\in \sr A$ be given by \eqref{E:acceptable}. Define $A^1=\Im(A)$, $\theta_{\pm}^1=\theta_{\pm}$ and
\begin{equation*}
	\qquad	A^{2}=\set{p\in A^1}{d(p,\bd A^1_{-})\leq R+\tfrac{1}{2}}.
\end{equation*}
We will call a geodesic $\mu_z$ in $\Ss^2_\nu\equiv \Ss^1_\nu\times (-\frac{\pi}{2},\frac{\pi}{2})$ of the form $\se{z}\times (-\frac{\pi}{2},\frac{\pi}{2})$ a \tdef{meridian}, and parametrize it by $\mu_z(t)=(z,t)$. We begin by establishing the following facts:
\begin{enumerate}
	\item [(a)] Each meridian $\mu_z$ intersects $\bd A^2$ at exactly two points $\mu_z(\theta^2_-(z))$ and $\mu_z(\theta^2_+(z))$, with $\theta^2_+\geq 0$ and $\theta^2_-\leq 0$. We define $\bd A^2_{\pm}$ as the set of all $\mu_z(\theta^2_{\pm}(z))$ for $z\in \Ss^1_\nu$.
	\item [(b)] $\bd A^2_-=\bd A^1_-$.
	\item [(c)] $p\in \bd A_{+}^{2}$ if and only if one of the following holds:
		\begin{alignat*}{9}
			 & p\in \bd A^{1}_{+} & & \text{\quad and\quad }d(p,\bd A^1_-)\leq R+\tfrac{1}{2}, \quad \text{or} \\ 
			 & p\in \ring A^1& &\text{\quad and\quad} d(p,\bd A^1_{-})=R+\tfrac{1}{2}.
		\end{alignat*} 
	\item [(d)] The boundary $\bd A^2$ of $A^2$ is the disjoint union of $\bd A^2_+$ and $\bd A^2_-$. Moreover, 
	\begin{equation*}
		R\leq d(p,\bd A^2_-)\leq R+\frac{1}{2}\text{\quad and \quad}R\leq d(q,\bd A^2_+)\leq d(q,\bd A^1_+)
	\end{equation*}
	for any $p\in \bd A^2_+$ and $q\in \bd A^2_-$.  
	\item [(e)] $A^2$ is the (image of) an acceptable band, and the functions in (\ref{D:acceptable}\?(i)) corresponding to $A^2$ are $\theta^2_{\pm}$. Moreover, 
	\begin{equation}\label{E:thetas}
		0\leq \theta^2_+\leq \min\{R+\tfrac{1}{2},\theta^1_+\}\text{\quad and \quad}-R\leq \theta^2_-=\theta^1_-\leq 0.
	\end{equation}
\end{enumerate}

The inclusion $\bd A^1_-\subs \Ss^1_\nu\times [-R,0]$ implies, firstly, that 
\begin{equation}\label{E:fcons}
	A^2\cap \big(\Ss^1_\nu\times [-R,0])=A^1\cap \big(\Ss^1_\nu\times [-R,0]),
\end{equation}
as every point of $A^1\cap \big(\Ss^1_\nu\times [-R,0])$ lies at a  distance less than or equal to $R$ from $\bd A^1_-$. Secondly, it implies that 
\begin{equation*}
	t\mapsto d(\mu_z(t),\bd A^1_-)
\end{equation*} is a monotone decreasing function of $t$ when $t\geq 0$. 

It follows from \eqref{E:fcons} and the properties of $A^1$ that, for any $z\in \Ss^1_\nu$, there exists a unique $\theta^2_-(z)\in [-R,0]$ such that $\mu_z(\theta^2_-(z))\in \bd A^2$, unless $\mu_z(0)\in \bd A^1_+$. In the latter case, $d(\mu_z(0),\bd A^1_-)=R$, $\theta^2_-(z)=-R$ and  $\theta^2_+(z)=0$. If $\mu_z(0)\nin \bd A^1_+$, let $\theta^2_+(z)>0$ be the smallest $t\in (0,R]$ such that either $\mu_z(t)\in \bd A^1_+$ or $d(\mu_z(t),\bd A^1_-)=R+\frac{1}{2}$. Suppose $\mu_z(\theta^2_+(z))\in \bd A^1_+$. Then  $\mu_z(\theta^2_+(z))\in A^2$ (because it lies a distance $\leq R+\frac{1}{2}$ from $\bd A_-^1$), while $\mu_z(t)\nin A^1\sups A^2$ for $t>\theta^2_+(z)$. Thus, $\mu_z(\theta^2_+(z))\in \bd A^2$. If $d(\mu_z(\theta^2_+(z)),\bd A^1_-)=R+\frac{1}{2}$, then again $\mu_z(\theta^2_+(z))\in A^2$ while $\mu_z(t)\nin A^2$ for $t>\theta^2_+(z)$, since, for such $t$, $d(\mu_z(t),\bd A^1_-)>R+\frac{1}{2}$ by the second consequence. Moreover, in both cases $\mu_z(t)$ does not intersect $\bd A^2$ again for $t>0$. This proves (a), (b), (c) and also establishes \eqref{E:thetas}.

Since 
\[
\bd A^2=\bcup_{z\in \Ss^1_\nu}\mu_z\cap \bd A^2,
\]
(a) implies the first assertion of (d). In turn, (b) and (c) together immediately imply that 
\[
R\leq d(p,\bd A^2_-)=d(p,\bd A^1_-)\leq R+\frac{1}{2}
\]
for any $p\in \bd A^2_+$. That $d(q,\bd A^2_+)\leq d(q,\bd A^1_+)$ for any $q\in \bd A^2_-$ follows from the fact that $\bd A^2_+$ lies below $\bd A^1_+$, in the sense that any geodesic joining $\bd A^1_-$ to $\bd A^1_+$ must first intersect a point of $\bd A^2_+$. Indeed, $\theta^2_+(z)\leq \theta^1_+(z)$ for any $z\in \Ss^1_\nu$, as we have already seen in \eqref{E:thetas}. Thus, (d) holds.

By construction, 
\begin{equation*}
	A^2=\set{p\in \Ss^2_\nu\equiv \Ss^1_\nu\times (-\tfrac{\pi}{2},\tfrac{\pi}{2})}{p=(z,\theta)\text{\ for some\ }\theta\in [\theta^2_-(z),\theta^2_+(z)]}.
\end{equation*}
Hence, $A^2$ is the image of the acceptable band given by 
\begin{equation*}
	(t,u)\mapsto \big(\exp(2\pi \nu it)\?,\?(1-u)\theta^2_-(t)+u\theta^2_+(t)\big)\quad (t,\,u\in [0,1]).
\end{equation*}

Using induction and the corresponding versions of items (a)--(e) (whose proofs are the same in the general case), define
\begin{equation*}
\qquad	A^{n+1}=\set{p\in A^n}{d(p,\bd A^n_{(-1)^n})\leq R+2^{-n}}\qquad (n\in \N).
\end{equation*} 
Finally, let $B=\bcap_{n=1}^{+\infty}A^n$. We claim that $B$ is the image of a good band. 

Given $N\in \N$ and $m,n> N$, we have 
\begin{equation*}
	\abs{\theta^n_{\pm}(z)-\theta^m_{\pm}(z)}\leq 2^{-N+1}\text{\ \ for any $z\in \Ss^1_\nu$}
\end{equation*}
by construction. Therefore, $\theta^n_{+}\decr \theta_+$ and $\theta^n_-\incr \theta_-$ for some functions $\theta_{\pm}\colon \Ss^1_\nu\to [-R,R]$, which are continuous as the uniform limit of continuous functions. Moreover, $B$ is the image of the map
\begin{equation*}
	(t,u)\mapsto \big(\exp(2\pi \nu it)\?,\?(1-u)\theta_-(t)+u\theta_+(t)\big)\quad (t,\,u\in [0,1]),
\end{equation*}
again by construction. We claim that $d(x,\bd B_{\pm})=R$ for any $x\in \bd B_{\mp}$. Suppose for a contradiction that $d(p,\bd B_-)<R$ for some $p\in \bd B_+$, and let $pq$ be a geodesic of length $d(p,\bd B_-)$, with $q\in \bd B_-$. Choose neighborhoods $U\ni p$ and $V\ni q$ such that $d(x,y)<R$ for any $x\in U$, $y\in V$. Since $p,q\in \bd B_{\pm}$, by choosing a sufficiently large $n\in \N$, we may find $x\in \bd A^n_+\cap U$ and $y\in \bd A^n_-\cap V$ with $d(x,y)<R$, a contradiction. Similarly, if $d(p,\bd B_-)=R+\eps$ for some $\eps>0$, choose neighborhoods $U\ni p$ and $V\ni q$ such that $d(x,y)\geq R+\frac{\eps}{2}$ for any $x\in U$ and $V\ni q$. Let $N\in \N$ be so large that $2^{-N}<\frac{\eps}{2}$. Since $p,q\in \bd B_{\pm}$, we may find some $n>2N$ and $x\in \bd A^n_+\cap U$, $y\in \bd A^n_-\cap V$. Then $d(x,y)\geq R+ \frac{\eps}{2}>R+2^{-N}$, again a contradiction. The assumption that $d(q,\bd B_+)\neq R$ for some $q\in \bd B_-$ also yields a contradiction. We conclude that $B$ is a good band of width $R$.

If $r\colon \sr A\to \sr G$ is the map which associates to an acceptable band $A$ the good band $B$ obtained by the process described above, then $r(A)=A$ whenever $A\in \sr G$. In addition, we see by induction that the map $A\mapsto A^n$ is continuous on $\sr A$ for every $n\in \N$. Given $\eps>0$, we can arrange that $\norm{A^n-A^m}_{C^0}<\eps$ for any $A\in \sr A$ by choosing $m,\,n\geq N$ and a sufficiently large $N\in \N$. Hence, $r\colon \sr A\to \sr G$ is a retraction.
\end{proof}

\begin{cor}\label{C:contratil}
	The space $\sr G$ is contractible.
\end{cor}
\begin{proof}
	This is an immediate consequence of (\ref{L:contratil}) and (\ref{L:retrato}).
\end{proof}

\begin{defn}
Let $B$ be a good band of width $R$. A \tdef{track} of $B$ is a curve on $\Ss^2_\nu$ of length $R$ joining a point of $\bd B_+$ to a point of $\bd B_-$. 
\end{defn}
In other words, a track is a length-minimizing geodesic joining $\bd B_+$ to $\bd B_-$; in particular, it is a smooth curve. Also, if $\Ga_1$, $\Ga_2$ are tracks through $p\in \bd B_+$ and $q\in \bd B_-$ then $\Ga_1=\Ga_2$, since two geodesics on $\Ss^2$ intersect at a pair of antipodal points, and  $p$ and $q$ do not map to the same point nor to a pair of antipodal points on $\Ss^2$ under the covering map.
\begin{lemma}\label{L:nocrossing}
	Let $B$ be a good band. Then two tracks of $B$ cannot intersect at a point lying in $\ring B$.
\end{lemma}
\begin{proof}
	Suppose for the sake of obtaining a contradiction that two tracks $p_1q_1$ and $p_2q_2$, with $p_i\in \bd B_+$ and $q_i\in \bd B_-$, intersect at a point $x\in \ring B$ (see fig.~\ref{F:tracks}). Then one of the following must occur (here $ab$ denotes the segment of the corresponding geodesic and also its length):

\begin{figure}[ht]
	\begin{center}
		\includegraphics[scale=.36]{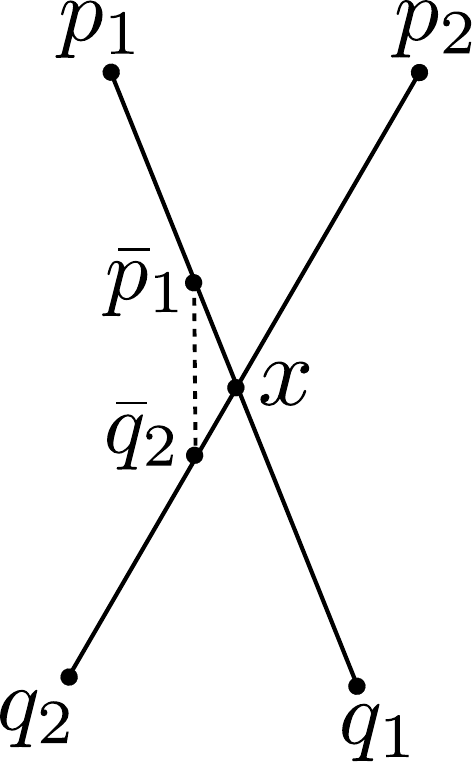}
		\caption{}
		\label{F:tracks}
	\end{center}
\end{figure}

\begin{enumerate}
	\item [(i)] $xq_1=xq_2$; 
	\item [(ii)]  $xq_1>xq_2$;
	\item [(iii)] $xq_1<xq_2$.
\end{enumerate}

If (i) holds, let $\bar{p}_1,\bar{q}_2$ be points on $p_1x$ and $xq_2$, respectively, which lie in a normal neighborhood of $x$. Then, by the triangle inequality,
\begin{equation*}
	R=p_1q_1=p_1x+xq_2>p_1\bar{p}_1+\bar{p}_1\bar{q}_2+\bar{q}_2q_2.
\end{equation*}
This contradicts the fact that $B$ is a good band of width $R$.

If (ii) holds then $R=p_1q_1>p_1x+xq_2$. Again, this contradicts the fact that $p_1q_1$ is a path of minimal length joining $p_1$ to $\bd B_-$. Similarly, if (iii) holds then $R=p_2q_2>p_2x+xq_1$, contradicting the fact that $p_2q_2$ is a path of minimal length joining $p_2$ to $\bd B_-$.
\end{proof}

\begin{urem}
	Note that this result may be false for an acceptable band. In the proof, we have implicitly used the fact that if $pq$ is a path of minimal length joining $p\in \bd B_+$ to $\bd B_-$ then $pq$ is also a path of minimal length joining $q$ to $\bd B_+$, and this is not necessarily true for an acceptable band.
\end{urem}

\begin{lemma}\label{L:uniquetrack}
	Every point in the interior of a good band $B$ lies in a unique track of $B$.
\end{lemma}
\begin{figure}[ht]
	\begin{center}
		\includegraphics[scale=.32]{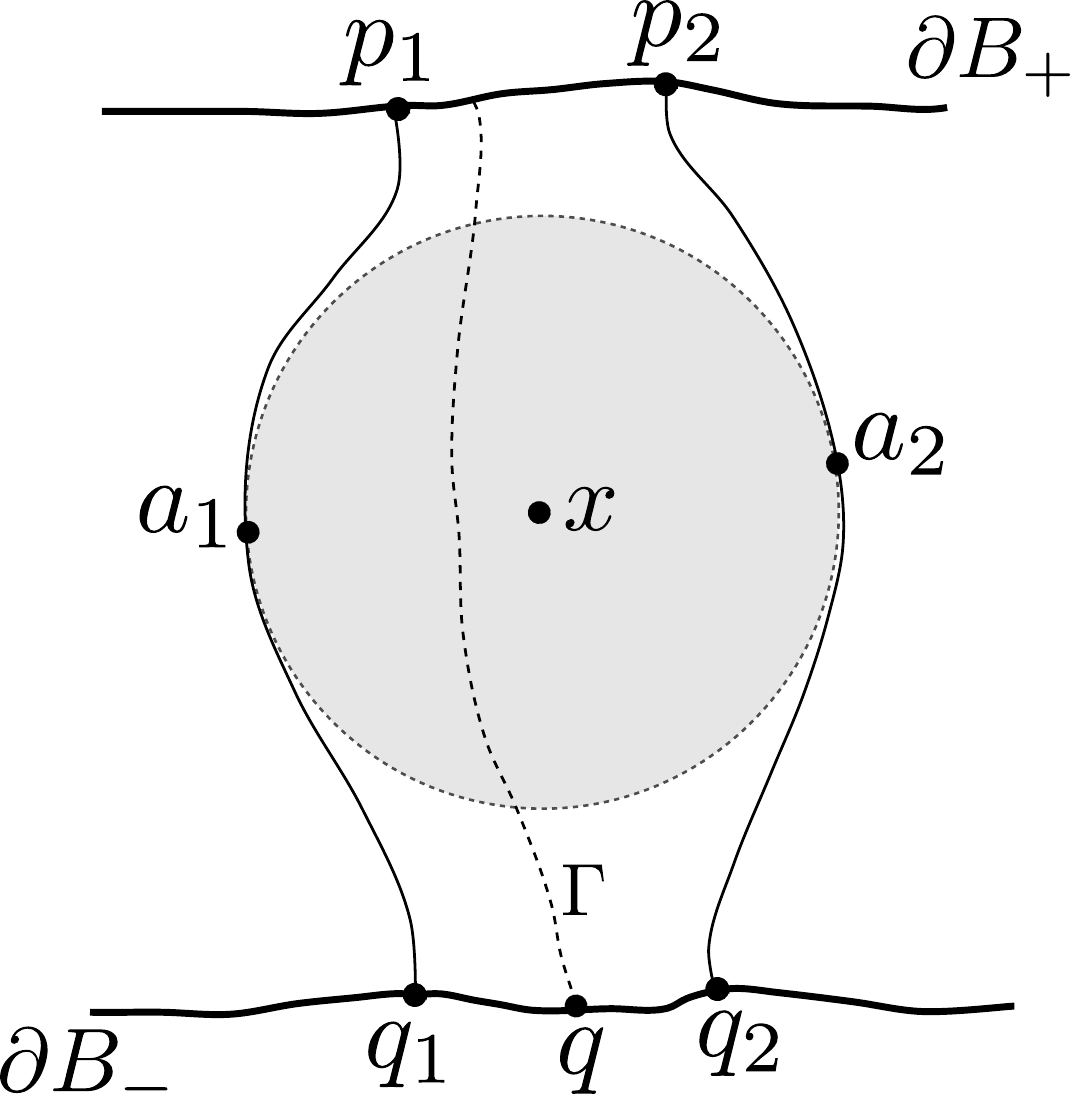}
		\caption{}
		\label{F:tracks2}
	\end{center}
\end{figure}
\begin{proof}
	Let $R$ be the width of $B$ and let $T\subs \Im(B)$ consist of all points which lie on some track of $B$. It is clear from the definitions that $\bd B_{\pm}\subs T$. We claim that $a\in T$ if and only if 
	\begin{equation}\label{E:sum}
		d(a,\bd B_+)+d(a,\bd B_-)=R
	\end{equation}
	The existence of a track through $a$ implies that $d(a,\bd B_+)+d(a,\bd B_-)\leq R$. If the inequality were strict, then there would exist a path of length less than $R$ joining $\bd B_+$ to $\bd B_-$, which is impossible. Conversely, suppose \eqref{E:sum} holds, and let $p\in \bd B_+$, $q\in \bd B_-$ be the points of $\bd B_+$ (resp.~$\bd B_-$) which are closest to $a$. Then the concatenation of the geodesics $pa$ and $aq$ is a path of length $R$ joining $\bd B_+$ to $\bd B_-$, i.e., a track. Hence, $a\in T$. 
	
	The characterization of $T$ that we have established implies that the latter is a closed set. Now suppose that $x\nin T$, let $V$ be the component of $\ring B\ssm T$ containing $x$ (see fig.~\ref{F:tracks2}, where $V$ is depicted as a gray open ball). Since $T$ is closed, any point in $\bd V$ lies in $T$. Choose points $a_1,a_2\in \bd V\ssm (\bd B_+\cup \bd B_-)$ such that the (unique) tracks $p_iq_i$ through $a_i$ do not coincide, where $p_i\in \bd B_+$ and $q_i\in \bd B_-$ ($i=1,2$). Such points $a_i$ exist because otherwise $V= \ring B$, which is absurd since any point on a track lies in $T$. Because the tracks are distinct, at least one of $p_1\neq p_2$ or $q_1\neq q_2$ must hold. Assume without loss of generality that $q_1\neq q_2$, and let $q\in \bd B_-$ be such that it is possible to join $q$ to $x$ in $\Im(B)$ without crossing $p_1q_1$ nor $p_2q_2$. Let $\Ga$ be a track through $q$. Then $\Ga$ joins $q$ to $\bd B_+$, but it does not intersect $p_1q_1$ nor $p_2q_2$ by  (\ref{L:nocrossing}). It follows that $\Ga$ must contain points of $V$, a contradiction which shows that  $T=\Im(B)$. In other words,  every point of $\Im(B)$ lies in a track of $B$; uniqueness has already been established in (\ref{L:nocrossing}).
\end{proof}

\begin{lemma}\label{L:central}
	Let $B$ be a good band of width $R$ and let $0<r<R$. Then the set $\ga_r$ consisting of all those points in $\ring B$ at distance $r$ from $\bd B_+$ is (the image of) a closed admissible curve whose radius of curvature $\rho$ satisfies $r\leq \rho\leq \pi-R+r$ almost everywhere.
\end{lemma}

\begin{proof}
	For $p\in \ring B$, let $\Ga_p\colon [0,R]\to \Ss^2_\nu$ denote the unique track through $p$, parametrized by arc-length, with $\Ga_p(0)\in \bd B_-$ and $\Ga_p(R)\in \bd B_+$. Define vector fields $\no$ and $\ta$ on $\ring B$ by letting $\no(p)$ be the unit tangent vector to $\Ga_p$ at $p$ and $\ta(p)=\no(p)\times p$. We claim that the restriction of $\no$ (and consequently that of $\ta$) to any compact subset $K$ of $\ring B$ satisfies a Lipschitz condition. Let $d_0<\min\{d(K,\bd B_+)\?,\?d(K,\bd B_-)\}$, let $a_0,a_1\in K$, with $a_1$ close to $a_0$, and consider the (spherical) triangle having $\Ga_{a_0}$, $\Ga_{a_1}$, $a_0a_1$ as sides and $a_0$, $a_1$, $a_2$ as vertices (see fig.~\ref{F:tracks3}). The point $a_2$ must lie outside of $\ring B$ by (\ref{L:nocrossing}). Let $p_0$ be the point where the geodesic segment $a_0a_2$ intersects $\bd B_{\pm}$. Then
	\[
	a_0a_2\geq a_0p_0\geq  d_0.
	\]
	 Hence, by the law of sines (for spherical triangles) applied to $\tri a_0a_1a_2$,
	\begin{equation*}
		\frac{\sin a_2	}{\sin(a_0a_1)}=\frac{\sin{a_1}}{\sin(a_0a_2)}\leq \frac{1}{\sin d_0},
	\end{equation*}
	Using parallel transport we may compare
	\begin{equation*}
		\frac{\angle (\no(a_0),\no(a_1))}{a_0a_1}\text{\quad with \quad} \frac{\san a_2}{a_0a_1}\approx \frac{\sin a_2}{\sin(a_0a_1)}
	\end{equation*} 
	to obtain a Lipschitz condition satisfied by the former, but we omit the computations.
	\begin{figure}[ht]
		\begin{center}
			\includegraphics[scale=.30]{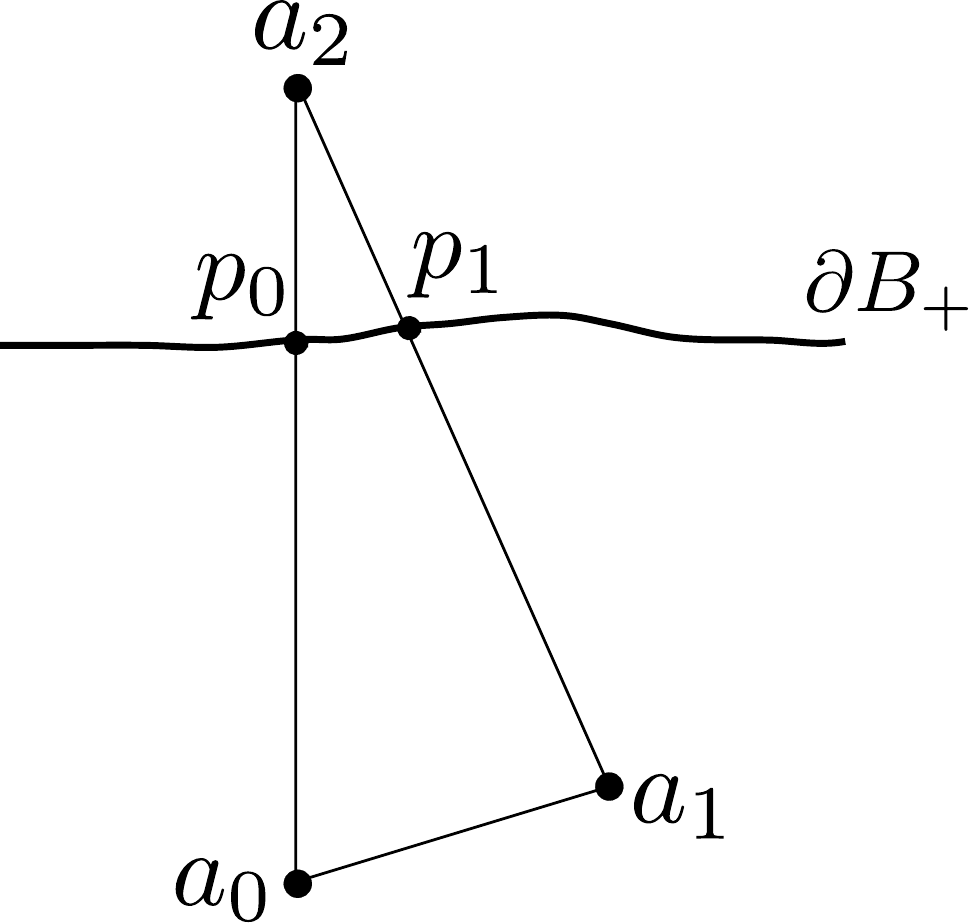}
			\caption{}
			\label{F:tracks3}
		\end{center}
	\end{figure}
	
	Now given $p\in \ring B$ at distance $r$ from $\bd B_+$, $0<r<R$, let $\ga_r$ be the integral curve through $p$ of the vector field $\ta$. Then $\ga_r$ is parametrized by arc-length and its frame is given by
	\begin{equation*}
		\Phi_{\ga_r}(t)=\begin{pmatrix}
		| & | & | \\
		\ga_r(t) & \ta(\ga_r(t)) & \no(\ga_r(t)) \\
		| & | & |
		\end{pmatrix}
	\end{equation*}
	by construction. If $d(t)=d(\ga_r(t),\bd B_+)$ then $\dot d\equiv 0$, since $\ta(\ga_r(t))$ is orthogonal to the track through $\ga_r(t)$ for every $t$. Hence $d$ is constant, equal to $r$, and $\ga_r$ is a closed curve. Moreover, since $\ta$ and $\no$ satisfy a Lipschitz condition when restricted to the image of $\ga_r$, we see that the entries of $\Phi_{\ga_r}$ are absolutely continuous with bounded derivative. In particular, these derivatives belong to $L^2$. We conclude that $\ga_r$ is admissible.
	
	For $r-R<\theta<r$, the curve $\ga_{r-\theta}$ is the translation of $\ga_r$ by $\theta$ (as defined on p.~\pageref{E:translation}, eq.~\eqref{E:translation}) by construction. Since $\ga_{r-\theta}$ is also regular, we deduce from \eqref{E:partialt2} in (\ref{L:band}) that the radius of curvature $\rho$ of $\ga_r$ satisfies
	\begin{equation*}
		0<\rho(t)-\theta<\pi
	\end{equation*}
	for all $t$ at which $\rho$ is defined and all $\theta$ in $(r-R,r)$. Therefore, $r\leq \rho\leq \pi-R+r$ a.e.. 
\end{proof}
	
	\begin{cor}\label{C:central}
	Let $B$ be a good band of width $R$. Then the central curve $\ga_{\frac{R}{2}}$ is an admissible curve whose radius of curvature is restricted to $\big[\frac{R}{2},\pi-\frac{R}{2}\big]$.\hfill\qed
\end{cor}

Before finally presenting a proof of (\ref{P:negativecondensed}), we extend the definition of the regular band of a curve to any space $\sr L_{\ka_1}^{\ka_2}$. 
\begin{defn}
Let $\ga\in \sr L_{\ka_1}^{\ka_2}$. The \tdef{\tup(regular\tup) band} $B_\ga$ spanned by $\ga$ is the map:
	\begin{equation*}
			B_\ga\colon [0,1]\times [\rho_1-\pi,\rho_2]\to \Ss^2,\quad B_\ga(t,\theta)=\cos \theta\, \ga(t)+\sin \theta\, \no(t).
	\end{equation*}
\end{defn}
The statement and proof of (\ref{L:band}) still hold, except for obvious modifications.
	
	\begin{proof}[Proof of (\ref{P:negativecondensed})] By (\ref{L:C^2}), we may replace $\sr L_{\ka_0}^{+\infty}(I)$ with $\sr C_{\ka_0}^{+\infty}(I)$, that is, we may assume that the curves $\ga_p=f(p)$ are of class $C^2$. Let $\rho_0=\arccot \ka_0$, 
	\begin{equation}\label{E:centranslation}
		\rho_1=\frac{\pi-\rho_0}{2},\ \ \ka_1=\cot \rho_1
	\end{equation}
	and let $\eta_p$ be the translation of $\ga_p$ by $\rho_1$ (compare (\ref{C:twoviewpoints})). Then the radius of curvature $\rho_{\eta_p}$ of $\eta_p$ satisfies $\rho_1<\rho_{\eta_p}<\pi-\rho_1$ for all $p\in K$. Since $\rho_{\eta_p}$ is continuous and $K$ is compact, there exists $\bar\rho_1\in (\rho_1,\frac{\pi}{2})$ such that 
	\begin{equation*}
\bar\rho_1<\rho_{\eta_p}<\pi-\bar\rho_1\quad\text{for all $p\in K$.}
	\end{equation*}
	In particular, the regular band of $\eta_p$ may be extended from $[0,1]\times [-\rho_1,\rho_1]$ to $[0,1]\times [-\bar\rho_1,\bar\rho_1]$, for any $p$. Consider the space $\sr G$ of good bands of width $R=2\bar\rho_1$ and the corresponding space $\sr A\sups \sr G$ of acceptable bands. 
	
	Recall that $K$ is connected by hypothesis, hence $p\mapsto \nu(\ga_p)$ is constant; let $\nu$ be the common rotation number of all the $\ga_p$. Let $B^0_p\in \sr G$ be the regular band, of width $2\bar\rho_1$, of $\eta_p$ (whose image is the same as that of the regular band of $\ga_p$), and let $B^1_p\in \sr G$ be the regular band of a geodesic $\sig$ of rotation number $\nu$, the latter being the central curve of the band. The combination of (\ref{L:regularisgood}), (\ref{C:contratil}) and  (\ref{C:central}) yields a homotopy $(s,p)\mapsto \eta^s_p$ from the map $p\mapsto \eta_p^0=\eta_p$ to the constant map $p\mapsto \eta_p^1=\sig$, where $\eta^s_p$ is the central curve of a good band $B^s_p\in \sr G$, $s\in [0,1]$, $p\in K$. Moreover, (\ref{C:central}) guarantees that the radius of curvature $\rho_{\eta_p^s}$ of $\eta_p^s$ satisfies $\bar\rho_1\leq \rho_{\eta_p^s}\leq \pi-\bar\rho_1$ for each $s\in [0,1]$ and $p\in K$. Consequently,
	\begin{equation*}
		\rho_1<\rho_{\eta_p^s}<\pi-\rho_1\text{\quad for each $s\in [0,1]$,~$p\in K$,}
	\end{equation*}
	and it follows that $(s,p)\mapsto \eta_p^s$ is a homotopy in $\sr C_{-\ka_1}^{+\ka_1}$ from the map $p\mapsto \eta_p$ to a constant map. If we let $\ga_p^s$ be the translation of $\eta_p^s$ by $-\rho_1$, then $\ga_p^0$ is the original curve $\ga_p=f(p)$ for each $p$, and $(s,p)\mapsto \ga_p^s$ is a homotopy in $\sr C_{\ka_0}^{+\infty}$ from $f$ to the constant map $p\mapsto \bar\sig$, where $\bar\sig$ (the translation of $\sig$ by $-\rho_1$) is a circle traversed $\nu$ times.
	
	We have proved that $f\colon K\to  \sr C_{\ka_0}^{+\infty}(I)$ is null-homotopic in $\sr C_{\ka_0}^{+\infty}$. The latter space may be replaced by $\sr C_{\ka_0}^{+\infty}(I)$ without altering the conclusion by the usual trick of substituting $\ga_p^s$ by $\Phi_{\ga_p^s}(0)^{-1}\ga_p^s$ ($s\in [0,1]$,~$p\in K$).
	\end{proof}
	
\begin{mthm}\label{C:contrensed}
	Let $\ka_0\in \R$, $\nu\geq 1$ and let $\sr O_\nu\subs \sr L_{\ka_0}^{+\infty}(I)$ be the subspace consisting of all condensed curves having rotation number $\nu$. Then $\sr O_\nu$ is weakly contractible.
\end{mthm}
\begin{proof}
	This was established in (\ref{P:positivecondensed}) for $\ka_0\geq 0$ and in (\ref{P:negativecondensed}) for $\ka_0<0$. 
\end{proof}
\begin{urem}
	For $\ka_0=-\infty$, the only condensed curves in $\sr L_{\ka_0}^{+\infty}$ are geodesic circles (and so $\sr O_\nu$ has only one element). In any case, the topology of $\sr L_{-\infty}^{+\infty}$ is well understood, and this is the main reason why we always assume that $\ka_0\in \R$.
\end{urem}


\section{Non-diffuse Curves}
In this section we define a notion of rotation number for any non-diffuse curve in $\sr L_{\ka_0}^{+\infty}$ and establish a bound on the total curvature of such a curve which depends only on its rotation number and $\ka_0$ (prop.~(\ref{P:totbound})). The rotation number defined here is more abstract than the one in \S\ref{S:condensed} (it is given by the number of sheets of a certain covering map). However, as we shall prove right after its definition, the two notions agree when the curve is both condensed and non-diffuse. 
\begin{lemma}\label{L:connectedness}
Suppose $X$ is a connected, locally connected topological space and $C\neq \emptyset$ is a closed connected subspace. Let $\Du_{\al \in J}B_\al$ be the decomposition of $X\ssm C$ into connected components. Then:
\begin{enumerate}
	\item [(a)] $\bd B_\al\subs C$ for all $\al\in J$.
	\item [(b)]  For any $J_0\subs J$, the union $C\cup \bcup_{\be \in J_0}B_\be$ is also connected.
\end{enumerate} 
\end{lemma}
\begin{proof} The proof is not difficult, and will be omitted. See \cite{tese}, (7.1) for the details.\end{proof}

We will also need the following well-known results.\footnote{Part (b) of (\ref{T:Schoenflies}) is an immediate corollary of the Riemann mapping theorem and part (c) is the 2-dimensional case of the annulus theorem.}

\begin{thm}\label{T:Schoenflies}Let $A\subs \Ss^2$ be a connected open set.
	\begin{enumerate}
		\item [(a)] $A$ is simply-connected if and only if $\Ss^2\ssm A$ is connected.
		\item [(b)] If $A$ is simply-connected and $\Ss^2\ssm A\neq \emptyset$, then $A$ is homeomorphic to an open disk.
		\item[(c)] Let $S_{\pm}\subs \Ss^2$ be disjoint and homeomorphic to $\Ss^1$. Then the closure of the region bounded by $S_{-}$ and $S_+$ is homeomorphic to $\Ss^1\times [-1,1]$.\qed
	\end{enumerate}
\end{thm}

\begin{lemma}\label{L:annulus}
	Let $U_{\pm}\subs  \Ss^2$ be homeomorphic to open disks, $U_-\cup U_+=\Ss^2$. Then \[
	U_-\cap U_+\home \Ss^1\times (-1,1).\]
\end{lemma}
\begin{proof} We first make two claims:
\begin{enumerate}
	\item [(a)] Suppose $C\home \Ss^1\times [-1,1]$ and $h\colon \bd C_-\to \Ss^1\times \se{-1}$ is a homeomorphism, where $\bd C_-$ is one of the boundary circles of $C$. Then $h$ may be extended to a homeomorphism $H\colon C\to \Ss^1\times [-1,1]$.
	\item [(b)]  Let $M$ be a tower of cylinders, in the sense that:
	\begin{enumerate}
		\item [(i)]  $M_i\home \Ss^1\times [-1,1]$ for each $i\in \Z$;
		\item [(ii)] $M=\bcup_{i\in \Z}M_i$ and $M$ has the weak topology determined by the $M_i$;
		\item [(iii)] $M_i\cap M_j=\emptyset$ for $j\neq i\pm 1$ and $M_{i}\cap M_{i+1}=S_{i}^+=S_{i+1}^-$, where $S_i^{\pm}$ are the boundary circles of $M_i$.
	\end{enumerate} Then $M\home \Ss^1\times (-1,1)$.
\end{enumerate}
Claim (a) is obviously true if $C=\Ss^1\times [-1,1]$: Just set $H(z,t)=(h(z),t)$. In the general case let $F\colon C\to \Ss^1\times [-1,1]$ be a homeomorphism. Note that $\bd C$ is well-defined as the inverse image of $\Ss^1\times \se{\pm 1}$ ($p\in \bd C$ if and only if $U\ssm \se{p}$ is contractible whenever $U$ is a sufficiently small neighborhood of $p$). Hence $\bd C$ consists of two topological circles, $\bd C_{\pm}=F^{-1}\big(\Ss^1\times \se{\pm 1}\big)$. Let $f=F|_{\bd C_-}$ and let $g=h\circ f^{-1}\colon \Ss^1\to \Ss^1$. As we have just seen, we can extend $g$ to a self-homeomorphism $G$ of $\Ss^1\times [-1,1]$. Now define $H\colon C\to \Ss^1\times [-1,1]$ by $H=G\circ F$. Then $H|_{\bd C_-}=g\circ f=h$, as desired.

To prove claim (b), let $H_0\colon M_0\to \Ss^1\times [-\frac{1}{2},\frac{1}{2}]$ be any homeomorphism. By applying (a) to $M_{\pm 1}$ and $h_{\pm 1}=H_0|_{S_0^{\pm}}$, we can extend $H_0$ to a homeomorphism 
\[
H_1\colon M_0\cup M_{\pm 1}\to \Ss^1\times \Big[{-}\tfrac{2}{3},\tfrac{2}{3}\Big],
\]
and, inductively, to a homeomorphism 
\begin{equation*}
\qquad	H_k\colon \bcup_{\abs{i}\leq k} M_i\to \Ss^1\times \Big[-1+\frac{1}{k+2},1-\frac{1}{k+2}\Big]\qquad (k\in \N).
\end{equation*}
Finally, let $H\colon M\to \Ss^1\times (-1,1)$ be defined by $H(p)=H_i(p)$ if $p\in M_i$. Then $H$ is bijective, continuous and proper, so it is the desired homeomorphism.

Returning to the statement of the lemma, note first that $\bd U_{\pm}\subs U_{\mp}$. Indeed, if $p\in \bd U_-\cap(\Ss^2\ssm U_+)$ then $p\nin U_-\cup U_+=\Ss^2$, hence no such $p$ exists. Let $h_{\pm}\colon B(0;1)\to U_{\pm}$ be homeomorphisms, and define $f_{\pm}\colon [0,1)\to \R$ by 
\begin{equation*}
	f_{\pm}(r)=\sup\set{d\big(p,\bd U_{\pm}\big)}{p\in h_{\pm}(r\Ss^1)},
\end{equation*}
where $d$ denotes the distance on $\Ss^2$. We claim that $\lim_{r\to 1}f_{\pm}(r)=0$. Observe first that $f_{\pm}$ is strictly decreasing, for if $q\in h_{\pm}(r_0\Ss^1)$, $r_0<r$, then any geodesic joining $q$ to $\bd U_{\pm}$ intersects $h(r\Ss^1)$. Hence the limit exists; if it were positive, then $U_{\pm}$ would be at a positive distance from $\bd U_{\pm}$, which is absurd.

Now choose $n\in \N$ such that 
\begin{equation*}
	f_{\pm}(t)<\frac{1}{2}\min\se{d\big(\bd U_{-},\Ss^2\ssm U_+\big),d\big(\bd U_{+},\Ss^2\ssm U_-\big)}
\end{equation*} 
for any $t>1-\tfrac{1}{n}$. Set 
\begin{equation*}
	S_i=h_+\Big(\Big(1-\frac{1}{n+i}\Big)\Ss^1\Big)\text{ for $i>0$ and }	S_i=h_-\Big(\Big(1-\frac{1}{n-i}\Big)\Ss^1\Big)\text{ for $i<0$}.
\end{equation*}	 
Finally, let $M_0$ be the region of $U_-\cap U_+$ bounded by $S_{1}$ and $S_{-1}$ and, for $i>0$ (resp.~$<0$), let $M_i$ the region bounded by $S_i$ and $S_{i+1}$ (resp.~$S_{i-1}$). Using (\ref{T:Schoenflies}\?(c)) we see that $U_-\cap U_+=\bcup M_i$ is a tower of cylinders as in claim (b), and we conclude that $U_-\cap U_+\home \Ss^1\times (-1,1)$.
\end{proof}


We now return to spaces of curves.

\begin{defns}\label{D:BCD}
For fixed $\ka_0\in \R$ and $\ga\in \sr L_{\ka_0}^{+\infty}$, let $C$ denote the image of $C_\ga$ and $D=-C$. Assuming $\ga$ non-diffuse (meaning that $C\cap D=\emptyset$), let $\hat{C}$ (resp.~$\hat{D}$) be the connected component of $\Ss^2\ssm D$ containing $C$ (resp.~the component of $\Ss^2\ssm C$ containing $D$) and let $B=\hat{C}\cap \hat{D}$. 
\end{defns}
\begin{figure}[ht]
	\begin{center}
		\includegraphics[scale=.23]{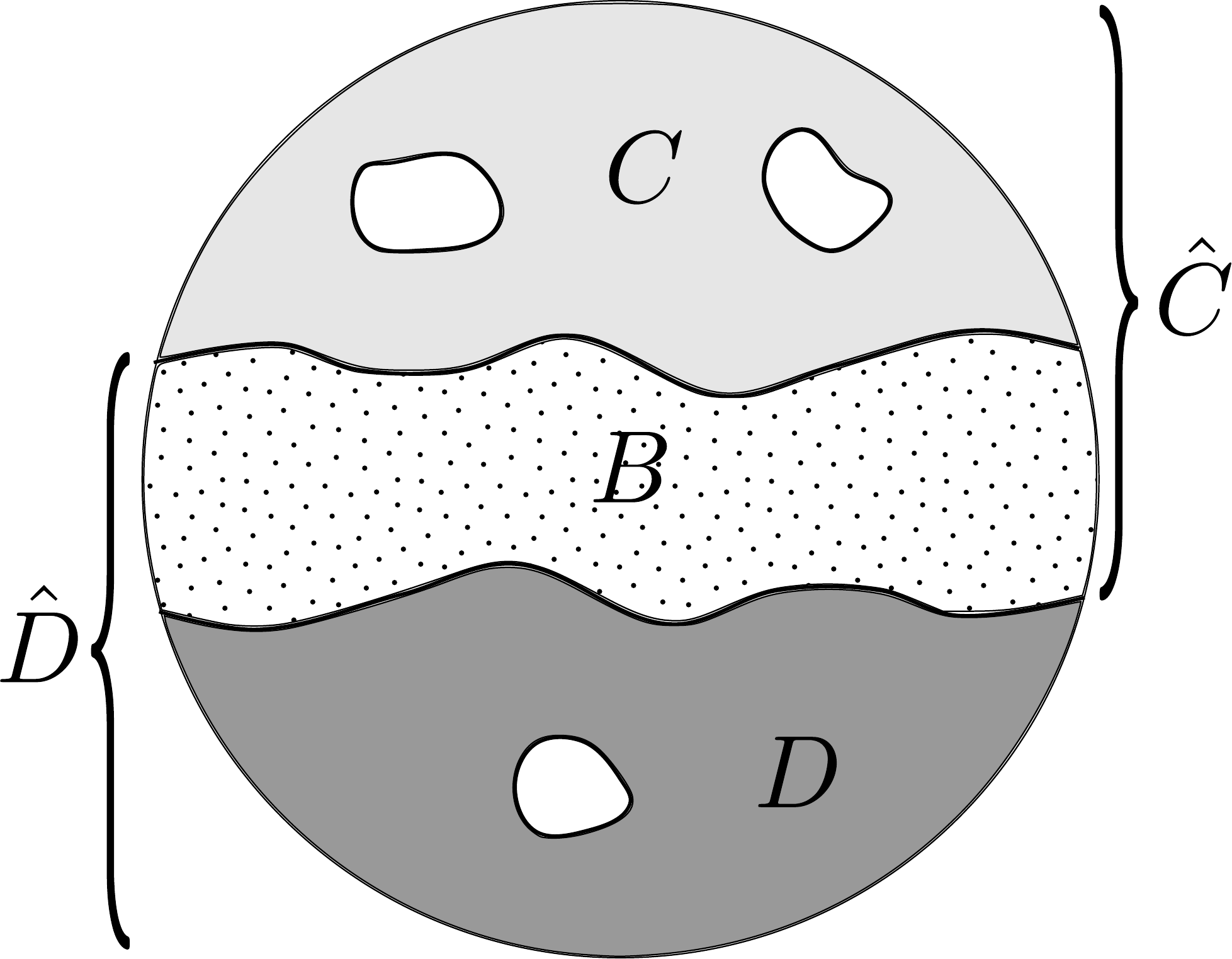}
		\caption{A sketch of the sets defined in (\ref{D:BCD}) for a non-diffuse curve $\ga\in \sr L_{\ka_0}^{+\infty}$. The lightly shaded region is $C$ and the darkly shaded region is $D={-}C$; both are closed. The dotted region represents $B$, which is homeomorphic to $\Ss^1\times (-1,1)$ by (\ref{L:Bannulus}\?(c)).}
		\label{F:bcd}
	\end{center}
\end{figure}

\begin{lemma}\label{L:Bannulus}Let the notation be as in (\ref{D:BCD}).
\begin{enumerate}
	\item [(a)] $C$ and $D$ are at a positive distance from each other.
	\item [(b)] $B\subs \Ss^2\ssm (C\cup D)$ is open and consists of all $p\in \Ss^2$ such that: there exists a path $\eta\colon [-1,1]\to \Ss^2$ with 
		\begin{equation*}
			\eta(-1)\in D, \quad\eta(1)\in C,\quad \eta(0)=p\quad\text{and}\quad\eta(-1,1)\subs \Ss^2\ssm (C\cup D).
		\end{equation*} 
	\item [(c)] The set $B$ is homeomorphic to $\Ss^1\times (-1,1)$.
\end{enumerate}
\end{lemma}
\begin{proof}
The proof of each item will be given separately. 
\begin{enumerate}
	\item [(a)] This is clear, since $C$ and $D$ are compact sets which, by hypothesis, do not intersect.
	\item [(b)] Being components of open sets,  $\hat{C}$ and $\hat{D}$ are open, hence so is $B$. 
	
	Suppose $p\in B$. Since $p\in \hat{C}$, there exists $\eta_+\colon [0,1]\to \Ss^2$ such that 
	\begin{equation*}
		\eta_+(0)=p,\quad \eta_+(1)\in C\quad\text{and}\quad \eta_+[0,1]\subs \Ss^2\ssm D.
	\end{equation*} 
	We can actually arrange that $\eta_+[0,1)\subs \Ss^2\ssm (C\cup D)$ by restricting the domain of $\eta_+$ to $[0,t_0]$, where $t_0=\inf\set{t\in [0,1]}{\eta_+(t)\in C}$ and reparametrizing; note that $t_0>0$ because $B$ is open and disjoint from $C$. Similarly, there exists $\eta_-\colon [-1,0]\to \Ss^2$ such that 
		\begin{equation*}
		\eta_-(-1)\in D,\quad \eta_-(0)=p\quad\text{and}\quad \eta_-(-1,0]\subs \Ss^2\ssm (C\cup D).
	\end{equation*} 
	Thus, $\eta=\eta_-\ast\eta_+$ satisfies all the requirements stated in (b).

	Conversely, suppose that such a path $\eta$ exists. Then $p\in \hat{C}$, for there is a path $\eta_+=\eta|_{[0,1]}$ joining $p$ to a point of $C$ while staying outside of $D$ at all times. Similarly, $p\in \hat{D}$, whence $p\in B$.
	
	\item[(c)] The set $\hat{C}$ is open and connected by definition. Its complement is also connected by (\ref{L:connectedness}\?(b)), as it consists of $D$ and the components of $\Ss^2\ssm D$ distinct from $\hat{C}$. From (\ref{T:Schoenflies}\?(a)) it follows that $\hat{C}$ is simply-connected. Further, $\hat{C}\cap D=\emptyset$, hence the complement of $\hat{C}$ is non-empty and (\ref{T:Schoenflies}\?(b)) tells us that $\hat{C}$ is homeomorphic to an open disk. By symmetry, the same is true of  $\hat{D}$. 
	
	We claim that $\hat{C}\cup \hat{D}=\Ss^2$. To see this suppose $p\nin C$, and let $A$ be the component of $\Ss^2\ssm C$ containing $p$. If $A\cap D\neq \emptyset$ then $A=\hat{D}$ by definition. Otherwise $A\cap D=\emptyset$, hence there exists a path in $\Ss^2\ssm D$ joining $p$ to $\bd A$. By (\ref{L:connectedness}\?(a)), $\bd A\subs C$, consequently $A\subs \hat{C}$. In either case, $p\in \hat{C}\cup \hat{D}$.
	
	We are thus in the setting of (\ref{L:annulus}), and the conclusion is that
	\begin{equation*}
		B=\hat{C}\cap \hat{D}\home \Ss^1\times (-1,1).\qedhere
	\end{equation*}	
\end{enumerate}
\end{proof}

In what follows let $\bd B_\ga$ be the restriction of $B_\ga$ to $[0,1]\times \se{0,\rho_0-\pi}$, let 
\begin{equation*}
	\hat{B}=\Im(B_\ga)\ssm \Im(\bd B_\ga),
\end{equation*}
and let 
\begin{equation*}
	\bar{B}_\ga\colon \Ss^1\times [\rho_0-\pi,0]\to \Ss^2
\end{equation*}
be the unique map satisfying $\bar{B}_\ga\circ (\pr\times \id)=B_\ga$, $\pr(t)=\exp(2\pi it)$. 

\begin{lemma}\label{L:coveredband}
	Let $\ka_0\in \R$ and suppose that $\ga\in \sr L_{\ka_0}^{+\infty}$ is non-diffuse. Then:
	\begin{enumerate}
		\item [(a)] For any $t\in [0,1]$, $B_\ga\big(\se{t}\times (\rho_0-\pi,0)\big)$ intersects $B$.
		\item [(b)] $B\subs \hat{B}$.
		\item [(c)] $\bar{B}_\ga^{-1}(q)$ is a finite set for any $q\in \Ss^2$ and $\bar{B}_\ga\colon \bar{B}_\ga^{-1}(\hat{B})\to \hat{B}$ is a covering map. 
	\end{enumerate} 
\end{lemma}
\begin{proof} We split the proof into parts.
\begin{enumerate}
	\item [(a)] Note first that $B_\ga(t,0)\in C$  and $B_\ga(t,\rho_0-\pi)\in D$ for any $t\in [0,1]$ by definition. Let 
	\begin{alignat*}{5}
		\theta_1=&\inf\set{\theta\in [\rho_0-\pi,0]}{B_\ga(t,\theta)\in C},\\
		\theta_0=&\sup\set{\theta\in [\rho_0-\pi,\theta_1]}{B_\ga(t,\theta)\in D}.
	\end{alignat*}	
	Then $\theta_0<\theta_1$ by (\ref{L:Bannulus}\?(a)). Let $\eta=B_\ga|_{\se{t} \times [\theta_0,\theta_1]}$. Then
	\begin{equation*}
		\eta(\theta_0)\in D,\quad \eta(\theta_1)\in C\quad \text{and}\quad \eta(\theta_0,\theta_1)\subs \Ss^2\ssm (C\cup D)
	\end{equation*}
	by construction. Therefore, any point $\eta(\theta)$ for $\theta \in (\theta_0,\theta_1)$ satisfies the characterization of $B$ given in (\ref{L:Bannulus}\?(b)), and we conclude that 
	\begin{equation*}
		B_\ga\big(\se{t}\times (\theta_0,\theta_1)\big)\subs B.
	\end{equation*}
	\item [(b)] Let $B_0=B\cap \Im(B_\ga)$. By part (a), $B_0\neq \emptyset$. Since $\Im(\bd B_\ga)\subs C\cup D$, while $B\cap (C\cup D)=\emptyset$ by definition, $B\cap \Im(\bd B_\ga)=\emptyset$. Hence,
		\begin{equation*}
		B_0=B\cap \bar{B}_\ga\big(\Ss^1\times (\rho_0-\pi,0)\big),
	\end{equation*}  
which is an open set because $\bar{B}_\ga$ is an immersion, by (\ref{L:band}\?(a)). Since $\Im(B_\ga)$ is compact, $B_0$ is also closed in $B$. But $B$ is connected by (\ref{L:Bannulus}\?(c)), consequently $B_0=B$ and $B\subs \hat{B}$.
	 
	\item [(c)] Let $q\in \Ss^2$ be arbitrary. The set $\bar{B}_\ga^{-1}(q)$ is discrete because $\bar{B}_\ga$ is an immersion, and it is compact as a closed subset of $\Ss^2$. Hence, it must be finite. Now suppose $q\in \hat{B}$. Let $\bar{B}_\ga^{-1}(q)=\se{p_i}_{i=1}^n$ and choose disjoint open sets $U_i\ni p_i$ restricted to which $\bar{B}_\ga$ is a  diffeomorphism. Let $U=\bcup_{i=1}^nU_i$ and
	\begin{equation*}
		W=\bar{B}_\ga(U_1)\cap \dots \cap \bar{B}_\ga(U_n)\ssm \bar{B}_\ga\big(\Ss^1\times [\rho_0-\pi,0]\ssm U\big).
	\end{equation*}
	Then $W$ is a distinguished neighborhood of $q$, in the sense that $\bar{B}_\ga^{-1}(W)=\Du_{i=1}^n V_i$ and $\bar{B}_\ga\colon V_i\to W$ is a diffeomorphism for each $i$, where 
	\begin{equation*}
		V_i=\bar{B}_\ga^{-1}(W)\cap U_i.\qedhere
	\end{equation*}	
\end{enumerate}
\end{proof}

Parts (b) and (c) of (\ref{L:coveredband}) allow us to introduce a useful notion which essentially counts how many times a non-diffuse curve winds around $\Ss^2$.
\begin{defn}\label{D:rotationnumber}
	Let $\ka_0\in \R$ and suppose that $\ga\in \sr L_{\ka_0}^{+\infty}$ is non-diffuse. We define the \tdef{rotation number} $\nu(\ga)$ of $\ga$ to be the number of sheets of the covering map $\bar{B}_\ga\colon \bar{B}_\ga^{-1}(B)\to B$.
\end{defn}

\begin{urem}\label{R:rotation}
Suppose now that $\ga\in \sr L_{\ka_0}^{+\infty}$ is not only non-diffuse but also condensed (meaning that $C$ is contained in a closed hemisphere). In this case, a ``more natural'' notion of the rotation number of $\ga$ is available, as described on p.~\pageref{rotation}. Let us temporarily denote by $\bar\nu(\ga)$ the latter rotation number. We claim that $\bar\nu(\ga)=\nu(\ga)$ for any condensed and non-diffuse curve $\ga$. It is easy to check that this holds whenever $\ga$ is a circle traversed a number of times. If $\ga_s$ ($s\in [0,1]$) is a continuous family of curves of this type then $\nu(\ga_s)=\nu(\ga_0)$ and $\bar{\nu}(\ga_s)=\bar{\nu}(\ga_0)$ for any $s$, since $\nu$ and $\bar\nu$ can only take on integral values and every element in their definitions depends continuously on $s$. Moreover, it follows from (\ref{C:positivecondensed}) and (\ref{P:negativecondensed}) that any condensed and non-diffuse curve is homotopic through curves of this type to a circle traversed a number of times. 
\end{urem}

\begin{prop}\label{P:totbound}
		Let $\ka_0\in \R$, $\rho_0=\arccot \ka_0$ and suppose that $\ga\in \sr L_{\ka_0}^{+\infty}$ is non-diffuse. Then 
		\begin{equation}\label{E:totbound}
			\tot(\ga)< \frac{4\pi}{\cos^2\Big(\frac{\rho_0}{2}\Big)}\,\nu(\ga).
		\end{equation}
\end{prop}

Recall that a non-convex curve in $\sr L_{0}^{+\infty}$ is not simple.
It thus follows from (\ref{L:coveredband}\?(a)) that a non-diffuse curve $\eta$
which is also non-convex must satisfy $\nu(\eta) > 1$.
The previous proposition is closely related to lemma 4.1\?(b) in \cite{Sal3}, 
which states that if $\ga$ is a closed convex curve, then $\tot(\ga)< 4\pi$.
Unlike the latter, the inequality \eqref{E:totbound} is not tight:
it implies only that $\tot(\ga)< 8\pi$
if $\ga \in \sr L_{0}^{+\infty}$ is non-diffuse with $\nu(\ga) = 1$ (note that $\rho_0=\frac{\pi}{2}$ if $\ka_0=0$).


\begin{proof}
It is easy to check that being non-diffuse is an open condition. Using (\ref{L:dense}), we deduce that the closure of the subset of all $C^2$ non-diffuse curves in $\sr L_{\ka_0}^{+\infty}$ contains the set of all (admissible) non-diffuse curves. Therefore, we lose no generality in restricting our attention to $C^2$ curves.

Let $b\in B$ be arbitrary; we have $B=-B$, hence  $-b\in B$ also. Let $\hat{\ga}$ be the other boundary curve of $B_\ga$:
\begin{equation*}
\quad	\hat{\ga}(t)=B_\ga(t,\rho_0-\pi)=-\cos\rho_0\, \ga(t)-\sin\rho_0\,\no(t)\quad (t\in [0,1]).
\end{equation*}
Then
\begin{equation}\label{E:GB1}
\quad	\hat{\ga}'(t)=\big(\ka(t)\sin\rho_0-\cos\rho_0\big)\ga'(t)=\frac{\sin(\rho_0-\rho(t))}{\sin\rho(t)}\ga'(t)\quad (t\in [0,1]).\footnote{In this proof, derivatives with respect to $t$ are denoted using a $'$ to simplify the notation.}
\end{equation}
(Here, as always, $\ka=\cot \rho$ is the geodesic curvature of $\ga$.) In particular, the unit tangent vector $\hat{\ta}$ to $\hat\ga$ satisfies $\hat{\ta}=\ta$. By (\ref{C:radiusofcurvature}), the geodesic curvature $\hat\ka$ of $\hat\ga$ is given by
\begin{equation}\label{E:GB2}
\quad	\hat\ka(t)=\cot(\rho(t)-(\rho_0-\pi))=\cot(\rho(t)-\rho_0)\quad (t\in [0,1]).
\end{equation} 
Define $h,\hat{h}\colon [0,1]\to (-1,1)$ by 
\begin{equation}\label{E:GB9}
	h(t)=\gen{\ga(t),b}\text{\quad and\quad} \hat{h}(t)=\gen{\hat{\ga}(t),b}.
\end{equation}
These functions measure the ``height'' of $\ga$ and $\hat{\ga}$ with respect to $\pm b$.  We cannot have $\abs{h(t)}=1$ nor $\vert\hat{h}(t)\vert=1$ because the images of $\ga$ and $\hat{\ga}$ are contained in $C$ and $D$ respectively, which are disjoint from $B$ (by definition (\ref{D:BCD})).  Also,
\begin{alignat}{9}\label{E:GB10}
	h'(t)=&\abs{\ga'(t)}\gen{b,\ta(t)},\quad&\quad \hat{h}'(t)&=\frac{\sin(\rho_0-\rho(t))}{\sin\rho(t)}h'(t).
\end{alignat}
Let $\Ga_t$ be the great circle whose center on $\Ss^2$ is $\ta(t)$, 
\begin{equation*}
	\Ga_t=\set{\cos \theta\, \ga(t)+\sin\theta\, \no(t)}{\theta\in [-\pi,\pi)}.
\end{equation*} 
We have $\ga(t),\,\hat{\ga}(t)\in \Ga_t$ by definition. Moreover, the following conditions are equivalent:
\begin{enumerate}
	\item [(i)] $b\in \Ga_t$.
	\item [(ii)]  $h'(t)=0$.
	\item [(iii)] $\hat{h}'(t)=0$.
	\item [(iv)] The segment $B_\ga\big(\se{t}\times (\rho_0-\pi,0)\big)$ contains either $b$ or $-b$.
\end{enumerate} 
The equivalence of the first three conditions follows from \eqref{E:GB10}. The equivalence (i)$\?\leftrightarrow\?$(iv) follows from the facts that $b\nin C\cap D$ and that $\Ga_t$ is the union of the segments $\pm B_\ga\big(\se{t}\times (\rho_0-\pi,0)\big)$ and $\pm C_\ga\big(\se{t}\times [0,\rho_0]\big)$ (see fig.~\ref{F:Gammat}, p.~\pageref{F:Gammat}). The equivalence of the last three conditions tells us that $h$ and $\hat{h}$ have exactly $2\nu(\ga)$ critical points, for each of $B_\ga^{-1}(b)$ and $B_\ga^{-1}(-b)$ has cardinality $\nu(\ga)$, by definition (\ref{D:rotationnumber}).

Suppose that $\tau$ is a critical point of $h$ and $\hat{h}$. Because $b\in \Ga_\tau\ssm (C\cup D)$, we can write
\begin{equation}\label{E:thetae}
	b=\cos \theta\, \ga(\tau)+\sin \theta\, \no(\tau)\text{, for some $\theta\in (\rho_0-\pi,0)\cup (\rho_0,\pi)$}.
\end{equation}
A straightforward calculation shows that:
\begin{equation*}
	h''(\tau)=\gen{\ga''(\tau),b}=\frac{\abs{\ga'(\tau)}^2}{\sin\rho(\tau)}\sin(\theta-\rho(\tau)).
\end{equation*} 
Using \eqref{E:thetae} and $0<\rho(\tau)<\rho_0$ we obtain that either
\begin{equation*}
	-\pi<\theta-\rho(\tau)<0\text{\ \ or\ \ }0<\theta-\rho(\tau)<\pi.
\end{equation*}
In any case, we deduce that $h''(\tau)\neq 0$. The proof that $\tau$ is a nondegenerate critical point of $\hat{h}$ is analogous: one obtains by another calculation that 
\begin{equation*}
	\hat h''(\tau)=\frac{\abs{\ga'(\tau)}^2}{\sin^2(\rho(\tau))}\sin(\rho_0-\rho(\tau))\?\sin(\theta-\rho(\tau)),
\end{equation*}
and it follows from the above inequalities that $\hat{h}''(\tau)\neq 0$. In particular, two neighboring critical points $\tau_0<\tau_1$ of $h$ (and $\hat h$) cannot be both maxima or both minima for $h$ (and $\hat{h}$). We will prove the proposition by obtaining an upper bound for $\tot\big(\ga|_{[\tau_0,\tau_1]}\big)$.

\begin{figure}[ht]
	\begin{center}
		\includegraphics[scale=.32]{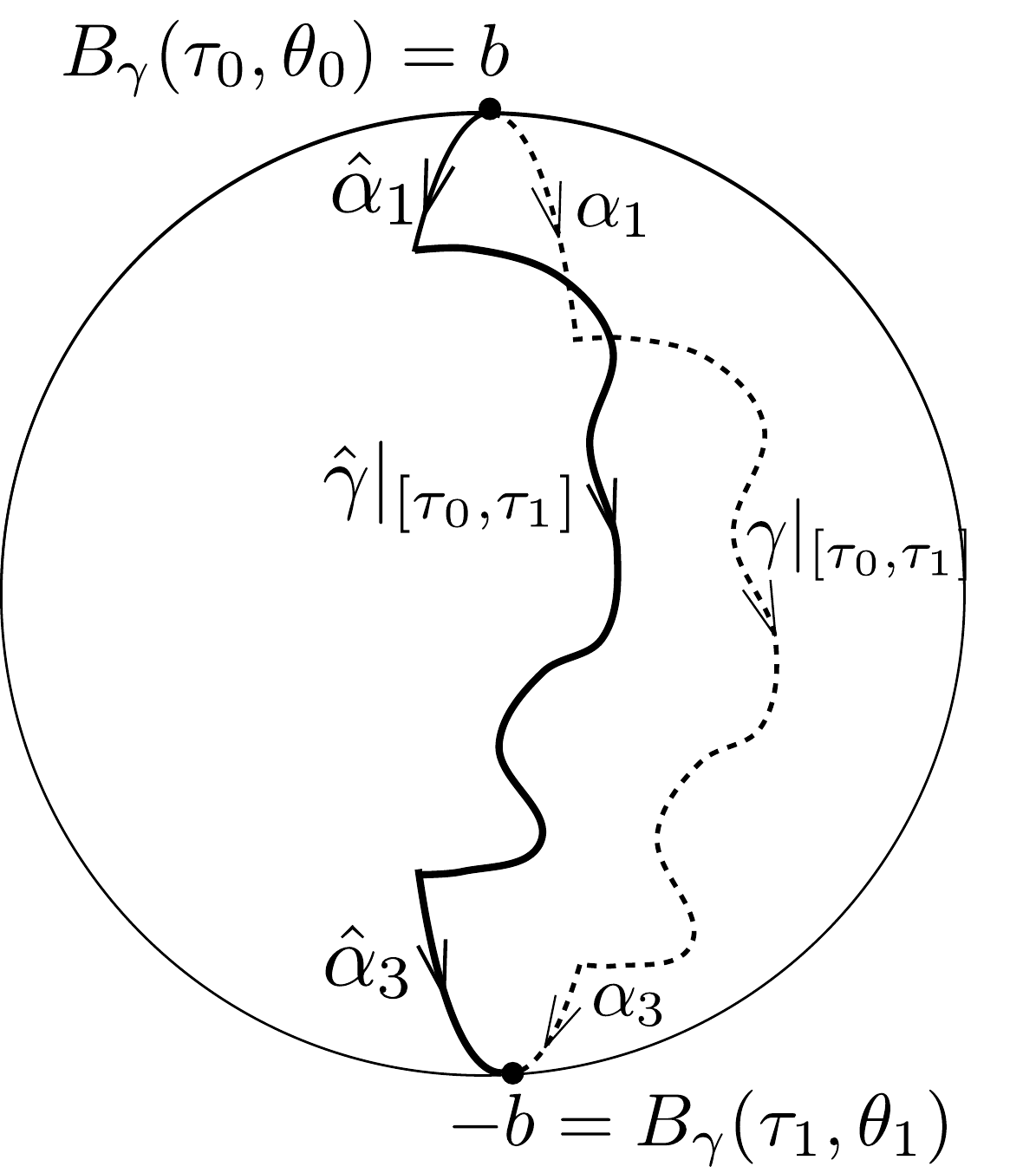}
		\caption{An illustration of the boundary of the rectangle $R=B_\ga|_{ [ \tau_0,\tau_1]\times [\rho_0-\pi,0] }$ considered in the proof of (\ref{P:totbound}).}
		\label{F:rectangle}
	\end{center}
\end{figure}

We first claim that $B_\ga|_{[\tau_0,\tau_1]\times [\rho_0-\pi,0]}$ is injective. Suppose for concreteness that $h'<0$ throughout $(\tau_0,\tau_1)$ and that $b=B_\ga(\tau_0,\theta_0)$, $-b=B_\ga(\tau_1,\theta_1)$, where $\theta_0,\theta_1\in (\rho_0-\pi,0)$.  Let $\al=\al_1\ast\al_2\ast\al_3$ be the concatenation of the curves $\al_i\colon [0,1]\to \Ss^2$ given by 
\begin{alignat*}{9}
\al_1(t)=&B_\ga\big(\tau_0\?,\?(1-t)\theta_0\big),\quad \al_2(t)=\ga\big((1-t)\tau_0+t\tau_1\big),\\
\al_3(t)=&B_\ga\big(\tau_1\?,\? t\theta_1\big),
\end{alignat*}
as sketched in fig.~\ref{F:rectangle}. Similarly, let $\hat\al$ be the concatenation of the curves $\hat\al_i\colon [0,1]\to \Ss^2$,   
\begin{alignat*}{9}
\hat\al_1(t)=&B_\ga\big(\tau_0\?,\?(1-t)\theta_0+t(\rho_0-\pi)\big),\quad \hat\al_2(t)=\hat\ga\big((1-t)\tau_0+t\tau_1\big),\\
\hat\al_3(t)=&B_\ga\big(\tau_1\?,\?(1-t)(\rho_0-\pi)+t\theta_1\big).
\end{alignat*}
Define six functions $h_i,\hat{h}_i\colon [0,1]\to [-1,1]$ by the formulas 
\begin{equation*}
\quad	h_i(t)=\gen{\al_i(t),b}\text{\quad and\quad }\hat h_i(t)=\gen{\hat\al_i(t),b}\quad (i=1,2,3).
\end{equation*}	 Note that $h_2$ is essentially the restriction of $h$ to $[\tau_0,\tau_1]$ and similarly for $\hat{h}_2$ (see \eqref{E:GB9}). Moreover, all of these functions are monotone decreasing. For $i=2$ this is immediate from \eqref{E:GB10} and the hypothesis that $h'<0$ on $(\tau_0,\tau_1)$. For $i=1,3$ this follows from the fact that $\al_i$, $\hat{\al}_i$ are geodesic arcs through $\pm b$, and our choice of orientations for these curves. 

Because the map $B_\ga|_{[\tau_0,\tau_1]\times [\rho_0-\pi,0]}$ is an immersion, if  $B_\ga$ is not injective then either $\al$ and $\hat{\al}$ intersect each other, or one of them has a self-intersection. We can discard the possibility that either curve has a self-intersection from the fact that all functions $h_i$, $\hat{h}_i$ are monotone decreasing. Further, since $B\home \Ss^1\times (-1,1)$, we can find a Jordan curve $\be\colon [0,1]\to B$  through $\pm b$ winding once around the $\Ss^1$ factor. If $\al$ and $\hat\al$ intersect (at some point other than $\al(0)=\hat{\al}(0)$ or $\al(1)=\hat{\al}(1)$), then this must be an intersection of $\ga$ and $\hat\ga$. This is impossible because $\be$, which has image in $B$, separates $C$ and $D$, which contain the images of $\ga$ and $\hat\ga$, respectively.

Thus, $R=B_\ga|_{[\tau_0,\tau_1]\times [\rho_0-\pi,0]}$ is diffeomorphic to a rectangle, and its boundary consists of $\hat\ga|_{[\tau_0,\tau_1]}$, $\ga|_{[\tau_0,\tau_1]}$ (the latter with reversed orientation) and the two geodesic arcs $B_\ga\big(\se{\tau_0}\times [\rho_0-\pi,0]\big)$ and $B_\ga\big(\se{\tau_1}\times [\rho_0-\pi,0]\big)$. Recall from (\ref{L:band}) that $\tfrac{\bd B_\ga}{\bd t}$ is always orthogonal to $\tfrac{\bd B_\ga}{\bd \theta}$. Using Gauss-Bonnet we deduce that
\begin{equation*}
	\Big(\frac{\pi}{2}+\frac{\pi}{2}+\frac{\pi}{2}+\frac{\pi}{2}\Big)+\int_{\tau_0}^{\tau_1}\hat\ka(t)\abs{\hat\ga'(t)}\,dt-\int_{\tau_0}^{\tau_1}\ka(t)\abs{\ga'(t)}\,dt+\text{Area}(R)=2\pi.
\end{equation*}
Using \eqref{E:GB1}, \eqref{E:GB2} and the fact that $\text{Area}(R)<\text{Area}(\Ss^2)=4\pi$ we obtain:
\begin{equation}\label{E:GB3}
	\int_{\tau_0}^{\tau_1}\Big(\cot\rho(t)+\frac{\sin(\rho_0-\rho(t))}{\sin\rho(t)}\cot(\rho_0-\rho(t))\Big)\abs{\ga'(t)}\,dt<4\pi.
\end{equation}
Let us see how this yields an upper bound for $\tot\big(\ga|_{[\tau_0,\tau_1]}\big)$. From $\cos(x)+\cos(y)=2\cos\big(\frac{x+y}{2}\big)\cos\big(\frac{x-y}{2}\big)$ and $\abs{\rho(t)-\frac{\rho_0}{2}}<\frac{\rho_0}{2}$ we deduce that 
\begin{alignat*}{9}
	&\sin\rho(t)\Big(\cot\rho(t)+\frac{\sin(\rho_0-\rho(t))}{\sin\rho(t)}\cot(\rho_0-\rho(t))\Big)\\
	=&\cos\rho(t)+\cos(\rho_0-\rho(t))=2\cos \Big( \frac{\rho_0}{2} \Big)\cos \Big(\rho(t)- \frac{\rho_0}{2} \Big)\geq 2\cos^2 \Big( \frac{\rho_0}{2} \Big).
\end{alignat*}
The Euclidean curvature $K$ of $\ga$ thus satisfies
\begin{alignat}{9}
	K(t)&=\sqrt{1+\ka(t)^2}=\sqrt{1+\cot\rho(t)^2}=\csc \rho(t) \label{E:GB4}
	\\ &\leq\frac{1}{2\cos^2 \big( \frac{\rho_0}{2} \big)}\Big(\cot\rho(t)+\frac{\sin(\rho_0-\rho(t))}{\sin\rho(t)}\cot(\rho_0-\rho(t))\Big). \notag
\end{alignat}	
Combining \eqref{E:GB3} and \eqref{E:GB4} we obtain:
\begin{alignat*}{9}
	&\tot\big(\ga|_{[\tau_0,\tau_1]}\big)=\int_{\tau_0}^{\tau_1}K(t)\abs{\ga'(t)}\,dt <\frac{2\pi}{\cos^2\Big(\frac{\rho_0}{2}\Big)}.
\end{alignat*}
Extending $\ga$ to all of $\R$ by declaring it to be $1$-periodic and choosing consecutive critical points $\tau_0<\tau_1<\dots<\tau_{2\nu(\ga)-1}<\tau_{2\nu(\ga)}$, so that $\tau_{2\nu(\ga)}=\tau_0+1$, we finally conclude from the previous estimate (with $[\tau_{i-1},\tau_i]$ in place of $[\tau_0,\tau_1]$)  that 
\begin{equation*}
	\tot(\ga)=\sum_{i=1}^{2\nu(\ga)}\tot\big(\ga|_{[\tau_{i-1},\tau_i]}\big) < \frac{4\pi}{\cos^2\Big(\frac{\rho_0}{2}\Big)}\,\nu(\ga).\qedhere
\end{equation*}
\end{proof}

\section{Homotopies of Circles}\label{S:7}
Let $k\geq 1$ be an integer. The \tdef{bending of the $k$-equator} is an explicit homotopy (to be defined below) from a great circle traversed $k$ times to a great circle traversed $k+2$ times. It is an ``optimal'' homotopy of this type, in the following sense: It is possible to deform a circle traversed $k$ times into a circle traversed $k+2$ times in $\sr L_{-\ka_1}^{+\ka_1}(I)$ if and only if we may carry out the bending of the $k$-equator in this space (meaning that the absolute value of the geodesic curvature is bounded by $\ka_1$ throughout the bending). 

\begin{figure}[ht]
	\begin{center}
		\includegraphics[scale=.27]{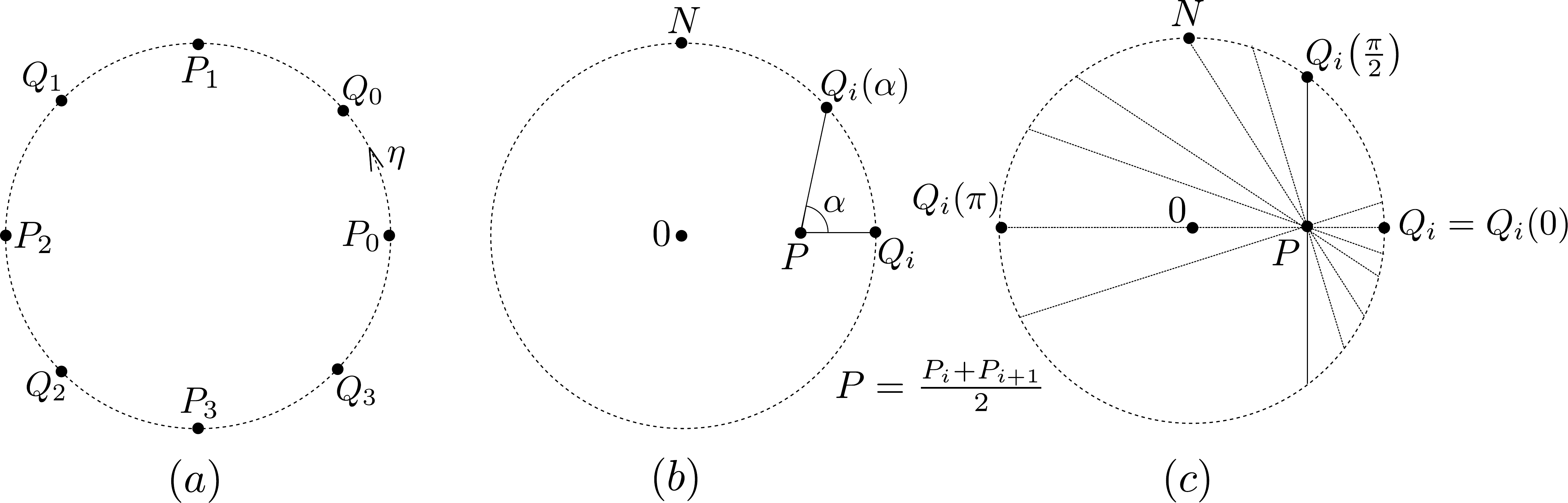}
		\caption{}
		\label{F:fourcorners}
	\end{center}
\end{figure}

Let $N=(0,0,1)\in \Ss^2$ be the north pole, let
\begin{equation*}
\qquad	\eta(t)=\big( \cos(2k\pi t),\sin(2k\pi t),0 \big)\quad (t\in [0,1])
\end{equation*}
be a parametrization of the equator traversed $k\geq 1$ times ($k\in \N$) and let 
\begin{equation*}
\qquad	P_i=\eta\Big(\frac{i}{2k+2}\Big),\quad Q_i=\eta\Big(\frac{i+\tfrac{1}{2}}{2k+2}\Big) \quad (i=0,1,\dots,2k+1),
\end{equation*}
as illustrated in fig.~\ref{F:fourcorners}\?(a) for $k=1$. Define $Q_i(\al)$ (see fig.~\ref{F:fourcorners}\?(b)) to be the unique point in the geodesic through $N$ and $Q_i$ such that 
\begin{equation*}
\qquad	\sphericalangle Q_i \Big(\frac{P_i+P_{i+1}}{2}\Big)Q_{i}(\al)=\al \quad (-\pi\leq \al\leq \pi,~i=0,1,\dots,2k+1).
\end{equation*}

Let $A_i(\al)\subs \Ss^2$ be the arc of circle through $P_iQ_i(\al)P_{i+1}$, with orientation determined by this ordering of the three points, and define
	\begin{equation*}
\qquad		\sig_{\al,i}\colon \Big[0,\frac{1}{2k+2}\Big]\to \Ss^2 \quad(0\leq \al\leq \pi,~i=0,\dots,2k+1)
	\end{equation*}
to be a parametrization of $A_i((-1)^i\al)$ by a multiple of arc-length, as illustrated in fig.~\ref{F:fourcorners2} below for $k=1$. Note that $A_i(0)$ is just $\frac{k}{2k+2}$ of the equator, while $A_i(\pi)$ is the ``complement'' of $A_i(0)$, which is $\frac{k+2}{2k+2}$ of the equator.
	
	Let $\sig_\al\colon [0,1]\to \Ss^2$ be the concatenation of all the $\sig_{\al,i}$, for $i$ increasing from $0$ to $2k+1$ (as in fig.~\ref{F:fourcorners2}). Then $\sig_0$ is the equator traversed $k$ times, while $\sig_\pi$ is the equator traversed $k+2$ times, in the opposite direction. The curve $\sig_\al$ is closed and regular for all $\al\in [0,\pi]$. However, its geodesic curvature is a step function, taking the value $(-1)^i\ka(\al)$ for $t\in (\frac{i}{2k+2},\frac{i+1}{2k+2})$, where $\ka(\al)$ depends only on $\al$. At the points $t=\frac{i}{2k+2}$ the curvature is not defined, except for $\al=0,\pi$, when the curvature vanishes identically. 
	
	We are only interested in the maximum value of $\ka(\al)$ for $0\leq \al\leq \pi$, which can be easily determined. For any $\al$, the center of the circle $C$ of which $A_i(\al)$ is an arc is contained in the plane $\Pi_1$ through $0$, $Q_i$ and $N$, since this plane is the locus of points equidistant from $P_i$ and $P_{i+1}$ ($\Pi_1$ is the plane of figures \ref{F:fourcorners}\?(b) and \ref{F:fourcorners}\?(c)). By definition, $C$ is contained in the plane $\Pi_2$ through $P_i$, $Q_i(\al)$ and $P_{i+1}$. Thus, the center of $C$ lies in the line $\Pi_1\cap \Pi_2=PQ_k(\al)$, and the segment of this line bounded by $\Ss^2$ is a diameter of $C$. Clearly, this diameter is shortest when $\al=\tfrac{\pi}{2}$ (see fig.~\ref{F:fourcorners}\?(c)). The corresponding spherical radius is $\rho=\frac{k\pi}{2k+2}$, hence \tit{the maximum value attained by $\ka(\al)$ for $0\leq \al\leq \pi$ is} 
	\begin{equation*}
		\ka(\tfrac{\pi}{2})=\cot\Big(\frac{k\pi}{2k+2}\Big)=\tan\Big(\frac{\pi}{2k+2}\Big),
	\end{equation*}
	\tit{and the minimum value is $-\ka(\tfrac{\pi}{2})$.}
	\begin{defn}\label{D:bending}
	Let $\sig_\al$ be as in the discussion above ($0\leq \al\leq \pi$) and assume that
	\begin{equation}\label{E:kappa1}
	\ka_1>\tan\Big(\frac{\pi}{2k+2}\Big).		
	\end{equation}	The \tdef{bending of the $k$-equator} is the family of curves $\eta_s\in \sr L_{-\ka_1}^{+\ka_1}(I)$ given by:
	\begin{equation*}
		\eta_s(t)=\big(\Phi_{\sig_{s\pi}}(0)\big)^{-1}\sig_{s\pi}(t)\quad (s,t\in [0,1]).
	\end{equation*}
	\end{defn}
	\begin{figure}[ht]
	\begin{center}
		\includegraphics[scale=.26]{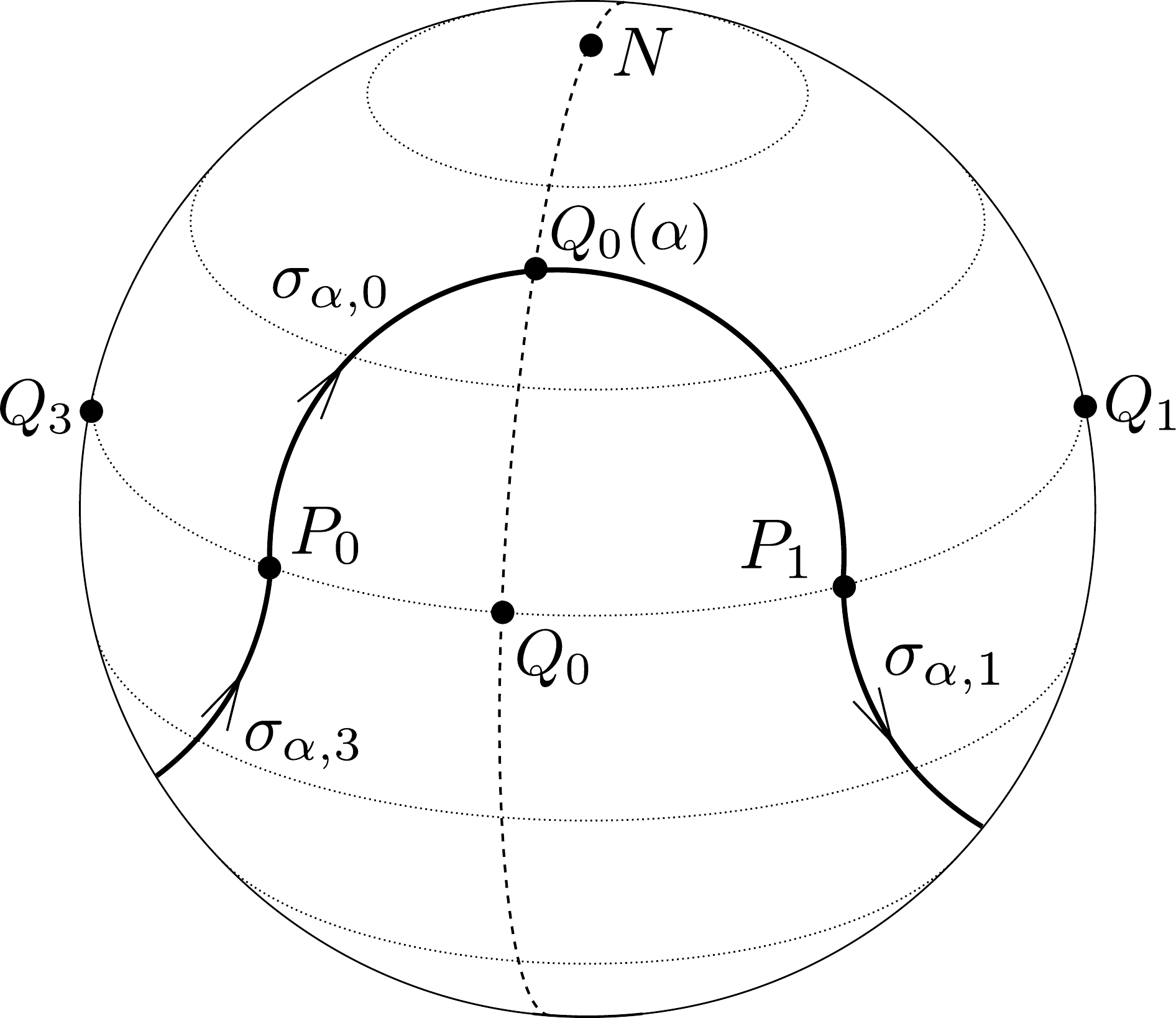}
		\caption{An illustration of the bending of the 1-equator. The curve $\sig_\al$ is the concatenation of $\sig_{\al,0},\dots,\sig_{\al,3}$.}
		\label{F:fourcorners2}
	\end{center}
\end{figure}
Note that $\eta_0$ is the equator of $\Ss^2$ traversed $k$ times and $\eta_1$ is the equator traversed $k+2$ times, in the same direction. The following result is an immediate consequence of the discussion above.
\begin{prop}\label{P:borderline0}
	Let $\ka_0=\cot \rho_0\in \R$ and let $\sig_k, \sig_{k+2}\in \sr L_{\ka_0}^{+\infty}(I)$ be circles traversed $k$ and $k+2$ times, respectively. Then $\sig_k$ lies in the same component of $\sr L_{\ka_0}^{+\infty}(I)$ as $\sig_{k+2}$ if 
	\begin{equation}
		k\geq \left\lfloor{\frac{\pi}{\rho_0}}\right\rfloor.
	\end{equation}
\end{prop}
	\begin{proof}
		Let $\rho_1=\frac{\pi-\rho_0}{2}$, so that $\ka_1=\cot\rho_1$ satisfies \eqref{E:kappa1}.  Let $\ga_s$ ($s\in [0,1]$) be the image of the bending $\eta_s$ of the $k$-equator under the homeomorphism $\sr L_{-\ka_1}^{+\ka_1}(I)\home \sr L_{\ka_0}^{+\infty}(I)$ of (\ref{C:twoviewpoints}). Then $\ga_0$ is some circle traversed $k$ times, while $\ga_1$ is a circle traversed $k+2$ times. Using (\ref{L:homocircles}) we deduce that $\sig_k\iso \ga_0\iso \ga_1\iso \sig_{k+2}$, hence $\sig_k$ and $\sig_{k+2}$ lie in the same component of $\sr L_{\ka_0}^{+\infty}(I)$.
	\end{proof}
\begin{cor}\label{C:negacircles}
	Let $\rho_i=\arccot(\ka_i)$, $i=1,2$, and suppose that $\rho_1-\rho_2>\frac{\pi}{2}$. Let $\sig_{k_0},\sig_{k_1}\in\sr L_{\ka_1}^{\ka_2}(I)$ (resp.~$\sr L_{\ka_1}^{\ka_2}$) be two parametrized circles traversed $k_0$ and $k_1$ times, respectively. Then $\sig_{k_0}$ and $\sig_{k_1}$ lie in the same connected component if and only if $k_0\equiv k_1\pmod 2$.
\end{cor}
\begin{proof}
	By (\ref{P:arbitrary}), it suffices to prove the result for $\sr L_{\ka_1}^{\ka_2}(I)$. It follows from (\ref{L:bit}) that if $\sig_{k_0}$ and $\sig_{k_1}$ lie in the same component of $\sr L_{\ka_1}^{\ka_2}(I)$, then $k_0\equiv k_1\pmod 2$. Under the homeomorphism $\sr L_{\ka_1}^{\ka_2}(I)\home \sr L_{\ka_0}^{+\infty}(I)$ of (\ref{C:belowandabove}), the condition $\rho_1-\rho_2>\frac{\pi}{2}$ translates into $\rho_0>\frac{\pi}{2}$, hence the converse is a consequence of (\ref{L:homocircles}) and (\ref{P:borderline0}).
\end{proof}

\subsection*{Homotopies of condensed curves} The previous corollary settles the question of when two circles in $\sr L_{\ka_0}^{+\infty}(I)$ lie in the same component of this space for $\ka_0<0$. Because of this, we will assume for the rest of the section that $\ka_0\geq 0$; the following proposition implies the converse to (\ref{P:borderline0}), and together with it, settles the same question in this case.
\begin{prop}\label{P:borderline}
	Let $\ka_0=\cot \rho_0\geq 0$ and let 
	\[
	n=\fl{\frac{\pi}{\rho_0}}+1.
	\]
	Suppose that $s\mapsto \ga_s\in \sr L_{\ka_0}^{+\infty}(I)$ is a homotopy, with $\ga_0$ condensed and $\nu(\ga_0)\leq n-2$ $(s\in [0,1])$. Then $\ga_s$ is condensed and $\nu(\ga_s)=\nu(\ga_0)$ for all $s\in [0,1]$.
\end{prop}

The case $\ka_0=0$ of the above proposition
is equivalent to the well-known fact (due to M.~Shapiro, cf.~\cite{Ani})
that a homotopy of a convex curve
through curves of positive curvature is actually
a homotopy through convex curves.
Indeed, it will follow from our characterization of the components of $\sr L_{0}^{+\infty}$ that
a curve $\ga$ in this space is convex if and only if
it is condensed with $\nu(\ga) = 1$.
 
Taking $\ga_0$ to be a circle $\sig_k$ traversed $k$ times for $k\leq n-2$,
we conclude that it is not possible to deform $\sig_k$ into a circle traversed $k+2$ times in $\sr L_{\ka_0}^{+\infty}$. The proof of (\ref{P:borderline}) will be broken into several parts. We start with the definition of an equatorial curve, which is just a borderline case of a condensed curve.
\begin{defn}
Let $\ka_0\geq 0$. We shall say that a curve $\ga\in \sr L_{\ka_0}^{+\infty}$ is \tdef{equatorial} if the image $C$ of its caustic band is contained in a closed hemisphere, but not in any open hemisphere. Let 
\[
H_\ga=\set{p\in \Ss^2}{\gen{p,h_\ga}\geq 0}
\]
be a closed hemisphere containing $\ga$, and let
\[
E_\ga=\set{p\in \Ss^2}{\gen{p,h_\ga}=0}
\]
denote the corresponding \tdef{equator}. Also, let $\ce \ga\colon [0,1]\to \Ss^2$ be the curve given by 
\begin{equation*}
	\ce\ga(t)=C_\ga(t,\rho_0).
\end{equation*}
\end{defn}

\begin{lemma}\label{L:cega}
	Let $\ka_0\geq 0$, let $\ga\in \sr L_{\ka_0}^{+\infty}$ be an equatorial curve of class $C^2$. Then:
	\begin{enumerate}
		\item [(a)] The hemisphere $H_\ga$ and the equator $E_\ga$ defined above are uniquely determined by $\ga$.
		\item [(b)] The geodesic curvature $\ce{\ka}$ of $\ce{\ga}$ is given by:
		\begin{equation*}\label{E:cecurvature}
				\ce\ka=\cot(\rho_0-\rho)>0.
		\end{equation*}
	\end{enumerate}
\end{lemma}
\begin{proof} Suppose that $C=\Im(C_\ga)$ is contained in distinct closed hemispheres $H_1$ and $H_2$. Then it is contained in the closed lune $H_1\cap H_2$. The boundary of $C$ is contained in the union of the images of $\ga, \ce{\ga}$, and these curves have a unit tangent vector at all points, so they cannot pass through either of the points in $E_1\cap E_2$ (where $E_i$ is the equator corresponding to $H_i$). It follows that $C$ is contained in an open hemisphere, a contradiction which establishes (a).

For part (b) we calculate:\footnote{For the rest of the section we denote derivatives with respect to $t$ by a $'$ to unclutter the notation.}
\begin{alignat}{9}\label{E:cefirst}
	\ce{\ga}'(t)&=\abs{\ga'(t)}\big(\cos\rho_0-\ka(t)\sin\rho_0\big)\ta(t) \\ \label{E:cesecond}
	\ce{\ga}''(t)&=\abs{\ga'(t)}^2\big(\cos\rho_0-\ka(t)\sin\rho_0\big)\big(-\ga(t)+\ka(t)\no(t)\big)+\la(t)\ta(t),
\end{alignat}
where $\ka$, $\ta$ and $\no$ denote the geodesic curvature of and unit and normal vectors to $\ga$, respectively, and the value of $\la(t)$ is irrelevant to us. Hence,
\begin{equation*}
	\ce\ka=\frac{\gen{\ce\ga\,,\,\ce\ga'\times \ce\ga''}}{\abs{\ce\ga'}^3}=\frac{\ka\cos\rho_0+\sin\rho_0}{\abs{\cos\rho_0-\ka\sin\rho_0}}=\frac{\cos(\rho_0-\rho)}{\abs{\sin(\rho-\rho_0)}}=\cot(\rho_0-\rho).\qedhere
\end{equation*}
\end{proof}

\begin{lemma}\label{L:borderline}
	Let $\ka_0\geq 0$ and $\ga\in \sr L_{\ka_0}^{+\infty}$ be an equatorial curve of class $C^2$. Take $N\in E_\ga$ and define $h,\ce{h}\colon [0,1]\to \R$ by 
	\begin{equation}\label{E:borderline}
		h(t)=\gen{\ga(t),N},\quad \ce h(t)=\gen{\ce\ga(t),N}.
	\end{equation}
	\begin{enumerate}
		\item [(a)] The following conditions are equivalent:
	\begin{enumerate}
	\item [(i)] $\pm N\in \Ga_\tau$ for some $\tau\in [0,1]$.
	\item [(ii)]  $\tau\in [0,1]$ is a critical point of $h$.
	\item [(iii)]  $\tau\in [0,1]$ is a critical point of $\ce h$.
\end{enumerate} 
		\item [(b)] If $\tau$ is a common critical point of $h$, $\ce h$, then $h''(\tau)\ce h''(\tau)<0$.
		\item [(c)] If $\tau<\bar\tau$ are neighboring critical points then $h''(\tau)h''(\bar\tau)<0$ and  $\ce{h}''(\tau)\ce h''(\bar\tau)<0$.
	\end{enumerate} 
\end{lemma}
Recall that $\Ga_t$ is the great circle 
\begin{equation*}
\Ga_t=\set{\cos \theta\, \ga(t)+\sin \theta\, \no(t)}{\theta\in [-\pi,\pi)}.
\end{equation*}
Part (b) implies in particular that all critical points of $h$, $\ce h$ are nondegenerate.
\begin{proof} A straightforward calculation using (\ref{E:cefirst}) shows that:
\begin{equation}\label{E:tau1}
\qquad	h'(t)=\abs{\ga'(t)}\gen{N,\ta(t)},\quad \ce{h}'(t)=\frac{\sin(\rho(t)-\rho_0)}{\sin\rho(t)}h'(t)\qquad (t\in [0,1]).
\end{equation}
The equivalence of the conditions in (a) is immediate from this and the definition of $\Ga_t$. 

From $\pm N\in E_\ga$ and $C=\Im(C_\ga)\subs H_\ga$, it follows that $\pm N\nin C \big([0,1]\times (0,\rho_0) \big)$. Thus, if $\tau$ is a critical point of $h$, $\ce h$, i.e., if $N\in \Ga_\tau$,  then we can write
\begin{equation}\label{E:range}
	N=\cos \theta\, \ga(\tau)+\sin \theta\, \no(\tau)\text{\ \ for some $\theta\in [\rho_0-\pi,0]\cup [\rho_0,\pi]$.}
\end{equation}
Another calculation, with the help of \eqref{E:cesecond}, yields:
	\begin{equation*}
		h''(\tau)=\frac{\abs{\ga'(\tau)}^2}{\sin\rho(\tau)}\sin\big(\theta-\rho(\tau)\big),\quad \ce h''(\tau)=\frac{\abs{\ga'(\tau)}^2}{\sin^2\rho(\tau)}\sin\big(\theta-\rho(\tau)\big)\sin\big(\rho(\tau)-\rho_0\big)
	\end{equation*}
	Taking the possible values for $\theta$ in \eqref{E:range} and $0<\rho(\tau)<\rho_0$ into account, we deduce that
	\begin{equation*}
		h''(\tau)\ce h''(\tau)=\frac{\abs{\ga'(\tau)}^4}{\sin^3\rho(\tau)}\sin^2\big(\theta-\rho(\tau)\big)\sin\big(\rho(\tau)-\rho_0\big)<0,
	\end{equation*}
	since all terms here are positive except for $\sin\big(\rho(\tau)-\rho_0\big)$. This proves (b).
	
	For part (c), suppose that $\tau<\bar\tau$ are neighboring critical points, but $h''(\tau)h''(\bar\tau)>0$. This means that $h'$ vanishes at $\tau,\bar\tau$ and takes opposite signs on the intervals $(\tau,\tau+\eps)$ and $(\bar\tau-\eps,\bar\tau)$ for small $\eps>0$. Hence, it must vanish somewhere in $(\tau,\bar\tau)$, a contradiction. The proof for $\ce h$ is the same. 
\end{proof}

Let $\ka_0\geq 0$, $\ga\in \sr L_{\ka_0}^{+\infty}$ be an equatorial curve and  $\pr\colon \Ss^2\to \R^2$ denote the stereographic projection from $-h_\ga$, where $H_\ga=\set{p\in \Ss^2}{\gen{p,h_\ga}\geq 0}$. As for any condensed curve, we may define a (non-unique) continuous angle function $\theta$ by the formula:
\begin{equation*}
\qquad	\exp(i\theta(t))=\ta_\eta(t),\quad \eta(t)=\pr\circ \ga(t)\quad (t\in [0,1]);
\end{equation*}
here $\ta_\eta$ is the unit tangent vector, taking values in $\Ss^1$, of the plane curve $\eta$. The function $\theta$ is strictly decreasing since $\ka_0\geq 0$, and 
\begin{equation*}
	2\pi\nu(\ga)=\theta(0)-\theta(1).
\end{equation*}
\begin{figure}[ht]
	\begin{center}
		\includegraphics[scale=.20]{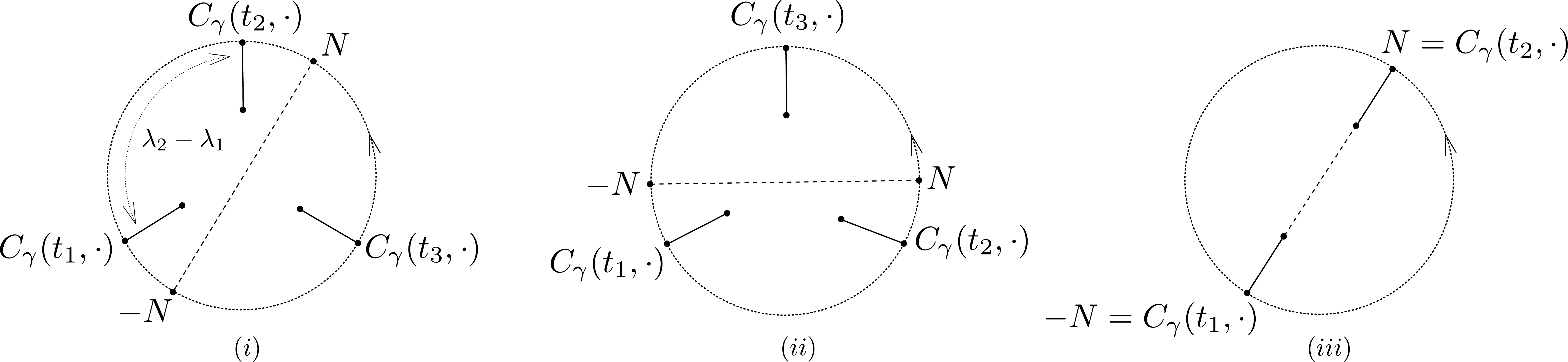}
		\caption{Three possibilities for an equatorial curve $\ga$. The circle represents $E_\ga$ and its interior represents $H_\ga$, seen from above.}
		\label{F:borderline}
	\end{center}
\end{figure}

\begin{lemma}\label{L:equatorial}
	Let $\ka_0\geq 0$, $\ga\in \sr L_{\ka_0}^{+\infty}$ be an equatorial curve of class $C^2$ and
	\begin{equation*}\label{E:n2}
				n=\left\lfloor{\frac{\pi}{\rho_0}}\right\rfloor+1.
	\end{equation*} 
	Then $\nu(\ga)\geq n-1$.
\end{lemma}

\begin{proof}
	Let $C=\Im(C_\ga)$, $H=H_\ga$ be the closed hemisphere containing $\ga$ and $E=E_\ga$ be the corresponding equator, oriented so that $H$ lies to its left. It follows from the combination of (\ref{L:basicconvex}), (\ref{L:Steinitz}) and (\ref{L:closedhemisphere}) that either we can find two antipodal points in $C\cap E$ or we can choose $t_1<t_2<t_3$ and $\theta_i\in \se{0,\rho_0}$ such that $0$ is a convex combination of the points $C_\ga(t_i,\theta_i)\in C\cap E$. There are three possibilities, as depicted in fig.~\ref{F:borderline}; the only difference between the first two is the order of the points in the orientation of $E$.

In cases (i) and (ii), choose $N$ in $E$ so that 
\[
\gen{C_\ga(t_2,\theta_2),N}=-\gen{C_\ga(t_1,\theta_1),N}>0.
\]
Let $h$ and $\ce h$ be as in \eqref{E:borderline} and define latitude functions $\la$, $\ce\la$ by
\begin{equation*}
\qquad	\la(t)=\arcsin(h(t)),\quad \ce\la(t)=\arcsin(\ce h(t))\qquad (t\in [0,1]).
\end{equation*} 
Let $\tau_1<\dots<\tau_{k_1}$ be all the common critical points of these functions in the interval $[t_1,t_2)$, and let 
\begin{equation*}
	m_j=\min\{\la(\tau_j),\ce\la(\tau_j)\},\quad M_j=\max\{\la(\tau_j),\ce\la(\tau_j)\}.
\end{equation*}	
From (\ref{L:borderline}\?(a)), we deduce that 
\begin{equation}\label{E:bl1}
\qquad	M_j-m_j=\rho_0 \text{\quad for all $j=1,\dots,k_1$},
\end{equation}
while from (\ref{L:borderline}\?(b)) and (\ref{L:borderline}\?(c)), we deduce that the $\tau_j$ are alternatingly maxima and minima of $\la$ (resp.~minima and maxima of $\ce\la$) as $j$ goes from $1$ to $k_1$, whence
	\begin{equation}\label{E:bl2}
\qquad		M_j>m_{j+1}\text{\quad for all $j=1,\dots,k_1-1$}.
	\end{equation}
	 Let 
	 \[
	 \la_2=\max\se{\la(t_2),\ce\la(t_2)}\text{\quad and \quad}\la_1=\min\{\la(t_1),\ce{\la}(t_1)\}=-\la_2.
	 \]
	Then $\la_2-\la_1$ is just the angle between $C_\ga(t_1,\cdot)\cap E$ and $C_\ga(t_2,\cdot)\cap E$ measured along $E$, as depicted in fig.~\ref{F:borderline}\?(i). For the rest of the proof we consider each case separately.
	 
 In case (i), 
 \begin{equation}\label{E:bl3}
 	m_1\leq \la_1\text{\quad and \quad}\la_2\leq M_{k_1}.
 \end{equation}
	Combining \eqref{E:bl1}, \eqref{E:bl2} and \eqref{E:bl3}, we find that 
	\begin{equation}\label{E:casei}
		k_1\rho_0=\sum_{j=1}^{k_1}(M_j-m_j)>\sum_{j=1}^{k_1-1}(m_{j+1}-m_{j})+M_{k_1}-m_{k_1}=M_{k_1}-m_1\geq \la_2-\la_1.
	\end{equation}
		Let there be $k_2$ (resp.~$k_3$) critical points of $h$, $\ce h$ in the interval $[t_2,t_3)$ (resp.~$[t_3,t_1+1)$), where for the latter we are considering $\ga$ as a 1-periodic curve. Then an analogous result to \eqref{E:casei} holds for $k_2$ and $k_3$, and summing all three inequalities we conclude that  
		 \begin{equation*}
		 	k_1+k_2+k_3>\frac{2\pi}{\rho_0}\geq 2(n-1).
		 \end{equation*}
In case (i), the number of half-turns of the tangent vector to the image of $\ga$ under stereographic projection through $-h_\ga$ in $[0,1]$ is given by $k_1+k_2+k_3-2$. Hence, 
		 \begin{equation*}
		 	\nu(\ga)=\frac{k_1+k_2+k_3-2}{2}> n-2,
		 \end{equation*}
		 as claimed. 
		 
		 In case (ii), a direct calculation using basic trigonometry shows that
		 \begin{equation*}
		 	m_1<\arcsin ( \cos \rho_0 \sin \la_1)=-\arcsin(\cos \rho_0 \sin \la_2) \text{\quad and \quad} M_{k_1}>\arcsin(\cos\rho_0\sin\la_2).
		 \end{equation*}
		 Combining this with \eqref{E:bl1} and \eqref{E:bl2}, we obtain that 
		 \begin{equation*}
		 	k_1\rho_0=\sum_{j=1}^{k_1}(M_j-m_j)>\sum_{j=1}^{k_1-1}(m_{j+1}-m_j)+M_{k_1}-m_{k_1}=M_{k_1}-m_1>2\arcsin(\cos\rho_0\sin \la_2),
		 \end{equation*}
		 and similarly for $k_2$ and $k_3$, where the latter denote the number of critical points of $h$, $\ce h$ in the intervals $[t_2,t_3)$ and $[t_3,t_1+1)$, respectively. More precisely, we have:
		 \begin{equation}\label{E:bl4}
		 	k_1+k_2+k_3>\frac{2}{\rho_0} \big[ \arcsin(\cos\rho_0\sin\la_2)+\arcsin(\cos\rho_0\sin\la_4)+\arcsin(\cos\rho_0\sin\la_6)\big],
		 \end{equation}
where $\la_4=\max\se{\la(t_3),\ce\la(t_3)}$, $\la_6=\max\se{\la(t_1),\ce\la(t_1)}$ and these latitudes are measured with respect to the chosen points $\pm N$ corresponding to each of the intervals $[t_2,t_3]$ and $[t_3,t_3+1]$.  In case (ii), the number of half-turns of the tangent vector to the image of $\ga$ under stereographic projection through $-h_\ga$ in $[0,1]$ is given by $k_1+k_2+k_3-2$. Hence, it follows from \eqref{E:bl4} and lemma (\ref{L:calculus}) below that 
		 \begin{equation*}
		 	\nu(\ga)=\frac{k_1+k_2+k_3+2}{2}>\Big(\frac{\pi}{\rho_0}-2\Big)+1\geq n-2,
		 \end{equation*}
		 as we wished to prove.
		 
		 Finally, in case (iii), we may choose $\pm N\in E\cap C$, that is, we may find $t_1<t_2$ and $\theta_i\in \se{0,\rho_0}$ such that 
		 \begin{equation*}
		 	N=C_\ga(t_2,\theta_2)=-C_\ga(t_1,\theta_1).
		 \end{equation*}
		 In this case $\la_2-\la_1=\pi$ and
		 \begin{equation*}
		 	\nu(\ga)=\frac{k_1+k_2-2}{2},
		 \end{equation*} 
		 where $k_1$ (resp. $k_2$) is the number of critical points of $h,\,\ce h$ in $[t_1,t_2]$ (resp.~$[t_2,t_1+1]$). Note that $t_1$, $t_2$ are critical points of $h$ which are counted twice in the sum $k_1+k_2$ (under the identification of $t_1$ with $t_1+1$); this is the reason why we need to subtract $2$ from $k_1+k_2$ to calculate the number of half-turns of the tangent vector.  Using \eqref{E:bl2} one more time, we deduce that 
		 \begin{equation*}
		 	\qquad k_1\rho_0=\sum_{j=1}^{k_1}(M_j-m_j)>\sum_{j=1}^{k_1-1}(m_{j+1}-m_j)+M_{k_1}-m_{k_1}=M_{k_1}-m_1=\la_2-\la_1=\pi;
		 \end{equation*}
		 similarly, $k_2\rho_0>\pi$. Therefore,
		 \begin{equation*}
		 	\nu(\ga)=\frac{k_1+k_2-2}{2}>\frac{\pi}{\rho_0}-1\geq n-2.\qedhere
		 \end{equation*}
\end{proof}

Here is the technical lemma that was invoked in the proof of (\ref{L:equatorial}).
\begin{lemma}\label{L:calculus}
	Let $\la_2+\la_4+\la_6=\pi$, $0\leq \la_i\leq \frac{\pi}{2}$ and $0< \rho_0\leq \frac{\pi}{2}$. Then   
	\begin{equation*}
		\arcsin(\cos\rho_0\sin\la_2)+\arcsin(\cos\rho_0\sin\la_4)+\arcsin(\cos\rho_0\sin\la_6)\geq \pi-2\rho_0
	\end{equation*}
\end{lemma}
\begin{proof}
	Let $f\colon [0,\pi]\to \R$ be the function given by $f(t)=\arcsin(\cos\rho_0\sin t)$. Then
	\begin{equation*}
		f''(t)=-\frac{\sin^2\rho_0\cos\rho_0\sin t}{\big(1-\cos^2\rho_0\sin^2t\big)^{\frac{3}{2}}},
	\end{equation*}
	so that $f''(t)\leq 0$ for all $t\in (0,\pi)$ and $f$ is a concave function. Consequently,
	\begin{equation}\label{E:concave}
f(s_1a+s_2b+s_3c)\geq s_1f(a)+s_2f(b)+s_3f(c)\text{\ \ for any $a,b,c\in [0,\pi]$,\ $s_i\in [0,1]$,\ $s_1+s_2+s_3=1$.}
	\end{equation}
	Define $g\colon T\to \R$ by $g(x,y,z)=f(x)+f(y)+f(z)$, where
	\begin{equation*}
		T=\set{(x,y,z)\in\R^3}{x+y+z=\pi,\ \?x,\? y,\? z\in \big[0,\tfrac{\pi}{2}\big]}.
	\end{equation*}
In other words, $T$ is the triangle with vertices $A=(0,\frac{\pi}{2},\frac{\pi}{2})$, $B=(\frac{\pi}{2},0,\frac{\pi}{2})$ and $C=(\frac{\pi}{2},\frac{\pi}{2},0)$. It follows from \eqref{E:concave} (applied three times) that
\begin{equation}\label{E:concaveg}
g(s_1u+s_2v+s_3w)\geq s_1g(u)+s_2g(v)+s_3g(w) \text{\ \ for any $u,v,w\in T,~s_i\in [0,1]$, $s_1+s_2+s_3=1$.}
\end{equation}
Moreover, a direct verification shows that 
\begin{equation*}
	g(A)=g(B)=g(C)=2\arcsin(\cos\rho_0)=\pi-2\rho_0.
\end{equation*}
If $p\in T$ then we can write
\begin{equation*}
	p=s_1A+s_2B+s_3C\text{ for some $s_1,\,s_2,\,s_3\in [0,1]$ with $s_1+s_2+s_3=1$.}
\end{equation*} Therefore, \eqref{E:concaveg} guarantees that
\[
g(p)\geq s_1g(A)+s_2g(B)+s_3g(C)=\pi-2\rho_0.\qedhere
\]
\end{proof}

\begin{proof}[Proof of (\ref{P:borderline})] If $\ga_s$ is condensed for all $s\in [0,1]$, then $s\mapsto \nu(\ga_s)$ is defined and constant, since it can only take on integral values. Thus, if the assertion is false, there must exist $s\in [0,1]$, say $s=1$, such that $\ga_{s}$ is not condensed. By (\ref{P:positivecondensed}), $\ga_0$ is homotopic to a circle traversed $\nu(\ga_0)$ times. Moreover, the set of non-condensed curves is open. Together with (\ref{L:C^2}), this shows that there exist $C^2$ curves $\ga_{-1},\ga_2$ such that:
\begin{enumerate}
	\item [(i)] There exist a path joining $\ga_{-1}$ to $\ga_0$ and a path joining $\ga_1$ to $\ga_2$ in $\sr L_{\ka_0}^{+\infty}(I)$;
	\item [(ii)] $\ga_{-1}$ is condensed and has rotation number $\nu(\ga_0)$; 
	\item [(iii)] $\ga_2$ is not condensed.
\end{enumerate}
Consider the map $f\colon \Ss^0\to \sr L_{\ka_0}^{+\infty}(I)$ given by $f(-1)=\ga_{-1}$, $f(1)=\ga_2$. The existence of the homotopy $\ga_s$ $(s\in [0,1]$) tells us that $f$ is nullhomotopic in $\sr L_{\ka_0}^{+\infty}(I)$. By (\ref{L:C^2}), $f$ must be nullhomotopic in $\sr C_{\ka_0}^{+\infty}(I)$. In other words, we may assume at the outset that each $\ga_s$ is of class $C^2$ ($s\in [0,1]$). 

With this assumption in force, let $s_0$ be the infimum of all $s\in [0,1]$ such that $\ga_s$ is not condensed, and let $\ga=\ga_{s_0}$. Then $\ga$ must be condensed by (\ref{L:closedhemisphere}), and it must be equatorial by our choice of $s_0$. In addition, $\nu(\ga_s)$ must be constant ($s\in [0,s_0]$), since it can only take on integral values. This contradicts (\ref{L:equatorial}).
\end{proof}


\section{Statement and Proof of the Main Theorems}\label{S:main}
We will now collect some of the results from the previous sections in order to prove the theorems stated in \S \ref{S:connected}. We repeat their statements here for convenience.

\setcounter{cthm}{2}
\begin{mthm}\label{T:allisround1}
		Let $-\infty\leq \ka_1<\ka_2\leq +\infty$. Every curve in $\sr L_{\ka_1}^{\ka_2}(I)$ (resp.~$\sr L_{\ka_1}^{\ka_2}$) lies in the same component as a circle traversed $k$ times, for some $k\in \N$ (depending on the curve).
\end{mthm}
\begin{proof}
By the homeomorphism $\sr L_{\ka_1}^{\ka_2}\home \SO_3\times \sr L_{\ka_1}^{\ka_2}(I)$ of (\ref{P:arbitrary}), it does not matter whether we prove the theorem for $\sr L_{\ka_1}^{\ka_2}$ or for $\sr L_{\ka_1}^{\ka_2}(I)$. Further, by (\ref{C:belowandabove}), it suffices to consider spaces of type $\sr L_{\ka_0}^{+\infty}$, for $\ka_0\in \R$. If $\ga\in \sr L_{\ka_0}^{+\infty}$ is diffuse, then it is homotopic to a circle by (\ref{P:diffuseloose}). If it is condensed, then the same conclusion holds by (\ref{C:contrensed}).
	
	Assume then that $\ga$ is neither homotopic to a condensed nor to a diffuse curve. Since $\ga$ itself is non-condensed by hypothesis, (\ref{P:totting}) guarantees that we may find $\eps>0$ and a chain of grafts $(\ga_s)$  with $\ga_0=\ga$ and $\ga_s\in \sr L_{\ka_0}^{+\infty}$ for all $s\in [0,\eps)$. Let $(\ga_s)$, $s\in J$, be a maximal chain of grafts starting at $\ga=\ga_0$, where $J$ is an interval of type $[0,\sig)$ or $[0,\sig]$. That such a chain exists follows by a straightforward  argument involving Zorn's lemma, since the grafting relation is a partial order, as proved in (\ref{L:eqrel}).\footnote{By reasoning more carefully it would be possible to avoid using Zorn's lemma.} By hypothesis, no curve $\ga_s$ is diffuse, hence $\nu(\ga_s)$ is well-defined and independent of $s$, and (\ref{P:totbound}) yields that $\sig<+\infty$. If the interval is of the first type, then we obtain a contradiction from (\ref{L:chainofgrafts}), and if the interval is closed, then we can apply (\ref{P:totting}) to $\ga_\sig$ to extend the chain, again contradicting the choice of $J$. We conclude that $\ga$ must be homotopic either to a condensed or to a diffuse curve. In any case, $\ga$ is homotopic in $\sr L_{\ka_0}^{+\infty}$ to a circle traversed a number of times, as claimed.
\end{proof}

\begin{mthm}\label{T:disconnects1}
	Let $-\infty\leq \ka_1<\ka_2\leq +\infty$ and let $\sig_k\in \sr L_{\ka_1}^{\ka_2}(I)$ (resp.~$\sr L_{\ka_1}^{\ka_2}$) denote any circle traversed $k\geq 1$ times. Then $\sig_k$,~$\sig_{k+2}$ lie in the same component of $\sr L_{\ka_1}^{\ka_2}(I)$ (resp.~$\sr L_{\ka_1}^{\ka_2}$) if and only if 
	\begin{equation*}
\qquad		k\geq \left\lfloor{\frac{\pi}{\rho_1-\rho_2}}\right\rfloor\quad (\rho_i=\arccot\ka_i,~i=1,2).
	\end{equation*}
\end{mthm}
\begin{proof}
This follows from the combination of (\ref{L:homocircles}), (\ref{P:borderline0}) and (\ref{P:borderline}), if we use the homeomorphisms in (\ref{P:arbitrary}) and (\ref{C:belowandabove}).
\end{proof}

\begin{prop}\label{P:condensed}
	Let $\ka_0=\cot \rho_0\geq 0$, 
	\begin{equation*}
		n=\fl{\frac{\pi}{\rho_0}}+1.
	\end{equation*}
	Then the set $\sr O_\nu$ of all condensed curves $\ga\in \sr L_{\ka_0}^{+\infty}(I)$ having rotation number $\nu$, with $1\leq \nu\leq n-2$, is a contractible connected component of $\sr L_{\ka_0}^{+\infty}(I)$. Further, $\sr O_\nu\home \E$.
\end{prop}
\begin{proof}
	Proposition (\ref{P:positivecondensed}) guarantees that $\sr O_\nu$ is weakly contractible and, in particular, connected. Proposition (\ref{P:borderline}) then implies that $\sr O_\nu$ must be a connected component of $\sr L_{\ka_0}^{+\infty}(I)$. Using (\ref{L:Hilbert}\?(a)) we deduce that $\sr O_\nu$ is an open subset of this space. Hence $\sr O_\nu$ is also a Hilbert manifold, and it must be homeomorphic to $\E$ by (\ref{L:Hilbert}\?(b)).
\end{proof}
\begin{urem}
	Note that if $\ka_0<0$ (that is, if $\rho_0>\frac{\pi}{2}$), then it is a consequence of (\ref{T:allisround1}) and (\ref{T:disconnects1}) that $\sr L_{\ka_0}^{+\infty}(I)$ has only $n=2$ components, and the conclusion of (\ref{P:condensed}) does not make sense in this case (no curve $\ga$ satisfies $\nu(\ga)\leq 0$). Moreover, these two components are far from being contractible: Even for $\ka_0=-\infty$, the (co)homology groups of $\sr I=\sr L_{-\infty}^{+\infty}(I)\iso \Om\Ss^3\du \Om\Ss^3$ are non-trivial in infinitely many dimensions. 
\end{urem}

\setcounter{cthm}{1}
\begin{mthm}\label{T:components1}
	Let $-\infty\leq \ka_1<\ka_2\leq +\infty$, $\rho_i=\arccot \ka_i$ \tup{(}$i=1,2$\tup{)} and $\fl{x}$ denote the greatest integer smaller than or equal to $x$. Then $\sr L_{\ka_1}^{\ka_2}$ has exactly $n$ connected components $\sr L_1,\dots,\sr L_n$, where
	\begin{equation*}
\qquad		n=\fl{\frac{\pi}{\rho_1-\rho_2}}+1
	\end{equation*} 
and  $\sr L_j$ contains circles traversed $j$ times $(1\leq j\leq n)$. The component $\sr L_{n-1}$ also contains circles traversed $(n-1)+2k$ times, and $\sr L_n$ contains circles traversed $n+2k$ times, for $k\in \N$. Moreover, each of $\sr L_1,\dots,\sr L_{n-2}$ $(n\geq 3)$ is homeomorphic to $\E\times \SO_3$, where $\E$ is the separable Hilbert space.
\end{mthm}
\begin{proof}
	All of the assertions of the theorem but the last one follow from (\ref{T:allisround1}), (\ref{T:disconnects1}) and the homeomorphism $\sr L_{\ka_1}^{\ka_2}\home \SO_3\times \sr L_{\ka_1}^{\ka_2}(I)$ of (\ref{P:arbitrary}). 
	
	Assume that $n\geq 3$ and let $\sig_k\in \sr L_{\ka_1}^{\ka_2}(I)$ be a circle traversed $k\leq n-2$ times. In the notation of (\ref{P:condensed}), the connected component $\sr L_k(I)$ of $\sr L_{\ka_1}^{\ka_2}(I)$ containing $\sig_k$ is mapped to the component $\sr O_k$ under the homeomorphism $\sr L_{\ka_1}^{\ka_2}(I)\home \sr L_{\ka_0}^{+\infty}(I)$ of (\ref{C:belowandabove}), because $\sig_k$ is mapped to another circle traversed $k$ times (cf.~(\ref{R:circletocircle})). Therefore, $\sr L_k(I)$ is homeomorphic to $\E$ by (\ref{P:condensed}). The last assertion of the theorem is deduced from this and the homeomorphism $\sr L_{\ka_1}^{\ka_2}\home \SO_3\times \sr L_{\ka_1}^{\ka_2}(I)$.
\end{proof}

Theorem (\ref{T:components1}) characterizes the connected components of $\sr L_{\ka_1}^{\ka_2}$ in terms of the circles that they contain. This characterization is not very useful for actually deciding whether two curves in this space lie in the same component. However, a  more direct characterization in terms of the properties of a curve is also available; it is stated in (\ref{C:direct}).

\setcounter{cthm}{5}
\begin{prop}\label{T:direct}
	Let $\ka_0\in \R$ and let $\sr L_1,\dots,\sr L_n$ be the connected components of $\sr L_{\ka_0}^{+\infty}$, as described in (\ref{T:components1}). Then $\ga\in \sr L_{\ka_0}^{+\infty}$ lies in:
	\begin{enumerate}
	\item [(i)] $\sr L_j$ $(1\leq j\leq n-2)$ if and only if it is condensed and has rotation number $j$.
	\item [(ii)] $\sr L_{n-1}$ if and only if $\te{\Phi}_\ga(1)=(-1)^{n-1}\te{\Phi}_\ga(0)$ and either it is non-condensed or condensed with rotation number $\nu(\ga)\geq n-1$.
	\item [(iii)] $\sr L_{n}$ if and only if $\te{\Phi}_\ga(1)=(-1)^{n}\te{\Phi}_\ga(0)$ and either it is non-condensed or condensed with rotation number $\nu(\ga)\geq n-1$.
\end{enumerate}
\end{prop}
\begin{proof}
	This follows from (\ref{T:components1}) and (\ref{P:condensed}).	
\end{proof}

Recall that $\te{\Phi}\colon [0,1]\to \Ss^3$ is the lift of the frame $\Phi_\ga\colon [0,1]\to \SO_3$ of $\ga$ to $\Ss^3$ (cf.~(\ref{D:liftedframe})). When $-\infty\leq \ka_0<0$ (resp.~$\rho_1-\rho_2>\frac{\pi}{2}$) we have $n=2$, and this characterization of the two components $\sr L_1$, $\sr L_2$ of $\sr L_{\ka_0}^{+\infty}$ (resp.~$\sr L_{\ka_1}^{\ka_2}$) may be simplified to: $\ga$ lies in  $\sr L_i$ if and only if $\te{\Phi}_\ga(1)=(-1)^{i}\te{\Phi}_\ga(0)$.

\begin{lemma}\label{L:direct}
	Let $-\infty\leq \ka_1<\ka_2\leq +\infty$, $\rho_i=\arccot \ka_i$ and $\ga_i\in \sr L_{\ka_1}^{\ka_2}$ $(i=1,2)$. Then $\ga_1$ lies in the same component of $\sr L_{\ka_1}^{\ka_2}$ as $\ga_2$ if and only if the corresponding translations $\bar \ga_i$ of $\ga_i$ by $\rho_2$,
\begin{equation*}
	\bar\ga_i(t)=\cos \rho_2\?\ga_i(t)+\sin\rho_2\?\no_i(t)\quad (t\in [0,1],~i=1,2),
\end{equation*}
lie in the same component of $\sr L_{\ka_0}^{+\infty}$, where $\ka_0=\cot(\rho_1-\rho_2)$ and $\no_i$ is the unit normal to $\ga_i$.\qed
\end{lemma}
\begin{proof}
	The proof is immediate, since translation by $\rho_2$ is a homeomorphism from $\sr L_{\ka_1}^{\ka_2}$ onto $\sr L_{\ka_0}^{+\infty}$, as was seen in (\ref{T:size}).
\end{proof}


\begin{mthm}\label{C:direct}
	Let $-\infty\leq \ka_1<\ka_2\leq +\infty$, $\ga\in \sr L_{\ka_1}^{\ka_2}$ and $\bar{\ga}\in \sr L_{\ka_0}^{+\infty}$ be the curve given by:
	\begin{equation*}
		\bar \ga(t)=\cos \rho_2\?\ga(t)+\sin\rho_2\?\no(t)\quad (t\in [0,1]),
	\end{equation*}
	where $\ka_0=\cot(\rho_1-\rho_2)$, $\rho_i=\arccot \ka_i$ \tup($i=1,2$\tup). Then $\ga$ lies in the component $\sr L_j$ of $\sr L_{\ka_1}^{\ka_2}$ described in (\ref{T:components1}) if and only if $\bar \ga$ satisfies the corresponding condition among (i)--(iii) stated in (\ref{T:direct}).\qed
\end{mthm}

The statement and proof of thms.~A and E may be found on p.~\pageref{T:size} and p.~\pageref{C:contrensed}, respectively. Theorems A--F are the main results of the paper. Replacing $\sr L$ by $\sr C$ (cf.~(\ref{D:Cittle})) in their statements, we obtain versions of these results for spaces of $C^r$ curves, with the $C^r$ topology (for any $r\geq 2$). These follow from the corresponding theorems for the spaces of type $\sr L$ and (\ref{L:C^2}). 

Finally, we remark that the results established here may be applied to the study of spaces of curves on $\RP^2$, with the usual Riemannian metric. Even though the latter is not orientable, we can still speak of the unsigned geodesic curvature of a curve. This allows us to define the spaces $\sr L_{-\ka_0}^{+\ka_0}$ and $\sr C_{-\ka_0}^{+\ka_0}$ of curves on $\RP^2$ $(\ka_0\in \R$), and to obtain information on their connected components using the theorems above and the fact that the natural projection $\Ss^2\to \RP^2$ is a 2-sheeted covering map which is also a local isometry (and hence preserves the absolute value of the geodesic curvature).


\begin{appendix}
\section{Basic Results on Convexity}

In this section we collect some results on convexity, none of which is new, that are used throughout the work.  Let $C\subs \R^{n+1}$. We say that $C$ is \tdef{convex} if it contains the line segment $[p,q]$ joining $p$ to $q$ whenever $p,q\in C$. The \tdef{convex hull} $\co{X}$ of a subset $X\subs \R^{n+1}$ is the intersection of all convex subsets of $\R^{n+1}$ which contain $X$. It may be characterized as the set of all points $q$ of the form
\begin{equation}\label{E:lc}
	q=\sum_{k=1}^m s_k p_k,\quad \text{where}\quad \sum_{k=1}^ms_k=1,\  s_k>0 \ \text{and}\  p_k\in X\text{ for each $k$}.
\end{equation}

\begin{lemma}\label{L:convexcompact}
	If $X\subs \R^{n}$ is compact, then $\co{X}$ is compact. In particular, if $X\subs \Ss^n$ is closed, then $\co{X}$ is compact.\qed
\end{lemma}

\begin{lemma}\label{L:basicconvex}
	Let $X\subs \Ss^n$ and consider the conditions:
	\begin{enumerate}
		\item [(i)]	$0$ does not belong to the closure of $\hat{X}$.
		\item [(ii)] There exists an open hemisphere containing $X$.
		\item [(iii)] $0$ does not belong to $\hat{X}$.
		\item [(iv)] $X$ does not contain any pair of antipodal points.
	\end{enumerate}
	Then \tup{(i)\,$\Rar$\,(ii)\,$\Rar$\,(iii)\,$\Rar$\,(iv)}, but none of the implications is reversible. If $X$ is closed then (ii) and (iii) are equivalent.
\end{lemma}

\begin{proof}\  
\begin{enumerate}
	\item [(i) $\Rar$ (ii)] This is a special case of the Hahn-Banach theorem, since $\se{0}$ is a compact convex set and the closure of $\hat{X}$ is a closed convex set.
	\item [(ii) $\not \Rar$ (i)] For $X\subs \Ss^n$ the open upper hemisphere, we have
	\begin{equation*}
		\hat{X}=\set{(x_1,\dots,x_{n+1})\in \D^{n+1}}{x_{n+1}>0}.
	\end{equation*}
	Hence the closure of $\hat{X}$ contains the origin, even though $X$ is (contained in) an open hemisphere.
	\item [(ii) $\Rar$ (iii)] Let $H=\set{p\in \Ss^n}{\gen{p,h}>0}$ be an open hemisphere containing $X$ and $U=\set{p\in \R^{n+1}}{\gen{p,h}>0}$. Then $U$ is convex, $X\subs U$ and $0\nin U$. Thus, $0\nin \hat{X}$.
	\item [(iii) $\not \Rar$ (ii)] Let $X$ be the image of $[0,\pi)$ under $t\mapsto \exp(it)$.
	\item [(iii) $\Rar$ (iv)] If $p$ and $-p$ both belong to $X$, then $0\in [-p,p]\subs  \hat{X}$.
	\item [(iv) $\not \Rar$ (iii)] Let $X=\se{1,\ze,\ze^2}\subs \Ss^1$, where $\ze=\exp(\frac{2}{3}\pi i)$ is a primitive third root of unity. Then $X$ does not contain antipodal points, but $0=\frac{1}{3}\big(1+\ze+\ze^2\big)$. 
\end{enumerate}
The last assertion is the combination of (i)\,$\Rar$\,(ii) and (ii)\,$\Rar$\,(iii), together with (\ref{L:convexcompact}). 
\end{proof}

We refer the reader to the corresponding appendix to \cite{tese} for proofs of the results below and to \cite{Fen1} for further results of this type.

\begin{lemma}\label{L:closedhemisphere}
	Let $X\subs \Ss^n$. Then $0$ belongs to the interior of $\hat{X}$ if and only if $X$ is not contained in any closed hemisphere of $\Ss^n$. \qed
\end{lemma}
%

\begin{lemma}\label{L:nonempty}
	A convex set $C\subs \R^n$ has empty interior if and only if it is contained in a hyperplane.\qed
\end{lemma}


\begin{lemma}\label{L:Steinitz}
	Let $X\subs \R^n$ be any set. If $p\in \hat{X}$, then there exists a $k$-dimensional simplex which has vertices in $X$ and contains $p$, for some $k\leq n$.\qed
\end{lemma}
Another way to formulate the previous result is the following: If $X\subs \R^n$ and $p\in \hat{X}$, then it is possible to write $p$ as a convex combination of $k+1$ points in $X$ which are in general position, where $k$ is at most equal to $n$.

\end{appendix}

\providecommand{\bysame}{\leavevmode\hbox to3em{\hrulefill}\thinspace}
\providecommand{\MR}{\relax\ifhmode\unskip\space\fi MR }
\providecommand{\MRhref}[2]{%
  \href{http://www.ams.org/mathscinet-getitem?mr=#1}{#2}
}
\providecommand{\href}[2]{#2}


\begin{thebibliography}{10}

\bibitem{Ani}
Sergei Anisov, \emph{Convex curves in ${RP}^n$}, Proc. Steklov Inst. Math
  \textbf{221} (1998), no.~2, 3--39.

\bibitem{BottTu}
Raoul Bott and Loring~W. Tu, \emph{Differential forms in algebraic topology},
  Springer Verlag, 1982.

\bibitem{bst}
Dan Burghelea, Nicolau~C. Saldanha, and Carlos Tomei, \emph{Results on
  infinite-dimensional topology and applications to the structure of the
  critical sets of nonlinear \textrm{S}turm-\textrm{L}iouville operators},
  Journal of Differential Equations \textbf{188} (2003), 569--590.

\bibitem{CodLev}
Earl Coddington and Norman Levinson, \emph{Theory of ordinary differential
  equations}, Krieger Publishing Company, 1984.

\bibitem{CoxeterNEG}
H.S.M. Coxeter, \emph{Non-\textrm{E}uclidean geometry}, Mathematical
  Association of America, 1998.

\bibitem{EliashbergMishachev}
Yakov Eliashberg and Nikolai Mishachev, \emph{Introduction to the h-principle},
  American Mathematical Society, 2002.

\bibitem{Fen1}
Werner Fenchel, \emph{\"{U}ber \textrm{K}r\"ummung und \textrm{W}indung
  geschlossener \textrm{R}aumkurven}, Math. Ann. \textbf{101} (1929), 238--252.

\bibitem{GromovPR}
Mikhael Gromov, \emph{Partial differential relations}, Springer Verlag, 1986.

\bibitem{Henderson1}
David Henderson, \emph{Infinite-dimensional manifolds are open subsets of
  {H}ilbert space}, Bull. Amer. Math. Soc. \textbf{75} (1969), no.~4, 759--762.

\bibitem{KheSha}
Boris~A. Khesin and Boris~Z. Shapiro, \emph{Nondegenerate curves on
  $\mathit{S}^2$ and orbit classification of the \textrm{Z}amolodchikov
  algebra}, Commun. Math. Phys. \textbf{145} (1992), 357--362.

\bibitem{KheSha2}
\bysame, \emph{Homotopy classification of nondegenerate quasiperiodic curves on
  the 2-sphere}, Publ. Inst. Math. (Beograd) \textbf{66 (80)} (1999), 127--156.

\bibitem{Levy}
S.~Levy, D.~Maxwell, and T.~Munzner, \emph{Outside in}, Narrated video (21 min)
  from the Geometry Center, 1994.

\bibitem{Little}
John~A. Little, \emph{Nondegenerate homotopies of curves on the unit 2-sphere},
  J.~of Differential Geometry \textbf{4} (1970), 339--348.

\bibitem{MostovoySadykov}
Jacob Mostovoy and Rustam Sadykov, \emph{The space of non-degenerate closed
  curves in a {R}iemannian manifold}, arXiv:1209.4109, 2012.

\bibitem{PalHom}
Richard Palais, \emph{Homotopy theory of infinite dimensional manifolds},
  Topology \textbf{5} (1966), 1--16.

\bibitem{Sal1}
Nicolau~C. Saldanha, \emph{The homotopy and cohomology of spaces of locally
  convex curves in the sphere -- \textrm{I}}, arXiv:0905.2111v1, 2009.

\bibitem{Sal2}
\bysame, \emph{The homotopy and cohomology of spaces of locally convex curves
  in the sphere -- \textrm{II}}, arXiv:0905.2116v1, 2009.

\bibitem{Sal3}
\bysame, \emph{The homotopy type of spaces of locally convex curves in the
  sphere}, arXiv:1207.4221, 2012.

\bibitem{SalSha}
Nicolau~C. Saldanha and Boris~Z. Shapiro, \emph{Spaces of locally convex curves
  in $\mathit{S}^n$ and combinatorics of the group $\mathit{B}_{n+1}^{+}$},
  Journal of Singularities \textbf{4} (2012), 1--22.

\bibitem{Sal4}
Nicolau~C. Saldanha and Carlos Tomei, \emph{The topology of critical sets of
  some ordinary differential operators}, Progr. Nonlinear Differential
  Equations Appl. \textbf{66} (2005), 491--504.

\bibitem{ShaSha}
Boris~Z. Shapiro and Michael~Z. Shapiro, \emph{On the number of connected
  components in the space of closed nondegenerate curves on $\mathit{S}^n$},
  Bulletin of the AMS \textbf{25} (1991), no.~1, 75--79.

\bibitem{ShaTop}
Michael~Z. Shapiro, \emph{Topology of the space of nondegenerate curves}, Math.
  USSR \textbf{57} (1993), 106--126.

\bibitem{Sma56}
Stephen Smale, \emph{Regular curves on {R}iemannian manifolds}, Trans. of the
  AMS \textbf{87} (1956), no.~2, 492--512.

\bibitem{WhiGra}
Hassler Whitney, \emph{On regular closed curves in the plane}, Compositio
  Mathematica \textbf{4} (1937), no.~1, 276--284.

\bibitem{tese}
Pedro Z\"uhlke, \emph{Homotopies of curves on the 2-sphere with geodesic
  curvature in a prescribed interval}, Ph.D. thesis, PUC-Rio, 2012.

\end{thebibliography}

\vspace{12pt}
\noindent{\small \tsc{Departamento de Matem\'atica, PUC-Rio  \\
R. Marqu\^es de S. Vicente 225, Rio de Janeiro, RJ 22453-900, Brazil}

\vspace{2pt}
\noindent{\small \ttt{nicolau@mat.puc-rio.br\\ pedro@mat.puc-rio.br}}

\end{document}